\documentclass[reqno,11pt]{amsart}
\usepackage{amsmath,amssymb,latexsym,cancel,rotating,hyperref}

\usepackage{eucal}

\usepackage{amscd}
\usepackage[dvips]{color}
\usepackage{multicol}
\usepackage[all]{xy}           
\usepackage{graphicx}
\usepackage{color}
\usepackage{colordvi}
\usepackage{xspace}
\usepackage{tikz}

\usepackage{enumitem}

\numberwithin{equation}{section}

\newtheorem{theorem}{Theorem}[section]

\newtheorem{corollary}[theorem]{Corollary}
\newtheorem{lemma}[theorem]{Lemma}

\newtheorem{example}[theorem]{Example}
\newtheorem{definition}[theorem]{Definition}

\input xy
\xyoption{poly}
\xyoption{2cell}
\xyoption{all}

\def\e{{\bf{e}}}

\def\b{{\bf{b}}}

\renewcommand{\eqref}[1]{{\rm (\ref{#1})}}

\begin{document}

\title[Fundamental relations in  quantum cluster algebras]
{Fundamental relations in  quantum cluster algebras}

\author{Junyuan Huang, Xueqing Chen, Ming Ding and Fan Xu}
\address{School of Mathematics and Information Science\\
Guangzhou University, Guangzhou 510006, P.R.China}
\email{jy4545@e.gzhu.edu.cn (J.Huang)}
\address{Department of Mathematics,
 University of Wisconsin-Whitewater\\
800 W. Main Street, Whitewater, WI.53190. USA}
\email{chenx@uww.edu (X.Chen)}
\address{School of Mathematics and Information Science\\
Guangzhou University, Guangzhou 510006, P.R.China}
\email{dingming@gzhu.edu.cn (M.Ding)}
\address{Department of Mathematical Sciences\\
Tsinghua University\\
Beijing 100084, P.~R.~China} \email{fanxu@mail.tsinghua.edu.cn (F.Xu)}






\dedicatory{Dedicated to Professor Bangming Deng on the Occasion of His 60th Birthday}

\maketitle

\begin{abstract}
Let $\mathcal{A}_{q}$ be an arbitrary   quantum cluster algebra with principal coefficients.  We give the fundamental relations between the quantum cluster variables arising from one-step mutations from the initial cluster in  $\mathcal{A}_{q}$.
Immediately and directly, we obtain an algebra homomorphism from the corresponding (untwisted) quantum group to $\mathcal{A}_{q}$.\end{abstract}


\section{Background}

Cluster algebra was invented by Fomin and Zelevinsky \cite{ca1,ca2}  in their purpose to study the total positivity \cite{G3} and Lusztig's dual canonical bases in coordinate rings and their $q$-deformations \cite{G1,G2}.
As a subalgebra of the field of rational functions in $n$ variables, cluster algebra is  generated by a family of generators called cluster variables obtained recursively through mutations. As the quantization of cluster algebra,
quantum cluster algebra was later introduced by Berenstein and Zelevinsky \cite{AA}.

The Hall algebra  was introduced by Ringel \cite{R1} to encode a finitary abelian category.  Ringel proved that if $A$ is a representation-finite hereditary algebra over a finite field,
the Ringel--Hall algebra of $A$ is isomorphic to the positive part of the corresponding
quantum group \cite{R1,R2}. A connection between the representation theory
of algebras and Lie theory can then be established via Ringel's approach,  as well as an algebraic framework for studying the Lie theory from Hall algebras associated to various abelian and triangulated categories.
Recall that Ringel--Hall
algebras satisfy the so-called fundamental relations, which are similar to the defining relations-
 the quantum Serre relations for quantized enveloping algebras. It was shown in \cite{R3}
 that for Ringel-Hall algebras of hereditary algebras,  fundamental
relations differ from the quantum Serre relations only by a twist of the multiplication. For a more general setting, Chen and Deng \cite{CD}
showed that the Ringel--Hall algebra of a finitary category over a finite field satisfies fundamental relations.

The cluster multiplication theorem \cite{ck,cdz,XX,X} revealed the similarity between the multiplication
in a cluster algebra and that in a dual Hall algebra, and then it is natural to construct a framework to explicitly relate the dual Hall algebra with the quantum
cluster algebra.  Let $Q$ be a finite acyclic quiver, and let $\mathcal{AH}_{q}(Q)$ be the subalgebra of a certain skew-field
 of fractions generated by quantum cluster characters. A surjective algebra homomorphism from a certain dual Hall
algebra to $\mathcal{AH}_{q}(Q)$ was described in \cite{DXZ,FPZ}, where the quantum cluster algebra can be realized as a sub-quotient algebra
of this Hall algebra. By using this algebra homomorphism, one may obtain fundamental relations indirectly in the corresponding acyclic quantum cluster algebras.

In the present paper, we work on an arbitrary quantum cluster algebra with principal coefficients. We mainly establish the fundamental relations, upon which we further obtain some higher order fundamental relations.
Immediately and directly, we obtain an algebra homomorphism from the corresponding (untwisted) quantum group to quantum cluster algebra.

It is worth remarking that our approach in this paper is almost exclusively  elementary algebraic computation.

\section{Preliminaries}
At first we collect some necessary notations and concepts of quantum cluster algebras \cite{AA} and give some useful $q$-identities,
respectively.

\subsection{Quantum cluster algebras}
Given two positive integers $s$ and $t$ with $s<t$, for the sake of convenience we use the notion  $[s,t]$ to represent the set $\{s,s+1,\ldots,t-1,t\}$.
 A skew-symmetrizable matrix $B$ is  a square integer matrix such that there exists
some integer diagonal matrix $D$ with positive diagonal entries to make $DB$ being skew-symmetric and then $D$ is called the skew-symmetrizer of $B$. For two positive integers
$m$ and $n$ with $m\ge n$,  let $\widetilde{B}=(b_{ij})$ be an~$m\times n$ integer matrix with its principal part (the upper $n\times n$ submatrix) being a skew-symmetrizable matrix $B$. The coefficient matrix
in $\widetilde{B}$ is the submatrix composed of the last $m-n$ rows.

An $m\times m$  skew-symmetric integer matrix $\Lambda $ can be chosen to make $\widetilde{B}^{T}\Lambda=\begin{bmatrix} D & {\bf{0}} \end{bmatrix}$ for some integer diagonal
 matrix  $D$ with positive diagonal entries. The pair $(\widetilde{B},\Lambda)$ is then called a compatible pair. For later use for a given compatible pair $(\widetilde{B},\Lambda)$,
 we denote $D=\operatorname{diag}(d_1, d_2, \dots, d_n)$ and
$\widetilde{B}=\begin{bmatrix}  {\bf{b}}_1 & {\bf{b}}_2 &\dots &{\bf{b}}_n \end{bmatrix}$ with
${\bf{b}}_j \in \mathbb{Z}^m $ for $ j \in[1, n]$.

Note that one can identify the skew-symmetric $m\times m$ matrix $\Lambda=(\lambda_{ij})$ with the bilinear form still denoted by
 $\Lambda:\mathbb{Z}^m\times\mathbb{Z}^m\rightarrow \mathbb{Z}$ which satisfies the compatibility
condition with  $\widetilde{B}$, i.e.,
\begin{equation}\label{(dj)}
	\Lambda({\bf{b}}_j, {\bf{e}}_i)=\delta_{ij}d_j \quad (i \in [1,m], j\in [1,n])
		\end{equation}
where $\e_i$ is the $i$-th unit vector in $\mathbb{Z}^m$ for any $ i\in [1,m]$. This identification is achieved through $\lambda_{ij}:=\Lambda(\e_i,\e_j).$

Let $q$ be a formal variable and denote the ring of integer Laurent polynomials
in the variable $q^{\frac{1}{2}}$ by $\mathbb{Z}[q^{ \pm {\frac{1}{2}}}]$. The based quantum torus $\mathcal{T}=\mathcal{T}(\Lambda)$ is defined as the $\mathbb{Z}[q^{ \pm {\frac{1}{2}}}]$-algebra with a
distinguished  $\mathbb{Z}[q^{ \pm {\frac{1}{2}}}]$-basis $\{ {X^{\e}}: {\e}\in\mathbb{Z}^m \}$ and the multiplication
\begin{equation}\label{(ef)}
	X^{\e} X^{\bf{f}}=q^{\frac{\Lambda({\bf{e}}, {\bf{f}})}{2}} X^{{\bf{e}} +{\bf{f}}} , \quad (\e, {\bf{f}} \in \mathbb{Z}^m).
	 \end{equation}

For any ${i\in [1,m]}$, denote $x_i:=X^{\e_i}$, then as a  $\mathbb{Z}[q^{ \pm {\frac{1}{2}}}]$-algebra,  $\mathcal{T}$
is generated by the elements $x_i$ and their inverses subject to the quasi-commutative relations
\begin{equation}\label{(ij)}
	{x_i}{x_j} = {q^{{\lambda _{ij}}}}{x_j}{x_i}
\end{equation}
for $i,j \in [1,m]$.
Given ${\bf{a}}=(a_1, a_2, \dots, a_m)\in \mathbb{Z}^m$,  we have $$X^{\bf{a}}=q^{\frac{1}{2} \sum_{l<k} a_k a_l \lambda_{kl}} x_1^{a_1}x_2^{a_2}\cdots x_m^{a_m}.$$

Define the $\mathbb{Z}$-linear bar-involution $-: \mathcal{T} \longrightarrow \mathcal{T}$ by
\[\overline {{q^{\frac{l}{2}}}{X^\mathbf{c}}}  = {q^{ - \frac{l}{2}}}{X^\mathbf{c}},\quad \text{for all } l \in \mathbb{Z} \text{ and }\mathbf{c} \in \mathbb{Z}^m.\]

Note that  $\overline {fg}  = \overline g \overline f $ for any $f,g \in \mathcal{T}$. 

\begin{definition}\label{def of generalized seed}
With the above notations,  the triple $\big(\widetilde{\mathbf{x}},\Lambda,\widetilde{B}\big)$ is called a quantum seed where
the set $\widetilde{\mathbf{x}}=\{x_{1},x_2,\ldots, x_{m}\}$ is the extended cluster, $\mathbf{x}=\{x_{1},x_2,\ldots,x_{n}\}$ is the cluster,
elements $x_i$ for $i\in [1,n]$ are called quantum cluster variables and elements $x_{i}$ for $i\in [m+1, n]$ are called frozen variables.
\end{definition}

Define the function
\begin{gather*}
[x]_+:=
 \begin{cases}
 x,  &\text{if}\quad x\geq 0;
 \\
0,  & \text{if}\quad x< 0.
 \end{cases}
\end{gather*}

\begin{definition}
The mutation of a quantum seed $\big(\widetilde{\mathbf{x}},\Lambda,\widetilde{B}\big)$ in the direction $k \in [1,n]$ is the quantum seed \smash{$\mu_k\big(\widetilde{\mathbf{x}},\Lambda,\widetilde{B}\big):=\big(\widetilde{\mathbf{x}}',
\Lambda',\widetilde{B}'\big)$}, where
\begin{enumerate}\itemsep=0pt
\item the set $\widetilde{\mathbf{x}}':=(\widetilde{\mathbf{x}}- \{x_k\})\cup\{x'_{k}\}$ with
\begin{equation}\label{(er)}
x'_k=X^{-\e_k+[\b_k]_+}+X^{-\e_k+[-\b_k]_+};
\end{equation}
\item
the matrix $\widetilde{B}':=\mu_k\big(\widetilde{B}\big)$ is defined by
\begin{gather}\label{(B)}
b'_{ij}=
 \begin{cases}
 -b_{ij}, &\text{ if }  i=k \text{ or } j=k;
 \\
b_{ij}+\displaystyle\frac{|b_{ik}|b_{kj}+b_{ik}|b_{kj}|}{2}, &\text{otherwise};
 \end{cases}
\end{gather}
\item
the skew-symmetric matrix $\Lambda':=\mu_k\big(\Lambda\big)$ is defined by
\begin{gather}\label{(A)}
	\lambda'_{ij}=
	\begin{cases}
		\lambda_{ij}, &\text{ if } i,j\ne k;
		\\
		- {\lambda _{ij}} + \sum\limits_{t = 1}^m {{{[{b_{ti}}]}_ + }{\lambda _{tj}}} , &\text{ if } i= k, j\ne k.
	\end{cases}
\end{gather}
\end{enumerate}
\end{definition}

Note that $\mu_k$ is an involution.  Two quantum seeds are called mutation-equivalent if one can be obtained from
another through a sequence of mutations. Denote the skew-field of fractions of $\mathcal{T}$  by $\mathcal{F}$ and
$$\mathop{\mathbb{ZP}}:=\mathbb{Z}[q^{ \pm {\frac{1}{2}}}][x^{\pm}_{n+1},\ldots,x^{\pm}_m].$$

\sloppy\begin{definition}\label{def of gca}
Given an initial quantum seed $\big(\widetilde{\mathbf{x}}, \Lambda, \widetilde{B}\big)$, the quantum cluster algebra $\smash{\mathcal{A}\big(\widetilde{\mathbf{x}}, \Lambda, \widetilde{B}\big)}$ is the $\mathop{\mathbb{ZP}}$-subalgebra of~$\mathcal{F}$ generated by all quantum cluster variables from all quantum seeds mutation-equivalent to $\smash{\big(\widetilde{\mathbf{x}},\Lambda,\widetilde{B}\big)}$.
\end{definition}
Note that  by letting $q\rightarrow 1$ one can recover the classical cluster algebra.

Denoted by   $\Gamma\big(\widetilde{\mathbf{x}}, \Lambda,\widetilde{B}\big)$ the directed graph associated to a quantum seed~$\big(\widetilde{\mathbf{x}}, \Lambda,\widetilde{B}\big)$ with vertices $[1,n]$ and the directed edge from $i$ to $j$ if $b_{ij}>0$.

\begin{definition}
	A  quantum cluster algebra $\smash{\mathcal{A}\big(\widetilde{\mathbf{x}}, \Lambda, \widetilde{B}\big)}$ is called principal coefficients quantum cluster algebra if the coefficient matrix in the initial seed $\big(\widetilde{\mathbf{x}},\Lambda,\widetilde{B}\big)$ is an identity matrix.
\end{definition}

\subsection{Some  $q$-identities} Let $\mathbb{Z}[q]$ be a polynomial ring in an indeterminate $q$. For each $n\ge 0$, define
\[{[n]_{q}} = \frac{{{q^n} - 1}}{{q - 1}}; \]
\begin{gather}\label{(A2)}
	[n]_{{q}}^!=
	\begin{cases}
		{[1]_{{q}}}{[2]_{{q}}} \cdots {[n]_{{q}}}, &\text{ if } n\geq 1;
		\\
		1 , &\text{ if } n= 0.
	\end{cases}
\end{gather}

\begin{gather}\label{(A1)}
		{\left[ {\begin{array}{*{20}{c}}
			n\\
			r
	\end{array}}
	\right]_{{q}}}=
	\begin{cases}
		\frac{{[n]_{{q}}^!}}{{[r]_{{q}}^![n - r]_{{q}}^!}}, &\text{ if } 0\le r\le n;
		\\
		0 , &\text{ if } r< 0 \text{ or } r> n.
	\end{cases}
\end{gather}

A direct calculation shows that
\begin{eqnarray}\label{(n+1)}
	{\left[ {\begin{array}{*{20}{c}}
			{n + 1}\\
			r
	\end{array}}
	\right]_{{q}}} = {\left[ {\begin{array}{*{20}{c}}
			n\\
			r
	\end{array}}
	\right]_{{q}}} + {{q}^{n + 1 - r}}{\left[ {\begin{array}{*{20}{c}}
			n\\
			{r - 1}
	\end{array}}
	\right]_{{q}}} = {{q}^r}{\left[ {\begin{array}{*{20}{c}}
			n\\
			r
	\end{array}}
	\right]_{{q}}} + {\left[ {\begin{array}{*{20}{c}}
			n\\
			{r - 1}
	\end{array}}
	\right]_{{q}}},\ \ \   r \ge 0, n\ge 0;
\end{eqnarray}
\begin{eqnarray}\label{q+-b}
	{\left[ n \right]_{{q}}} = {q^{n-1 }}{\left[ n \right]_{{q^{ - 1}}}},\ \ \   n\geq 0;
\end{eqnarray}	
\begin{eqnarray}\label{q+-b1}	
	{\left[ {\begin{array}{*{20}{c}}
				n\\
				r
		\end{array}} \right]_{{q}}} =  {q}^{r(n - r)}{\left[ {\begin{array}{*{20}{c}}
				n\\
				r
		\end{array}} \right]_{{q^{ - 1}}}},\ \ \  0 \le r \le n;
\end{eqnarray}
\begin{eqnarray}\label{(q^btq^b)}
{\left[ n \right]_{{q^{r}}}} = {\left[ n \right]_{{q}}}\frac{{1 + {q^{n}} + {q^{2n}} +  \cdots  + {q^{(r - 1)n}}}}{{1 + {q} + {q^{2}} +  \cdots  + {q^{r - 1}}}},\ \ \ r\geq 1 ,n\ge 1.
\end{eqnarray}
	
In the following, we list some well known identities.
\begin{lemma}\cite{I}
For any positive integer $ d $, we have

\begin{equation}\label{(q=0)}
\sum_{r=0}^{d} (-1)^r q^{\frac{r(r-1)}{2}} \left[ {\begin{array}{*{20}{c}}
			d\\
			r
	\end{array}}
	\right]_q = 0.\end{equation}
\end{lemma}

\begin{corollary}\cite{CD}\label{-cr}
	For any integers $ d \geq 1 $ and $ 0 \leq c \leq d-1 $, we have
	
	\begin{equation}\label{(-cr)}
		\sum\limits_{r = 0}^{d} {{{( - 1)}^r}{{{q}}^{\frac{{r(r - 1)}}{2} - cr}}{{\left[ {\begin{array}{*{20}{c}}
							d\\
							r
					\end{array}} \right]}_{{q}}}}  = 0.
	\end{equation}
	
\end{corollary}

\begin{lemma}\label{q-erxiangshi}
	For  any   integer $n\geq 1$, we have \[\prod\limits_{r = 1}^n {\left( {1 + {q^r}x} \right)}  = \sum\limits_{k = 0}^n {{{\left[ {\begin{array}{*{20}{c}}
						n\\
						k
				\end{array}} \right]}_q}{q^{\frac{{k(k + 1)}}{2}}}x^k} .\]

\end{lemma}
\begin{corollary}\label{1q-erxiangshi}
	For  any   integer $n\geq 1$ and $xy=yx$, we have \[\prod\limits_{r = 1}^n {\left( {y + {q^r}x} \right)}  = \sum\limits_{k = 0}^n {{{\left[ {\begin{array}{*{20}{c}}
						n\\
						k
				\end{array}} \right]}_q}{q^{\frac{{k(k + 1)}}{2}}}{y^{n - k}}{x^k}}. \]

\end{corollary}
\begin{lemma}\cite{KP}\label{n-dk}
	For any integers $k\geq 0$ and  $0 \le d\le n$, we have
	\begin{equation}\label{(n-dk)}	
{\left[ {\begin{array}{*{20}{c}}
				n\\
				k
		\end{array}} \right]_q}  = \sum\limits_{r = 0}^k {{q^{(d - r)(k - r)}}{{\left[ {\begin{array}{*{20}{c}}
						d\\
						r
				\end{array}} \right]}_q}{{\left[ {\begin{array}{*{20}{c}}
						{n - d}\\
						{k - r}
				\end{array}} \right]}_q}} .
	\end{equation}
\end{lemma}
\begin{lemma}\label{lemma3}
We have

\begin{enumerate}\itemsep=0pt
\item for any  integer  $1\le k\le n$,	
\begin{equation}\label{(lemma3)}
\sum\limits_{t = 0}^n {{q^{ - tk}}\sum\limits_{r = 0}^t {{{( - 1)}^r}{q^{\frac{{r(r - 1)}}{2}}}{{\left[ {\begin{array}{*{20}{c}}
						{n + 1}\\
						r
				\end{array}} \right]}_q}} }  = 0;
			\end{equation}
\item for any  integer  $0\le k\le v-1$, $v\le n$,
\begin{equation}\label{(lemma31)}
	\sum\limits_{t = 0}^n {{q^{ t(v-k)}}\sum\limits_{r = 0}^t {{{( - 1)}^r}{q^{\frac{{r(r - 1)}}{2}-nr}}{{\left[ {\begin{array}{*{20}{c}}
							{n + 1}\\
							r
					\end{array}} \right]}_q}} }  = 0.			
\end{equation}

\end{enumerate}
\end{lemma}
\begin{proof}
\begin{enumerate}\itemsep=0pt
\item We calculate
	\begin{eqnarray*}
&&\sum\limits_{t = 0}^n {{q^{ - tk}}\sum\limits_{r = 0}^t {{{( - 1)}^r}{q^{\frac{{r(r - 1)}}{2}}}{{\left[ {\begin{array}{*{20}{c}}
							{n + 1}\\
							r
					\end{array}} \right]}_q}} }  = \sum\limits_{t = 0}^n {{q^{ - tk}}\sum\limits_{r = 0}^t {{D_r}} } \\
&	=& {D_0} + {q^{ - k}}\left( {{D_0} + {D_1}} \right) + {q^{ - 2k}}\left( {{D_0} + {D_1} + {D_2}} \right) +  \cdots  + {q^{ - nk}}\sum\limits_{r = 0}^n {{D_r}} \\
&	=& {D_0}\left( {1 + {q^{ - k}} + {q^{ - 2k}} +  \cdots  + {q^{ - nk}}} \right) + {D_1}\left( {{q^{ - k}} + {q^{ - 2k}} +  \cdots  + {q^{ - nk}}} \right)\\
	&&+ {D_2}\left( {{q^{ - 2k}} +  \cdots  + {q^{ - nk}}} \right) +  \cdots  +  {D_n}{q^{ - nk}}\\
&	=& {D_0}{\left[ {n + 1} \right]_{{q^{ - k}}}} + {D_1}{q^{ - k}}{\left[ n \right]_{{q^{ - k}}}} + {D_2}{q^{ - 2k}}{\left[ {n - 1} \right]_{{q^{ - k}}}}	+  \cdots   + {D_n}{q^{ - nk}}{\left[ 1 \right]_{{q^{ - k}}}}\\
&	=& \sum\limits_{r = 0}^n { {{q^{ - rk}}{{\left[ {n + 1 - r} \right]}_{{q^{ - k}}}}{D_r}} } ,
\end{eqnarray*}
where ${D_r} = {( - 1)^r}{q^{\frac{{r(r - 1)}}{2}}}{\left[ {\begin{array}{*{20}{c}}
			{n + 1}\\
			r
	\end{array}} \right]_q}.$

According to  (\ref{q+-b}) and (\ref{(q^btq^b)}), we obtain
\begin{eqnarray*}
	&&	\sum\limits_{r = 0}^n { {{q^{ - rk}}{{\left[ {n + 1 - r} \right]}_{{q^{ - k}}}}{D_r}}} \\
	&	=& \sum\limits_{r = 0}^n {{{( - 1)}^r}{{\left[ {n + 1 - r} \right]}_{{q^{ - k}}}}{q^{\frac{{r(r - 1)}}{2} - rk}}{{\left[ {\begin{array}{*{20}{c}}
							{n + 1}\\
							r
					\end{array}} \right]}_q}}  \\
	&	\mathop  = \limits^{(\ref{q+-b})}& \sum\limits_{r = 0}^n {{{( - 1)}^r}{q^{k(r - n)}}{{\left[ {n + 1 - r} \right]}_{{q^k}}}{q^{\frac{{r(r - 1)}}{2} - rk}}{{\left[ {\begin{array}{*{20}{c}}
						{n + 1}\\
						r
				\end{array}} \right]}_q}}   \\
	&	\mathop  = \limits^{(\ref{(q^btq^b)})}&	\sum\limits_{r = 0}^n {\left( {{{( - 1)}^r}{q^{ - nk}}{{\left[ {n + 1 - r} \right]}_q}\frac{{\sum\limits_{s = 0}^{k - 1} {{q^{s(n + 1 - r)}}} }}{{{{\left[ k \right]}_q}}}{q^{\frac{{r(r - 1)}}{2}}}{{\left[ {\begin{array}{*{20}{c}}
							{n + 1}\\
							r
					\end{array}} \right]}_q}} \right)} \\
&	= &\sum\limits_{r = 0}^n {\left( {{{( - 1)}^r}{q^{ - nk}}{{\left[ {n + 1} \right]}_q}\frac{{\sum\limits_{s = 0}^{k - 1} {{q^{s(n + 1 - r)}}} }}{{{{\left[ k \right]}_q}}}{q^{\frac{{r(r - 1)}}{2}}}{{\left[ {\begin{array}{*{20}{c}}
							n\\
							r
					\end{array}} \right]}_q}} \right)} \\
					\end{eqnarray*}		
				\begin{eqnarray*}
&	=& \frac{{{q^{ - nk}}{{\left[ {n + 1} \right]}_q}}}{{{{\left[ k \right]}_q}}}\sum\limits_{r = 0}^n {{{( - 1)}^r}\left( {\sum\limits_{s = 0}^{k - 1} {{q^{s(n + 1 - r)}}} } \right){q^{\frac{{r(r - 1)}}{2}}}{{\left[ {\begin{array}{*{20}{c}}
						n\\
						r
				\end{array}} \right]}_q}}  .
\end{eqnarray*}	

For  any $s \in [0,k-1]$, according to Corollary \ref{-cr}, we have
\begin{eqnarray*}
\sum\limits_{r = 0}^n {{{( - 1)}^r}{q^{s(n + 1 - r)}}{q^{\frac{{r(r - 1)}}{2}}}{{\left[ {\begin{array}{*{20}{c}}
					n\\
					r
			\end{array}} \right]}_q}}  = {q^{s(n + 1)}}\sum\limits_{r = 0}^n {{{( - 1)}^r}{q^{\frac{{r(r - 1)}}{2} - sr}}{{\left[ {\begin{array}{*{20}{c}}
					n\\
					r
			\end{array}} \right]}_q}} \mathop  = \limits^{(\ref{(-cr)})} 0.
\end{eqnarray*}


\item Similarly, we have
\begin{eqnarray*}
&&	\sum\limits_{t = 0}^n {{q^{t(v - k)}}\sum\limits_{r = 0}^t {{{( - 1)}^r}{q^{\frac{{r(r - 1)}}{2} - nr}}{{\left[ {\begin{array}{*{20}{c}}
							{n + 1}\\
							r
					\end{array}} \right]}_q}} }  = \sum\limits_{t = 0}^n {{q^{t(v - k)}}\sum\limits_{r = 0}^t {{F_r}} } \\
&	=& {F_0} + {q^{v - k}}\left( {{F_0} + {F_1}} \right) + {q^{2(v - k)}}\left( {{F_0} + {F_1} + {F_2}} \right) +  \cdots  + {q^{n(v - k)}}\sum\limits_{r = 0}^n {{F_r}} \\
&	=& {F_0}\left( {1 + {q^{v - k}} + {q^{2(v - k)}} +  \cdots  + {q^{n(v - k)}}} \right) + {F_1}\left( {{q^{v - k}} + {q^{2(v - k)}} +  \cdots  + {q^{n(v - k)}}} \right){\rm{ }}\\
&&	+ {F_2}\left( {{q^{2(v - k)}} +  \cdots  + {q^{n(v - k)}}} \right) +  \cdots  +  {F_n}{q^{n(v - k)}}\\
&	=& {F_0}{\left[ {n + 1} \right]_{{q^{v - k}}}} + {F_1}{q^{v - k}}{\left[ n \right]_{{q^{v - k}}}} + {F_2}{q^{2(v - k)}}{\left[ {n - 1} \right]_{{q^{v - k}}}} +  \cdots  +  {F_n}{q^{n(v - k)}}{\left[ 1 \right]_{{q^{v - k}}}}\\
&	=&\sum\limits_{r = 0}^n { {{q^{r(v - k)}}{{\left[ {n + 1 - r} \right]}_{{q^{v - k}}}}{F_r}} } ,
\end{eqnarray*}
where ${F_r} = {( - 1)^r}{q^{\frac{{r(r - 1)}}{2}-nr}}{\left[ {\begin{array}{*{20}{c}}
			{n + 1}\\
			r
	\end{array}} \right]_q}.$	

According to    (\ref{(q^btq^b)}), we obtain
\begin{eqnarray*}
&&	\sum\limits_{r = 0}^n {{q^{r(v - k)}}{{\left[ {n + 1 - r} \right]}_{{q^{v - k}}}}{F_r}}      =   \sum\limits_{r = 0}^n {{{( - 1)}^r}{{\left[ {n + 1 - r} \right]}_{{q^{v - k}}}}{q^{\frac{{r(r - 1)}}{2} - r(n - v + k)}}{{\left[ {\begin{array}{*{20}{c}}
						{n + 1}\\
						r
				\end{array}} \right]}_q}} \\
&	\mathop  = \limits^{(\ref{(q^btq^b)})} & \sum\limits_{r = 0}^n {\left( {{{( - 1)}^r}{{\left[ {n + 1 - r} \right]}_q}\frac{{\sum\limits_{s = 0}^{v - k - 1} {{q^{s(n + 1 - r)}}} }}{{{{\left[ {v - k} \right]}_q}}}{q^{\frac{{r(r - 1)}}{2} - r(n - v + k)}}{{\left[ {\begin{array}{*{20}{c}}
							{n + 1}\\
							r
					\end{array}} \right]}_q}} \right)} \\
&	= &\sum\limits_{r = 0}^n {\left( {{{( - 1)}^r}{{\left[ {n + 1} \right]}_q}\frac{{\sum\limits_{s = 0}^{v - k - 1} {{q^{s(n + 1 - r)}}} }}{{{{\left[ {v - k} \right]}_q}}}{q^{\frac{{r(r - 1)}}{2} - r(n - v + k)}}{{\left[ {\begin{array}{*{20}{c}}
							n\\
							r
					\end{array}} \right]}_q}} \right)} \\
&	=& \frac{{{{\left[ {n + 1} \right]}_q}}}{{{{\left[ {v - k} \right]}_q}}}\sum\limits_{r = 0}^n {{{( - 1)}^r}\left( {\sum\limits_{s = 0}^{v - k - 1} {{q^{s(n + 1 - r)}}} } \right){q^{\frac{{r(r - 1)}}{2} - r(n - v + k)}}{{\left[ {\begin{array}{*{20}{c}}
						n\\
						r
				\end{array}} \right]}_q}} .
\end{eqnarray*}

For  any $s \in [0,v-k-1]$, according to Corollary  \ref{-cr}, we have
\begin{eqnarray*}
\sum\limits_{r = 0}^n {{{( - 1)}^r}{q^{s(n + 1 - r)}}{q^{\frac{{r(r - 1)}}{2} - r(n - v + k)}}{{\left[ {\begin{array}{*{20}{c}}
					n\\
					r
			\end{array}} \right]}_q}}  = {q^{s(n + 1)}}\sum\limits_{r = 0}^n {{{( - 1)}^r}{q^{\frac{{r(r - 1)}}{2} - r(n - v + k + s)}}{{\left[ {\begin{array}{*{20}{c}}
					n\\
					r
			\end{array}} \right]}_q}}  = 0.
\end{eqnarray*}

\end{enumerate}

This completes the proof.	
\end{proof}


\section{The fundamental relations }

In this section,  we consider the  quantum cluster algebra with principal coefficients $\mathcal{A}_{q}$ as follows:
the initial quantum seed is $\big(\widetilde{\mathbf{x}}, \Lambda, \widetilde{B}\big)$, where ${\widetilde{\mathbf{x}}}=\{x_1,\cdots,x_{2n}\}$,  $\Lambda  = \left( {\begin{array}{*{20}{c}}
		0&{ - D}\\
		D&{ - DB}
\end{array}} \right)$	and ${\widetilde{B}}=\left( {\begin{array}{*{20}{c}}
		B\\
		{{I_n}}
\end{array}} \right)$, where  $D=\operatorname{diag}(d_1, d_2, \dots, d_n)$ with $d_i \in \mathbb{Z}^+$ and
\[B = \left( {\begin{array}{*{20}{c}}
		0&{{b_{12}}}& \cdots &{{b_{1n}}}\\
		{{b_{21}}}& \ddots & \ddots & \vdots \\
		\vdots & \ddots & \ddots &{{b_{n - 1n}}}\\
		{{b_{n1}}}& \cdots &{{b_{nn - 1}}}&0
\end{array}} \right)\] with $d_jb_{ji}=-d_ib_{ij}.$
For any $k \in [1,n]$, denote $\big(\widetilde{\mathbf{x}}_k, \Lambda_k, \widetilde{B}_k\big) :=\mu_k \big(\widetilde{\mathbf{x}},\Lambda,\widetilde{B}\big)$
where $\widetilde{\mathbf{x}}_k=\{x_1,\cdots,x_{k-1},y_k:=x_k',x_{k+1},\cdots,x_{2n}\}$, $\Lambda_k$ and $ \widetilde{B}_k$ are given by formulas (\ref{(B)}) and (\ref{(A)}).

Define a linear order  $\triangleleft$ on $[1,n]$.
The notation $ \prod\limits_{j \in [1,n]}^ \triangleleft$ represents that the product is taken in increasing order with
respect to $\triangleleft$.
\begin{lemma}\label{t-h}
For any integers $ t\geq 1 $ and $i \in [1,n]$,	we have
	\begin{equation}\label{(ynxn)n1}
y_i^tx_i^t = \prod\limits_{r = 1}^t {\left( {\prod\limits_{k \in [1,n]}^ \triangleleft  {x_k^{{{\left[ { - {b_{ki}}} \right]}_ + }}}  + {q^{ - \frac{{{d_i}}}{2} + {d_i}r}}\prod\limits_{k \in [1,n]}^ \triangleleft  {x_k^{{{\left[ {{b_{ki}}} \right]}_ + }}}  \cdot {x_{n + i}}} \right)}  ,
	\end{equation}
	\begin{equation}\label{(xnyn)n1}
		x_i^ty_i^t = \prod\limits_{r = 1}^t {\left( {\prod\limits_{k \in [1,n]}^ \triangleleft  {x_k^{{{\left[ { - {b_{ki}}} \right]}_ + }}}  + {q^{\frac{{{d_i}}}{2} - {d_i}r}}\prod\limits_{k \in [1,n]}^ \triangleleft  {x_k^{{{\left[ {{b_{ki}}} \right]}_ + }}}  \cdot {x_{n + i}}} \right)}   .
		\end{equation}			
\end{lemma}

\begin{proof}
	We prove (\ref{(ynxn)n1}) by induction on $t$.  The proof  of (\ref{(xnyn)n1}) is similar.

	When $t=1$, the proof follows from the exchange relation
\[{y_i} = {q^{ - \frac{{{d_i}}}{2}}}x_i^{ - 1}  {\prod\limits_{k \in [1,n]}^ \triangleleft  {x_k^{{{\left[ {{b_{ki}}} \right]}_ + }}} } \cdot{x_{n + i}} +   {\prod\limits_{k \in [1,n]}^ \triangleleft  {x_k^{{{\left[ { - {b_{ki}}} \right]}_ + }}} } \cdot x_i^{ - 1}.\]

Assume that the equation holds when $t=l$. When $t=l+1$, we have	
\begin{eqnarray*}
&&	y_i^{l + 1}x_i^{l + 1} = y_i^l\left( {\prod\limits_{k \in [1,n]}^ \triangleleft  {x_k^{{{\left[ { - {b_{ki}}} \right]}_ + }}}  + {q^{\frac{{{d_i}}}{2}}}\prod\limits_{k \in [1,n]}^ \triangleleft  {x_k^{{{\left[ {{b_{ki}}} \right]}_ + }}}  \cdot {x_{n + i}}} \right)x_i^l\\
	\end{eqnarray*}		
\begin{eqnarray*}
	&	= &y_i^lx_i^l\prod\limits_{k \in [1,n]}^ \triangleleft  {x_k^{{{\left[ { - {b_{ki}}} \right]}_ + }}}  + {q^{\frac{{{d_i}}}{2}}}y_i^l\prod\limits_{k \in [1,n]}^ \triangleleft  {x_k^{{{\left[ {{b_{ki}}} \right]}_ + }}}  \cdot {x_{n + i}}x_i^l\\
&	=&y_i^lx_i^l\left( {\prod\limits_{k \in [1,n]}^ \triangleleft  {x_k^{{{\left[ { - {b_{ki}}} \right]}_ + }}}  + {q^{ - \frac{{{d_i}}}{2} + {d_i}(l + 1)}}\prod\limits_{k \in [1,n]}^ \triangleleft  {x_k^{{{\left[ {{b_{ki}}} \right]}_ + }}}  \cdot {x_{n + i}}} \right)\\
&	=&\prod\limits_{r = 1}^l {\left( {\prod\limits_{k \in [1,n]}^ \triangleleft  {x_k^{{{\left[ { - {b_{ki}}} \right]}_ + }}}  + {q^{ - \frac{{{d_i}}}{2} + {d_i}r}}\prod\limits_{k \in [1,n]}^ \triangleleft  {x_k^{{{\left[ {{b_{ki}}} \right]}_ + }}}  \cdot {x_{n + i}}} \right)} \\
&&	\cdot \left( {\prod\limits_{k \in [1,n]}^ \triangleleft  {x_k^{{{\left[ { - {b_{ki}}} \right]}_ + }}}  + {q^{ - \frac{{{d_i}}}{2} + {d_i}(l + 1)}}\prod\limits_{k \in [1,n]}^ \triangleleft  {x_k^{{{\left[ {{b_{ki}}} \right]}_ + }}}  \cdot {x_{n + i}}} \right)\\
	&=& \prod\limits_{r = 1}^{l + 1} {\left( {\prod\limits_{k \in [1,n]}^ \triangleleft  {x_k^{{{\left[ { - {b_{ki}}} \right]}_ + }}}  + {q^{ - \frac{{{d_i}}}{2} + {d_i}r}}\prod\limits_{k \in [1,n]}^ \triangleleft  {x_k^{{{\left[ {{b_{ki}}} \right]}_ + }}}  \cdot {x_{n + i}}} \right)}.
\end{eqnarray*}

	This completes the proof.
\end{proof}
\begin{corollary}
	For any integers $ t\geq 1 $  and $i \in [1,n]$,	we have
	\begin{equation}\label{(ynxn)n}
y_i^tx_i^t = \sum\limits_{k = 0}^t {{{\left[ {\begin{array}{*{20}{c}}
						t\\
						k
				\end{array}} \right]}_{{q^{{d_i}}}}}{{\left( {{q^{{d_i}}}} \right)}^{\frac{{{k^2}}}{2}}}\prod\limits_{v \in [1,n]}^ \triangleleft  {x_v^{(t - k){{\left[ { - {b_{vi}}} \right]}_ + }}}  \cdot \prod\limits_{v \in [1,n]}^ \triangleleft  {x_v^{k{{\left[ {{b_{vi}}} \right]}_ + }}}  \cdot x_{n + i}^k} ,
	\end{equation}
	\begin{equation}\label{(xnyn)n}
x_i^ty_i^t = \sum\limits_{k = 0}^t {{{\left[ {\begin{array}{*{20}{c}}
						t\\
						k
				\end{array}} \right]}_{{q^{{d_i}}}}}{{\left( {{q^{{d_i}}}} \right)}^{\frac{{k(k - 2t)}}{2}}}\prod\limits_{v \in [1,n]}^ \triangleleft  {x_v^{(t - k){{\left[ { - {b_{vi}}} \right]}_ + }}}  \cdot \prod\limits_{v \in [1,n]}^ \triangleleft  {x_v^{k{{\left[ {{b_{vi}}} \right]}_ + }}}  \cdot x_{n + i}^k}   .
	\end{equation}			
\end{corollary}

\begin{proof}
We only prove the first equation, and the proof of the second one is similar.	

According to Corollary \ref{1q-erxiangshi} and Lemma \ref{t-h}, we have	
\begin{eqnarray*}
&&	y_i^tx_i^t = \prod\limits_{r = 1}^t {\left( {\prod\limits_{k \in [1,n]}^ \triangleleft  {x_k^{{{\left[ { - {b_{ki}}} \right]}_ + }}}  + {q^{ - \frac{{{d_i}}}{2} + {d_i}r}}\prod\limits_{k \in [1,n]}^ \triangleleft  {x_k^{{{\left[ {{b_{ki}}} \right]}_ + }}}  \cdot {x_{n + i}}} \right)} \\
	&=&\prod\limits_{r = 1}^t {\left( {\left( {\prod\limits_{k \in [1,n]}^ \triangleleft  {x_k^{{{\left[ { - {b_{ki}}} \right]}_ + }}} } \right) + {{\left( {{q^{{d_i}}}} \right)}^r}\left( {{q^{ - \frac{{{d_i}}}{2}}}\prod\limits_{k \in [1,n]}^ \triangleleft  {x_k^{{{\left[ {{b_{ki}}} \right]}_ + }}}  \cdot {x_{n + i}}} \right)} \right)} \\
&	= &\sum\limits_{k = 0}^t {{{\left[ {\begin{array}{*{20}{c}}
						t\\
						k
				\end{array}} \right]}_{q^{{d_i}}}}{{\left( {{q^{{d_i}}}} \right)}^{\frac{{k(k + 1)}}{2}}}\prod\limits_{v \in [1,n]}^ \triangleleft  {x_v^{(t - k){{\left[ { - {b_{vi}}} \right]}_ + }}}  \cdot {q^{ - \frac{{{d_i}k}}{2}}}\prod\limits_{v \in [1,n]}^ \triangleleft  {x_v^{k{{\left[ {{b_{vi}}} \right]}_ + }}}  \cdot x_{n + i}^k} \\
&	=&\sum\limits_{k = 0}^t {{{\left[ {\begin{array}{*{20}{c}}
						t\\
						k
				\end{array}} \right]}_{q^{{d_i}}}}{{\left( {{q^{{d_i}}}} \right)}^{\frac{{{k^2}}}{2}}}\prod\limits_{v \in [1,n]}^ \triangleleft  {x_v^{(t - k){{\left[ { - {b_{vi}}} \right]}_ + }}}  \cdot \prod\limits_{v \in [1,n]}^ \triangleleft  {x_v^{k{{\left[ {{b_{vi}}} \right]}_ + }}}  \cdot x_{n + i}^k} .
\end{eqnarray*}	

The proof follows.
\end{proof}

\begin{lemma}\label{lemma2}
For any  $i,j \in [1,n]$,

\begin{enumerate}\itemsep=0pt
\item when $b_{ij}<0$, we have
\begin{equation}\label{(lemma2)}
\sum\limits_{t = 0}^{ - {b_{ij}}} {{{\left( {{q^{{d_i}}}} \right)}^{ - t}}\left(\sum\limits_{r = 0}^t {{D_r}} \right)  y_i^{ - {b_{ij}} - t}x_i^{ - {b_{ij}} - 1}y_i^t}  = 0,
\end{equation}
where ${D_r} = {( - 1)^r}{\left( {{q^{{d_i}}}} \right)^{\frac{{r(r - 1)}}{2}}}{\left[ {\begin{array}{*{20}{c}}
				{ - {b_{ij}} + 1}\\
				r
		\end{array}} \right]_{{q^{{d_i}}}}};$
\item when $b_{ij}>0$, we have 	
\begin{equation}\label{(lemma21)}
\sum\limits_{t = 0}^{{b_{ij}}} {{{\left( {{q^{{d_i}}}} \right)}^{t{b_{ij}}}}\left(\sum\limits_{r = 0}^t {{F_r}}\right)   y_i^{{b_{ij}} - t}x_i^{{b_{ij}} - 1}y_i^t}  = 0,
\end{equation}
where ${F_r} = {( - 1)^r}{\left( {{q^{{d_i}}}} \right)^{\frac{{r(r - 1)}}{2} - {b_{ij}}r}}{\left[ {\begin{array}{*{20}{c}}
{{b_{ij}} + 1}\\
r
\end{array}} \right]_{{q^{{d_i}}}}}.$	
\end{enumerate}	
\end{lemma}

\begin{proof}
\begin{enumerate}\itemsep=0pt
\item According to  (\ref{(A)}), we have
\begin{eqnarray*}
&&	(\Lambda_{i})_{in+i} =  - {\lambda _{in + i}} + \sum\limits_{t = 1}^{2n} {{{[{b_{ti}}]}_ + }{\lambda _{tn + i}}}  = {d_i} + \sum\limits_{t = 1}^n {{{[{b_{ti}}]}_ + }{\lambda _{tn + i}}}  + \sum\limits_{t = n + 1}^{2n} {{{[{b_{ti}}]}_ + }{\lambda _{tn + i}}} \\
	&	=& {d_i} + {[{b_{ii}}]_ + }{\lambda _{in + i}} + {[{b_{n + ii}}]_ + }{\lambda _{n + in + i}} = {d_i}.
\end{eqnarray*}	 	
	
According to  (\ref{(ynxn)n}) and (\ref{(xnyn)n}), we have
	\begin{eqnarray*}
&&	\sum\limits_{t = 0}^{ - {b_{ij}}} {{{\left( {{q^{{d_i}}}} \right)}^{ - t}}\left(\sum\limits_{r = 0}^t {{D_r}}\right)   y_i^{ - {b_{ij}} - t}x_i^{ - {b_{ij}} - 1}y_i^t} \\
&	=& y_i^{ - {b_{ij}}}x_i^{ - {b_{ij}} - 1} + \sum\limits_{t = 1}^{ - {b_{ij}}} {{{\left( {{q^{{d_i}}}} \right)}^{ - t}}\left(\sum\limits_{r = 0}^t {{D_r}}\right)   y_i^{ - {b_{ij}} - t}x_i^{ - {b_{ij}} - t}x_i^{t - 1}y_i^{t - 1}{y_i}} \\
&	=&\sum\limits_{t = 1}^{ - {b_{ij}}} {\left( {{{\left( {{q^{{d_i}}}} \right)}^{ - t}}\left(\sum\limits_{r = 0}^t {{D_r}} \right)  \sum\limits_{k = 0}^{ - {b_{ij}} - t} {\left( {{{\left[ {\begin{array}{*{20}{c}}
									{ - {b_{ij}} - t}\\
									k
							\end{array}} \right]}_{{q^{{d_i}}}}}{{\left( {{q^{{d_i}}}} \right)}^{\frac{{{k^2}}}{2}}}{C_{ - {b_{ij}} - t,k}}x_{n + i}^k} \right)} } \right.} \\
&&	\left. { \cdot \sum\limits_{k = 0}^{t - 1} {\left( {{{\left[ {\begin{array}{*{20}{c}}
								{t - 1}\\
								k
						\end{array}} \right]}_{{q^{{d_i}}}}}{{\left( {{q^{{d_i}}}} \right)}^{\frac{{k(k - 2(t - 1))}}{2}}}{C_{t - 1,k}}x_{n + i}^k} \right)} } \right){y_i} \\
					&&+ {y_i}\sum\limits_{k = 0}^{ - {b_{ij}} - 1} {{{\left[ {\begin{array}{*{20}{c}}
						{ - {b_{ij}} - 1}\\
						k
				\end{array}} \right]}_{{q^{{d_i}}}}}{{\left( {{q^{{d_i}}}} \right)}^{\frac{{{k^2}}}{2}}}{C_{ - {b_{ij}} - 1,k}}x_{n + i}^k} \\
&	=& \sum\limits_{t = 1}^{ - {b_{ij}}} {\left( {{{\left( {{q^{{d_i}}}} \right)}^{ - t}}\left(\sum\limits_{r = 0}^t {{D_r}} \right)   \sum\limits_{k = 0}^{ - {b_{ij}} - t} {\left( {{{\left[ {\begin{array}{*{20}{c}}
									{ - {b_{ij}} - t}\\
									k
							\end{array}} \right]}_{{q^{{d_i}}}}}{{\left( {{q^{{d_i}}}} \right)}^{\frac{{{k^2}}}{2}}}{C_{ - {b_{ij}} - t,k}}x_{n + i}^k} \right)} } \right.} \\
&&	\left. { \cdot \sum\limits_{k = 0}^{t - 1} {\left( {{{\left[ {\begin{array}{*{20}{c}}
								{t - 1}\\
								k
						\end{array}} \right]}_{{q^{{d_i}}}}}{{\left( {{q^{{d_i}}}} \right)}^{\frac{{k(k - 2(t - 1))}}{2}}}{C_{t - 1,k}}x_{n + i}^k} \right)} } \right){y_i} \\
							\end{eqnarray*}		
					\begin{eqnarray*}
					&&+ \left( {\sum\limits_{k = 0}^{ - {b_{ij}} - 1} {{{\left[ {\begin{array}{*{20}{c}}
							{ - {b_{ij}} - 1}\\
							k
					\end{array}} \right]}_{{q^{{d_i}}}}}{{\left( {{q^{{d_i}}}} \right)}^{\frac{{{k^2} + 2k}}{2}}}{C_{ - {b_{ij}} - 1,k}}x_{n + i}^k} } \right){y_i},\\
	\end{eqnarray*}
	where ${C_{l,k}} = \prod\limits_{v \in [1,n]}^ \triangleleft  {x_v^{(l - k){{\left[ { - {b_{vi}}} \right]}_ + }}}  \cdot \prod\limits_{v \in [1,n]}^ \triangleleft  {x_v^{k{{\left[ {{b_{vi}}} \right]}_ + }}} $. \\
	
For any $k \in [0,- {b_{ij}}-1]$, we compute
\begin{eqnarray*}
	&&	{\left[ {\begin{array}{*{20}{c}}
					{ - {b_{ij}} - 1}\\
					k
			\end{array}} \right]_{{q^{{d_i}}}}}{\left( {{q^{{d_i}}}} \right)^{\frac{{{k^2} + 2k}}{2}}}{C_{ - {b_{ij}} - 1,k}}x_{n + i}^k\\
		&& + {\left( {{q^{{d_i}}}} \right)^{ - 1}}\left(\sum\limits_{r = 0}^1 {{D_r}}\right)   {\left[ {\begin{array}{*{20}{c}}
					{ - {b_{ij}} - 1}\\
					k
			\end{array}} \right]_{{q^{{d_i}}}}}{\left( {{q^{{d_i}}}} \right)^{\frac{{{k^2}}}{2}}}{C_{ - {b_{ij}} - 1,k}}x_{n + i}^k{\rm{ }}\\
	&&		+ {\left( {{q^{{d_i}}}} \right)^{ - 2}}\left( {\sum\limits_{r = 0}^2 {{D_r}} } \right)\left( {{{\left[ {\begin{array}{*{20}{c}}
							{ - {b_{ij}} - 2}\\
							k
					\end{array}} \right]}_{{q^{{d_i}}}}}{{\left[ {\begin{array}{*{20}{c}}
							1\\
							0
					\end{array}} \right]}_{{q^{{d_i}}}}}{{\left( {{q^{{d_i}}}} \right)}^{\frac{{{k^2}}}{2}}}{C_{ - {b_{ij}} - 2,k}}{C_{1,0}}x_{n + i}^k} \right.\\
	&&	\left. { + {{\left[ {\begin{array}{*{20}{c}}
							{ - {b_{ij}} - 2}\\
							{k - 1}
					\end{array}} \right]}_{{q^{{d_i}}}}}{{\left[ {\begin{array}{*{20}{c}}
							1\\
							1
					\end{array}} \right]}_{{q^{{d_i}}}}}{{\left( {{q^{{d_i}}}} \right)}^{\frac{{{{(k - 1)}^2} - 1}}{2}}}{C_{ - {b_{ij}} - 2,k - 1}}{C_{1,1}}x_{n + i}^k} \right)\\
	&&  +  \cdots  + {\left( {{q^{{d_i}}}} \right)^{ - t}}\left( {\sum\limits_{r = 0}^t {{D_r}} } \right)\left( {{{\left[ {\begin{array}{*{20}{c}}
						{ - {b_{ij}} - t}\\
						k
				\end{array}} \right]}_{{q^{{d_i}}}}}{{\left[ {\begin{array}{*{20}{c}}
						{t - 1}\\
						0
				\end{array}} \right]}_{{q^{{d_i}}}}}{{\left( {{q^{{d_i}}}} \right)}^{\frac{{{k^2}}}{2}}}{C_{ - {b_{ij}} - t,k}}{C_{t - 1,0}}x_{n + i}^k} \right. \\
	&&	+ {\left[ {\begin{array}{*{20}{c}}
					{ - {b_{ij}} - t}\\
					{k - 1}
			\end{array}} \right]_{{q^{{d_i}}}}}{\left[ {\begin{array}{*{20}{c}}
					{t - 1}\\
					1
			\end{array}} \right]_{{q^{{d_i}}}}}{\left( {{q^{{d_i}}}} \right)^{\frac{{{{(k - 1)}^2} + (1 - 2(t - 1))}}{2}}}{C_{ - {b_{ij}} - t,k - 1}}{C_{t - 1,1}}x_{n + i}^k\\
		&&	+  \cdots  + {\left[ {\begin{array}{*{20}{c}}
					{ - {b_{ij}} - t}\\
					{k - s}
			\end{array}} \right]_{{q^{{d_i}}}}}{\left[ {\begin{array}{*{20}{c}}
					{t - 1}\\
					s
			\end{array}} \right]_{{q^{{d_i}}}}}{\left( {{q^{{d_i}}}} \right)^{\frac{{{{(k - s)}^2} + s(s - 2(t - 1))}}{2}}}{C_{ - {b_{ij}} - t,k - s}}{C_{t - 1,s}}x_{n + i}^k{\rm{ }}\\
	&&	+  \cdots  + {\left[ {\begin{array}{*{20}{c}}
					{ - {b_{ij}} - t}\\
					1
			\end{array}} \right]_{{q^{{d_i}}}}}{\left[ {\begin{array}{*{20}{c}}
					{t - 1}\\
					{k - 1}
			\end{array}} \right]_{{q^{{d_i}}}}}{\left( {{q^{{d_i}}}} \right)^{\frac{{1 + (k - 1)(k - 1 - 2(t - 1))}}{2}}}{C_{ - {b_{ij}} - t,1}}{C_{t - 1,k - 1}}x_{n + i}^k{\rm{ }}\\			
	&&  \left. { + {{\left[ {\begin{array}{*{20}{c}}
						{ - {b_{ij}} - t}\\
						0
				\end{array}} \right]}_{{q^{{d_i}}}}}{{\left[ {\begin{array}{*{20}{c}}
						{t - 1}\\
						k
				\end{array}} \right]}_{{q^{{d_i}}}}}{{\left( {{q^{{d_i}}}} \right)}^{\frac{{k(k - 2(t - 1))}}{2}}}{C_{ - {b_{ij}} - t,0}}{C_{t - 1,k}}x_{n + i}^k} \right)\\
	&&	+  \cdots  + {\left( {{q^{{d_i}}}} \right)^{{b_{ij}}}}\left(\sum\limits_{r = 0}^{ - {b_{ij}}} {{D_r}}\right)   {\left[ {\begin{array}{*{20}{c}}
					{ - {b_{ij}} - 1}\\
					k
			\end{array}} \right]_{{q^{{d_i}}}}}{\left( {{q^{{d_i}}}} \right)^{\frac{{k(k - 2( - {b_{ij}} - 1))}}{2}}}{C_{ - {b_{ij}} - 1,k}}x_{n + i}^k\\
	&	=& {\left( {{q^{{d_i}}}} \right)^{\frac{{{k^2} + 2k}}{2}}}{C_{ - {b_{ij}} - 1,k}}x_{n + i}^k\left( {{{\left[ {\begin{array}{*{20}{c}}
						{ - {b_{ij}} - 1}\\
						k
				\end{array}} \right]}_{{q^{{d_i}}}}} + {{\left( {{q^{{d_i}}}} \right)}^{ - (1 + k)}}\left( {\sum\limits_{r = 0}^1 {{D_r}} } \right){{\left[ {\begin{array}{*{20}{c}}
						{ - {b_{ij}} - 1}\\
						k
				\end{array}} \right]}_{{q^{{d_i}}}}}} \right.\\
&& + {\left( {{q^{{d_i}}}} \right)^{ - 2(1 + k)}}\left( {\sum\limits_{r = 0}^2 {{D_r}} } \right)\left( {{{\left( {{q^{{d_i}}}} \right)}^k}{{\left[ {\begin{array}{*{20}{c}}
								{ - {b_{ij}} - 2}\\
								k
						\end{array}} \right]}_{{q^{{d_i}}}}} + {{\left[ {\begin{array}{*{20}{c}}
								{ - {b_{ij}} - 2}\\
								{k - 1}
						\end{array}} \right]}_{{q^{{d_i}}}}}} \right)\\
			&&	+  \cdots  + {\left( {{q^{{d_i}}}} \right)^{ - t(1 + k)}}\left( {\sum\limits_{r = 0}^t {{D_r}} } \right)\left( {{{\left( {{q^{{d_i}}}} \right)}^{(t - 1)k}}{{\left[ {\begin{array}{*{20}{c}}
											{ - {b_{ij}} - t}\\
											k
									\end{array}} \right]}_{{q^{{d_i}}}}}} \right.\\
								\end{eqnarray*}
								\begin{eqnarray*}
				&&		+ {\left( {{q^{{d_i}}}} \right)^{(t - 2)(k - 1)}}{\left[ {\begin{array}{*{20}{c}}
									{ - {b_{ij}} - t}\\
									{k - 1}
							\end{array}} \right]_{{q^{{d_i}}}}}{\left[ {\begin{array}{*{20}{c}}
									{t - 1}\\
									1
							\end{array}} \right]_{{q^{{d_i}}}}}\\
				&&		+  \cdots  + {\left( {{q^{{d_i}}}} \right)^{(t - 1 - s)(k - s)}}{\left[ {\begin{array}{*{20}{c}}
									{ - {b_{ij}} - t}\\
									{k - s}
							\end{array}} \right]_{{q^{{d_i}}}}}{\left[ {\begin{array}{*{20}{c}}
									{t - 1}\\
									s
							\end{array}} \right]_{{q^{{d_i}}}}}\\
				&&		\left. { +  \cdots  + {{\left( {{q^{{d_i}}}} \right)}^{(t - k)}}{{\left[ {\begin{array}{*{20}{c}}
											{ - {b_{ij}} - t}\\
											1
									\end{array}} \right]}_{{q^{{d_i}}}}}{{\left[ {\begin{array}{*{20}{c}}
											{t - 1}\\
											{k - 1}
									\end{array}} \right]}_{{q^{{d_i}}}}} + {{\left[ {\begin{array}{*{20}{c}}
											{t - 1}\\
											k
									\end{array}} \right]}_{{q^{{d_i}}}}}} \right)\\
					&&	\left. { +  \cdots  + {{\left( {{q^{{d_i}}}} \right)}^{{b_{ij}}(1 + k)}}\left( {\sum\limits_{r = 0}^{ - {b_{ij}}} {{D_r}} } \right){{\left[ {\begin{array}{*{20}{c}}
											{ - {b_{ij}} - 1}\\
											k
									\end{array}} \right]}_{{q^{{d_i}}}}}} \right)\\
		&\mathop  = \limits^{(\ref{(n-dk)})}&	{\left( {{q^{{d_i}}}} \right)^{\frac{{{k^2} + 2k}}{2}}}{C_{ - {b_{ij}} - 1,k}}x_{n + i}^k\left( {{{\left[ {\begin{array}{*{20}{c}}
								{ - {b_{ij}} - 1}\\
								k
						\end{array}} \right]}_{{q^{{d_i}}}}} + {{\left( {{q^{{d_i}}}} \right)}^{ - (1 + k)}}\left( {\sum\limits_{r = 0}^1 {{D_r}} } \right){{\left[ {\begin{array}{*{20}{c}}
								{ - {b_{ij}} - 1}\\
								k
						\end{array}} \right]}_{{q^{{d_i}}}}}} \right.\\
		&&	+ {\left( {{q^{{d_i}}}} \right)^{ - 2(1 + k)}}\left( {\sum\limits_{r = 0}^2 {{D_r}} } \right){\left[ {\begin{array}{*{20}{c}}
						{ - {b_{ij}} - 1}\\
						k
				\end{array}} \right]_{{q^{{d_i}}}}} +  \cdots  + {\left( {{q^{{d_i}}}} \right)^{ - t(1 + k)}}\left( {\sum\limits_{r = 0}^t {{D_r}} } \right){\left[ {\begin{array}{*{20}{c}}
						{ - {b_{ij}} - 1}\\
						k
				\end{array}} \right]_{{q^{{d_i}}}}}\\
		&&	\left. { +  \cdots  + {{\left( {{q^{{d_i}}}} \right)}^{{b_{ij}}(1 + k)}}\left( {\sum\limits_{r = 0}^{ - {b_{ij}}} {{D_r}} } \right){{\left[ {\begin{array}{*{20}{c}}
								{ - {b_{ij}} - 1}\\
								k
						\end{array}} \right]}_{{q^{{d_i}}}}}} \right)\\
		&	=& {\left( {{q^{{d_i}}}} \right)^{\frac{{{k^2} + 2k}}{2}}}{\left[ {\begin{array}{*{20}{c}}
						{ - {b_{ij}} - 1}\\
						k
				\end{array}} \right]_{{q^{{d_i}}}}}{C_{ - {b_{ij}} - 1,k}}x_{n + i}^k\sum\limits_{t = 0}^{ - {b_{ij}}} {{{\left( {{q^{{d_i}}}} \right)}^{ - t(1 + k)}}\sum\limits_{r = 0}^t {{D_r}} }
			\mathop  = \limits^{(\ref{(lemma3)})}0.
	\end{eqnarray*}

\item According to  (\ref{(ynxn)n}) and (\ref{(xnyn)n}), we have
\begin{eqnarray*}
&&	\sum\limits_{t = 0}^{{b_{ij}}} {{{\left( {{q^{{d_i}}}} \right)}^{t{b_{ij}}}}\left(\sum\limits_{r = 0}^t {{F_r}} \right)  y_i^{{b_{ij}} - t}x_i^{{b_{ij}} - 1}y_i^t} \\
&	= &y_i^{{b_{ij}}}x_i^{{b_{ij}} - 1} + \sum\limits_{t = 1}^{{b_{ij}}} {{{\left( {{q^{{d_i}}}} \right)}^{t{b_{ij}}}}\left(\sum\limits_{r = 0}^t {{F_r}}\right)    y_i^{{b_{ij}} - t}x_i^{{b_{ij}} - t}x_i^{t - 1}y_i^{t - 1}{y_i}} \\	
&	= &\sum\limits_{t = 1}^{{b_{ij}}} {\left( {{{\left( {{q^{{d_i}}}} \right)}^{t{b_{ij}}}}\left(\sum\limits_{r = 0}^t {{F_r}}  \right)  \sum\limits_{k = 0}^{{b_{ij}} - t} {\left( {{{\left[ {\begin{array}{*{20}{c}}
									{{b_{ij}} - t}\\
									k
							\end{array}} \right]}_{{q^{{d_i}}}}}{{\left( {{q^{{d_i}}}} \right)}^{\frac{{{k^2}}}{2}}}{C_{{b_{ij}} - t,k}}x_{n + i}^k} \right)} } \right.} \\
&&	\left. { \cdot \sum\limits_{k = 0}^{t - 1} {\left( {{{\left[ {\begin{array}{*{20}{c}}
								{t - 1}\\
								k
						\end{array}} \right]}_{{q^{{d_i}}}}}{{\left( {{q^{{d_i}}}} \right)}^{\frac{{k(k - 2(t - 1))}}{2}}}{C_{t - 1,k}}x_{n + i}^k} \right)} } \right){y_i}\\
					&& + {y_i}\sum\limits_{k = 0}^{{b_{ij}} - 1} {{{\left[ {\begin{array}{*{20}{c}}
						{{b_{ij}} - 1}\\
						k
				\end{array}} \right]}_{{q^{{d_i}}}}}{{\left( {{q^{{d_i}}}} \right)}^{\frac{{{k^2}}}{2}}}{C_{{b_{ij}} - 1,k}}x_{n + i}^k} \\
&	=& \sum\limits_{t = 1}^{{b_{ij}}} {\left( {{{\left( {{q^{{d_i}}}} \right)}^{t{b_{ij}}}}\left(\sum\limits_{r = 0}^t {{F_r}} \right) \sum\limits_{k = 0}^{{b_{ij}} - t} {\left( {{{\left[ {\begin{array}{*{20}{c}}
									{{b_{ij}} - t}\\
									k
							\end{array}} \right]}_{{q^{{d_i}}}}}{{\left( {{q^{{d_i}}}} \right)}^{\frac{{{k^2}}}{2}}}{C_{{b_{ij}} - t,k}}x_{n + i}^k} \right)} } \right.} \\
							\end{eqnarray*}
						\begin{eqnarray*}
&&	\left. { \cdot \sum\limits_{k = 0}^{t - 1} {\left( {{{\left[ {\begin{array}{*{20}{c}}
								{t - 1}\\
								k
						\end{array}} \right]}_{{q^{{d_i}}}}}{{\left( {{q^{{d_i}}}} \right)}^{\frac{{k(k - 2(t - 1))}}{2}}}{C_{t - 1,k}}x_{n + i}^k} \right)} } \right){y_i}\\
&&	+ \left( {\sum\limits_{k = 0}^{{b_{ij}} - 1} {{{\left[ {\begin{array}{*{20}{c}}
							{{b_{ij}} - 1}\\
							k
					\end{array}} \right]}_{{q^{{d_i}}}}}{{\left( {{q^{{d_i}}}} \right)}^{\frac{{{k^2} + 2k}}{2}}}{C_{{b_{ij}} - 1,k}}x_{n + i}^k} } \right){y_i},
\end{eqnarray*}
where ${C_{l,k}} = \prod\limits_{v \in [1,n]}^ \triangleleft  {x_v^{(l - k){{\left[ { - {b_{vi}}} \right]}_ + }}}  \cdot \prod\limits_{v \in [1,n]}^ \triangleleft  {x_v^{k{{\left[ {{b_{vi}}} \right]}_ + }}} $.\\

For any $k \in [0, {b_{ij}}-1]$, we compute
\begin{eqnarray*}
&&	{\left[ {\begin{array}{*{20}{c}}
				{{b_{ij}} - 1}\\
				k
		\end{array}} \right]_{{q^{{d_i}}}}}{\left( {{q^{{d_i}}}} \right)^{\frac{{{k^2} + 2k}}{2}}}{C_{{b_{ij}} - 1,k}}x_{n + i}^k\\
	&& + {\left( {{q^{{d_i}}}} \right)^{{b_{ij}}}}\left(\sum\limits_{r = 0}^1 {{F_r}}\right)    {\left[ {\begin{array}{*{20}{c}}
				{{b_{ij}} - 1}\\
				k
		\end{array}} \right]_{{q^{{d_i}}}}}{\left( {{q^{{d_i}}}} \right)^{\frac{{{k^2}}}{2}}}{C_{{b_{ij}} - 1,k}}x_{n + i}^k{\rm{ }}\\
&&	+ {\left( {{q^{{d_i}}}} \right)^{2{b_{ij}}}}\left( {\sum\limits_{r = 0}^2 {{F_r}} } \right)\left( {{{\left[ {\begin{array}{*{20}{c}}
						{{b_{ij}} - 2}\\
						k
				\end{array}} \right]}_{{q^{{d_i}}}}}{{\left[ {\begin{array}{*{20}{c}}
						1\\
						0
				\end{array}} \right]}_{{q^{{d_i}}}}}{{\left( {{q^{{d_i}}}} \right)}^{\frac{{{k^2}}}{2}}}{C_{{b_{ij}} - 2,k}}{C_{1,0}}x_{n + i}^k} \right.\\							
&&	\left. { + {{\left[ {\begin{array}{*{20}{c}}
						{{b_{ij}} - 2}\\
						{k - 1}
				\end{array}} \right]}_{{q^{{d_i}}}}}{{\left[ {\begin{array}{*{20}{c}}
						1\\
						1
				\end{array}} \right]}_{{q^{{d_i}}}}}{{\left( {{q^{{d_i}}}} \right)}^{\frac{{{{(k - 1)}^2} - 1}}{2}}}{C_{{b_{ij}} - 2,k - 1}}{C_{1,1}}x_{n + i}^k} \right)\\					
&&  +  \cdots  + {\left( {{q^{{d_i}}}} \right)^{t{b_{ij}}}}\left( {\sum\limits_{r = 0}^t {{F_r}} } \right)\left( {{{\left[ {\begin{array}{*{20}{c}}
					{{b_{ij}} - t}\\
					k
			\end{array}} \right]}_{{q^{{d_i}}}}}{{\left[ {\begin{array}{*{20}{c}}
					{t - 1}\\
					0
			\end{array}} \right]}_{{q^{{d_i}}}}}{{\left( {{q^{{d_i}}}} \right)}^{\frac{{{k^2}}}{2}}}{C_{{b_{ij}} - t,k}}{C_{t - 1,0}}x_{n + i}^k} \right.\\
&&	+ {\left[ {\begin{array}{*{20}{c}}
				{{b_{ij}} - t}\\
				{k - 1}
		\end{array}} \right]_{{q^{{d_i}}}}}{\left[ {\begin{array}{*{20}{c}}
				{t - 1}\\
				1
		\end{array}} \right]_{{q^{{d_i}}}}}{\left( {{q^{{d_i}}}} \right)^{\frac{{{{(k - 1)}^2} + (1 - 2(t - 1))}}{2}}}{C_{{b_{ij}} - t,k - 1}}{C_{t - 1,1}}x_{n + i}^k\\	
&&	+  \cdots  + {\left[ {\begin{array}{*{20}{c}}
				{{b_{ij}} - t}\\
				{k - s}
		\end{array}} \right]_{{q^{{d_i}}}}}{\left[ {\begin{array}{*{20}{c}}
				{t - 1}\\
				s
		\end{array}} \right]_{{q^{{d_i}}}}}{\left( {{q^{{d_i}}}} \right)^{\frac{{{{(k - s)}^2} + s(s - 2(t - 1))}}{2}}}{C_{{b_{ij}} - t,k - s}}{C_{t - 1,s}}x_{n + i}^k{\rm{ }}\\
&&	+  \cdots  + {\left[ {\begin{array}{*{20}{c}}
				{{b_{ij}} - t}\\
				1
		\end{array}} \right]_{{q^{{d_i}}}}}{\left[ {\begin{array}{*{20}{c}}
				{t - 1}\\
				{k - 1}
		\end{array}} \right]_{{q^{{d_i}}}}}{\left( {{q^{{d_i}}}} \right)^{\frac{{1 + (k - 1)(k - 1 - 2(t - 1))}}{2}}}{C_{{b_{ij}} - t,1}}{C_{t - 1,k - 1}}x_{n + i}^k{\rm{ }}\\
&& \left. { + {{\left[ {\begin{array}{*{20}{c}}
					{{b_{ij}} - t}\\
					0
			\end{array}} \right]}_{{q^{{d_i}}}}}{{\left[ {\begin{array}{*{20}{c}}
					{t - 1}\\
					k
			\end{array}} \right]}_{{q^{{d_i}}}}}{{\left( {{q^{{d_i}}}} \right)}^{\frac{{k(k - 2(t - 1))}}{2}}}{C_{{b_{ij}} - t,0}}{C_{t - 1,k}}x_{n + i}^k} \right)\\
&&	+  \cdots  + {\left( {{q^{{d_i}}}} \right)^{b_{ij}^2}}\left(\sum\limits_{r = 0}^{{b_{ij}}} {{F_r}} \right)  {\left[ {\begin{array}{*{20}{c}}
				{{b_{ij}} - 1}\\
				k
		\end{array}} \right]_{{q^{{d_i}}}}}{\left( {{q^{{d_i}}}} \right)^{\frac{{k(k - 2(  {b_{ij}} - 1))}}{2}}}{C_{{b_{ij}} - 1,k}}x_{n + i}^k\\
&	=&	{\left( {{q^{{d_i}}}} \right)^{\frac{{{k^2} + 2k}}{2}}}{C_{{b_{ij}} - 1,k}}x_{n + i}^k\left( {{{\left[ {\begin{array}{*{20}{c}}
						{{b_{ij}} - 1}\\
						k
				\end{array}} \right]}_{{q^{{d_i}}}}} + {{\left( {{q^{{d_i}}}} \right)}^{{b_{ij}} - k}}\left( {\sum\limits_{r = 0}^1 {{F_r}} } \right){{\left[ {\begin{array}{*{20}{c}}
						{{b_{ij}} - 1}\\
						k
				\end{array}} \right]}_{{q^{{d_i}}}}}} \right.\\
&&	+ {\left( {{q^{{d_i}}}} \right)^{2({b_{ij}} - k)}}\left( {\sum\limits_{r = 0}^2 {{F_r}} } \right)\left( {{{\left( {{q^{{d_i}}}} \right)}^k}{{\left[ {\begin{array}{*{20}{c}}
						{{b_{ij}} - 2}\\
						k
				\end{array}} \right]}_{{q^{{d_i}}}}} + {{\left[ {\begin{array}{*{20}{c}}
						{{b_{ij}} - 2}\\
						{k - 1}
				\end{array}} \right]}_{{q^{{d_i}}}}}} \right)  \\
				\end{eqnarray*}
			\begin{eqnarray*}
&&+  \cdots	+ {\left( {{q^{{d_i}}}} \right)^{t({b_{ij}} - k)}}\left( {\sum\limits_{r = 0}^t {{F_r}} } \right)\left( {{{\left( {{q^{{d_i}}}} \right)}^{(t - 1)k}}{{\left[ {\begin{array}{*{20}{c}}
						{{b_{ij}} - t}\\
						k
				\end{array}} \right]}_{{q^{{d_i}}}}} } \right.\\
	&&{ + {{\left( {{q^{{d_i}}}} \right)}^{(t - 2)(k - 1)}}{{\left[ {\begin{array}{*{20}{c}}
								{{b_{ij}} - t}\\
								{k - 1}
						\end{array}} \right]}_{{q^{{d_i}}}}}{{\left[ {\begin{array}{*{20}{c}}
								{t - 1}\\
								1
						\end{array}} \right]}_{{q^{{d_i}}}}}}\\
&&	+  \cdots  + {\left( {{q^{{d_i}}}} \right)^{(t - 1 - s)(k - s)}}{\left[ {\begin{array}{*{20}{c}}
				{{b_{ij}} - t}\\
				{k - s}
		\end{array}} \right]_{{q^{{d_i}}}}}{\left[ {\begin{array}{*{20}{c}}
				{t - 1}\\
				s
		\end{array}} \right]_{{q^{{d_i}}}}}\\
&&	\left. { +  \cdots  + {{\left( {{q^{{d_i}}}} \right)}^{t - k}}{{\left[ {\begin{array}{*{20}{c}}
						{{b_{ij}} - t}\\
						1
				\end{array}} \right]}_{{q^{{d_i}}}}}{{\left[ {\begin{array}{*{20}{c}}
						{t - 1}\\
						{k - 1}
				\end{array}} \right]}_{{q^{{d_i}}}}} + {{\left[ {\begin{array}{*{20}{c}}
						{t - 1}\\
						k
				\end{array}} \right]}_{{q^{{d_i}}}}}} \right)\\
&&	\left. { +  \cdots  + {{\left( {{q^{{d_i}}}} \right)}^{{b_{ij}}({b_{ij}} - k)}}\left( {\sum\limits_{r = 0}^{{b_{ij}}} {{F_r}} } \right){{\left[ {\begin{array}{*{20}{c}}
						{{b_{ij}} - 1}\\
						k
				\end{array}} \right]}_{{q^{{d_i}}}}}} \right)\\
	&\mathop  = \limits^{(\ref{(n-dk)})}&	{\left( {{q^{{d_i}}}} \right)^{\frac{{{k^2} + 2k}}{2}}}{C_{{b_{ij}} - 1,k}}x_{n + i}^k\left( {{{\left[ {\begin{array}{*{20}{c}}
							{{b_{ij}} - 1}\\
							k
					\end{array}} \right]}_{{q^{{d_i}}}}} + {{\left( {{q^{{d_i}}}} \right)}^{{b_{ij}} - k}}\left( {\sum\limits_{r = 0}^1 {{F_r}} } \right){{\left[ {\begin{array}{*{20}{c}}
							{{b_{ij}} - 1}\\
							k
					\end{array}} \right]}_{{q^{{d_i}}}}}} \right.\\
	&&	+ {\left( {{q^{{d_i}}}} \right)^{2({b_{ij}} - k)}}\left( {\sum\limits_{r = 0}^2 {{F_r}} } \right){\left[ {\begin{array}{*{20}{c}}
					{{b_{ij}} - 1}\\
					k
			\end{array}} \right]_{{q^{{d_i}}}}} +  \cdots  + {\left( {{q^{{d_i}}}} \right)^{t({b_{ij}} - k)}}\left( {\sum\limits_{r = 0}^t {{F_r}} } \right){\left[ {\begin{array}{*{20}{c}}
					{{b_{ij}} - 1}\\
					k
			\end{array}} \right]_{{q^{{d_i}}}}}\\
	&&	\left. { +  \cdots  + {{\left( {{q^{{d_i}}}} \right)}^{b_{ij}^{}({b_{ij}} - k)}}\left( {\sum\limits_{r = 0}^{{b_{ij}}} {{F_r}} } \right){{\left[ {\begin{array}{*{20}{c}}
							{{b_{ij}} - 1}\\
							k
					\end{array}} \right]}_{{q^{{d_i}}}}}} \right)\\
	&	=&{\left( {{q^{{d_i}}}} \right)^{\frac{{{k^2} + 2k}}{2}}}{\left[ {\begin{array}{*{20}{c}}
					{{b_{ij}} - 1}\\
					k
			\end{array}} \right]_{{q^{{d_i}}}}}{C_{{b_{ij}} - 1,k}}x_{n + i}^k\sum\limits_{t = 0}^{{b_{ij}}} {{{\left( {{q^{{d_i}}}} \right)}^{  t({b_{ij}} - k)}}\sum\limits_{r = 0}^t {{F_r}} }
	\mathop  = \limits^{(\ref{(lemma31)})}0.
\end{eqnarray*}
\end{enumerate}
	 This completes the proof.		
\end{proof}

\begin{theorem}\label{serry}
For any  $i\neq j \in [1,n]$,
\begin{enumerate}\itemsep=0pt
\item when $b_{ij}\le 0$, we have
\[\sum\limits_{r = 0}^{-{b_{ij}} + 1} {{{( - 1)}^r}{{\left( {{q^{{d_i}}}} \right)}^{\frac{{r(r - 1)}}{2}}}{{\left[ {\begin{array}{*{20}{c}}
					{-{b_{ij}}  + 1}\\
					r
			\end{array}} \right]}_{{q^{{d_i}}}}}y_i^{-{b_{ij}}  + 1 - r}{y_j}y_i^r}  = 0;\]
		
\item when $b_{ij}> 0$, we have	
		
		\[\sum\limits_{r = 0}^{{b_{ij}} + 1} {{{( - 1)}^r}{{\left( {{q^{{d_i}}}} \right)}^{\frac{{r(r - 1)}}{2}-rb_{ij}}}{{\left[ {\begin{array}{*{20}{c}}
							{{b_{ij}} + 1}\\
							r
					\end{array}} \right]}_{{q^{{d_i}}}}}y_i^{{b_{ij}} + 1 - r}{y_j}y_i^r}  = 0.\]
\end{enumerate}
\item					
					
\end{theorem}

\begin{proof}	

\begin{enumerate}\itemsep=0pt
\item By using the exchange relations, we have
	\begin{eqnarray*}
		{y_i} &=& {q^{ - \frac{{{d_i}}}{2}}}x_i^{ - 1}x_j^{{b_{ji}}}\prod\limits_{k \in [1,n],k \ne j}^ \triangleleft  {x_k^{{{\left[ {{b_{ki}}} \right]}_ + }}}  \cdot {x_{n + i}} + \prod\limits_{k \in [1,n],k \ne j}^ \triangleleft  {x_k^{{{\left[ { - {b_{ki}}} \right]}_ + }}}  \cdot x_i^{ - 1},\\
{y_j} &= &{q^{ - \frac{{{d_j}}}{2}}}x_j^{ - 1} {\prod\limits_{k \in [1,n],k \ne i}^ \triangleleft  {x_k^{{{\left[ {{b_{kj}}} \right]}_ + }}} } \cdot{x_{n + j}} +  {\prod\limits_{k \in [1,n],k \ne i}^ \triangleleft  {x_k^{{{\left[ { - {b_{kj}}} \right]}_ + }}} } \cdot x_i^{ - {b_{ij}}}x_j^{ - 1}.
	\end{eqnarray*}	
	
Thus, we get	
	\begin{eqnarray*}
		{y_i}{y_j} &= &\left( {{q^{ - \frac{{{d_i}}}{2}}}x_i^{ - 1}x_j^{{b_{ji}}} {\prod\limits_{k \in [1,n],k \ne j}^ \triangleleft  {x_k^{{{\left[ {{b_{ki}}} \right]}_ + }}} } \cdot {x_{n + i}} +  {\prod\limits_{k \in [1,n],k \ne j}^ \triangleleft  {x_k^{{{\left[ { - {b_{ki}}} \right]}_ + }}} }   \cdot x_i^{ - 1}} \right)\\
	&&	\cdot \left( {{q^{ - \frac{{{d_j}}}{2}}}x_j^{ - 1}{\prod\limits_{k \in [1,n],k \ne i}^ \triangleleft  {x_k^{{{\left[ {{b_{kj}}} \right]}_ + }}} } \cdot {x_{n + j}} +  {\prod\limits_{k \in [1,n],k \ne i}^ \triangleleft  {x_k^{{{\left[ { - {b_{kj}}} \right]}_ + }}} } \cdot x_i^{ - {b_{ij}}}x_j^{ - 1}} \right)\\
	&	=& {q^{ - \frac{{{d_i}{d_j}}}{2}}}x_i^{ - 1}x_j^{{b_{ji}}} {\prod\limits_{k \in [1,n],k \ne j}^ \triangleleft  {x_k^{{{\left[ {{b_{ki}}} \right]}_ + }}} } \cdot{x_{n + i}}x_j^{ - 1} {\prod\limits_{k \in [1,n],k \ne i}^ \triangleleft  {x_k^{{{\left[ {{b_{kj}}} \right]}_ + }}} } \cdot {x_{n + j}}\\
	&&	+ {q^{ - \frac{{{d_i}}}{2}}}x_i^{ - 1}x_j^{{b_{ji}}} {\prod\limits_{k \in [1,n],k \ne j}^ \triangleleft  {x_k^{{{\left[ {{b_{ki}}} \right]}_ + }}} } \cdot {x_{n + i}} {\prod\limits_{k \in [1,n],k \ne i}^ \triangleleft  {x_k^{{{\left[ { - {b_{kj}}} \right]}_ + }}} } \cdot x_i^{ - {b_{ij}}}x_j^{ - 1}\\
	&&	+ {q^{ - \frac{{{d_j}}}{2}}}{\prod\limits_{k \in [1,n],k \ne j}^ \triangleleft  {x_k^{{{\left[ { - {b_{ki}}} \right]}_ + }}} } \cdot x_i^{ - 1}x_j^{ - 1}{\prod\limits_{k \in [1,n],k \ne i}^ \triangleleft  {x_k^{{{\left[ {{b_{kj}}} \right]}_ + }}} } \cdot{x_{n + j}}\\
	&&	+  {\prod\limits_{k \in [1,n],k \ne j}^ \triangleleft  {x_k^{{{\left[ { - {b_{ki}}} \right]}_ + }}} } \cdot x_i^{ - 1} {\prod\limits_{k \in [1,n],k \ne i}^ \triangleleft  {x_k^{{{\left[ { - {b_{kj}}} \right]}_ + }}} } \cdot x_i^{ - {b_{ij}}}x_j^{ - 1}\\
	&	=& {q^{ - \frac{{{d_i} + {d_j}}}{2}}}x_i^{ - 1}x_j^{{b_{ji}} - 1} {\prod\limits_{k \in [1,n],k \ne j}^ \triangleleft  {x_k^{{{\left[ {{b_{ki}}} \right]}_ + }}} } \cdot {\prod\limits_{k \in [1,n],k \ne i}^ \triangleleft  {x_k^{{{\left[ {{b_{kj}}} \right]}_ + }}} } \cdot{x_{n + i}}{x_{n + j}}\\
	&&	+ {q^{ - \frac{{{d_i}}}{2} - {d_i}{b_{ij}}}}x_i^{ - 1 - {b_{ij}}}x_j^{{b_{ji}} - 1} {\prod\limits_{k \in [1,n],k \ne j}^ \triangleleft  {x_k^{{{\left[ {{b_{ki}}} \right]}_ + }}} } \cdot {\prod\limits_{k \in [1,n],k \ne i}^ \triangleleft  {x_k^{{{\left[ { - {b_{kj}}} \right]}_ + }}} } \cdot{x_{n + i}}\\	
	&&	+ {q^{ - \frac{{{d_j}}}{2}}}x_i^{ - 1}x_j^{ - 1} {\prod\limits_{k \in [1,n],k \ne j}^ \triangleleft  {x_k^{{{\left[ { - {b_{ki}}} \right]}_ + }}} } \cdot {\prod\limits_{k \in [1,n],k \ne i}^ \triangleleft  {x_k^{{{\left[ {{b_{kj}}} \right]}_ + }}} } \cdot{x_{n + j}}\\
	&&	+ x_i^{ - 1 - {b_{ij}}}x_j^{ - 1} {\prod\limits_{k \in [1,n],k \ne j}^ \triangleleft  {x_k^{{{\left[ { - {b_{ki}}} \right]}_ + }}} } \cdot {\prod\limits_{k \in [1,n],k \ne i}^ \triangleleft  {x_k^{{{\left[ { - {b_{kj}}} \right]}_ + }}} }
	\end{eqnarray*}		
and	
\begin{eqnarray*}
		{y_j}{y_i} &=& {q^{ - \frac{{{d_j} + {d_i}}}{2}}}x_i^{ - 1}x_j^{{b_{ji}} - 1}  {\prod\limits_{k \in [1,n],k \ne j}^ \triangleleft  {x_k^{{{\left[ {{b_{ki}}} \right]}_ + }}} } \cdot {\prod\limits_{k \in [1,n],k \ne i}^ \triangleleft  {x_k^{{{\left[ {{b_{kj}}} \right]}_ + }}} } \cdot{x_{n + i}}{x_{n + j}}\\
	&&	+ {q^{ - \frac{{{d_i}}}{2}}}x_i^{ - 1 - {b_{ij}}}x_j^{{b_{ji}} - 1}  {\prod\limits_{k \in [1,n],k \ne j}^ \triangleleft  {x_k^{{{\left[ {{b_{ki}}} \right]}_ + }}} } \cdot{\prod\limits_{k \in [1,n],k \ne i}^ \triangleleft  {x_k^{{{\left[ { - {b_{kj}}} \right]}_ + }}} } \cdot{x_{n + i}}\\
		&&+ {q^{ - \frac{{{d_j}}}{2}}}x_i^{ - 1}x_j^{ - 1}  {\prod\limits_{k \in [1,n],k \ne j}^ \triangleleft  {x_k^{{{\left[ { - {b_{ki}}} \right]}_ + }}} }\cdot {\prod\limits_{k \in [1,n],k \ne i}^ \triangleleft  {x_k^{{{\left[ {{b_{kj}}} \right]}_ + }}} } \cdot{x_{n + j}}\\
		&&+ x_i^{ - 1 - {b_{ij}}}x_j^{ - 1} {\prod\limits_{k \in [1,n],k \ne j}^ \triangleleft  {x_k^{{{\left[ { - {b_{ki}}} \right]}_ + }}} } \cdot {\prod\limits_{k \in [1,n],k \ne i}^ \triangleleft  {x_k^{{{\left[ { - {b_{kj}}} \right]}_ + }}} }  .
	\end{eqnarray*}	
	
So, we obtain
\begin{eqnarray*}
{y_i}{y_j} - {y_j}{y_i} = \left( {{q^{ - \frac{{{d_i}}}{2} - {d_i}{b_{ij}}}} - {q^{ - \frac{{{d_i}}}{2}}}} \right)x_i^{ - {b_{ij}- 1 }}x_j^{{b_{ji}} - 1} {\prod\limits_{k \in [1,n],k \ne j}^ \triangleleft  {x_k^{{{\left[ {{b_{ki}}} \right]}_ + }}} } \cdot {\prod\limits_{k \in [1,n],k \ne i}^ \triangleleft  {x_k^{{{\left[ { - {b_{kj}}} \right]}_ + }}} } \cdot{x_{n + i}}.
\end{eqnarray*}	

(i) When $b_{ij}=0$,  we have
\[\sum\limits_{r = 0}^1 {{{( - 1)}^r}{{\left( {{q^{{d_i}}}} \right)}^{\frac{{r(r - 1)}}{2}}}{{\left[ {\begin{array}{*{20}{c}}
 					1\\
 					r
 			\end{array}} \right]}_{{q^{{d_i}}}}}y_i^{1 - r}{y_j}y_i^r}   ={y_i}{y_j}-  {y_j}{y_i}   = 0.\]

(ii) When $b_{ij}\ne 0$, we rewrite the above equation as follows
	\begin{eqnarray}\label{(y1y2)n}
{y_i}{y_j} = {y_j}{y_i} + Ax_i^{- {b_{ij}} - 1}x_j^{{b_{ji}} - 1}B{x_{n + i}} ,
	\end{eqnarray}
	where $A = {{q^{ - \frac{{{d_i}}}{2} - {d_i}{b_{ij}}}} - {q^{ - \frac{{{d_i}}}{2}}}}$  and $B =   {\prod\limits_{k \in [1,n],k \ne j}^ \triangleleft  {x_k^{{{\left[ { - {b_{ki}}} \right]}_ + }}} } \cdot {\prod\limits_{k \in [1,n],k \ne i}^ \triangleleft  {x_k^{{{\left[ {{b_{kj}}} \right]}_ + }}} }  $.

Thus, for any $s\geq 1$, we have
\begin{eqnarray}\label{(yinyj)}
	y_i^s{y_j}&\mathop  = \limits^{(\ref{(y1y2)n})}& y_i^{s - 1}\left( { {y_j}{y_i} + Ax_i^{- {b_{ij}} - 1}x_j^{{b_{ji}} - 1}B{x_{n + i}} } \right)\nonumber\\
	&=& y_i^{s - 1}{y_j}{y_i} + Ay_i^{s - 1}x_i^{- {b_{ij}} - 1}x_j^{{b_{ji}} - 1}B{x_{n + i}} \nonumber\\
	&\mathop  = \limits^{(\ref{(y1y2)n})} &y_i^{s- 2}{y_j}y_i^2 + {A}y_i^{s- 2}x_i^{- {b_{ij}} - 1}x_j^{{b_{ji}} - 1}B{x_{n + i}}{y_i} 	+ Ay_i^{s - 1}x_i^{- {b_{ij}} - 1}x_j^{{b_{ji}} - 1}B{x_{n + i}}\nonumber\\
	&\mathop  = \limits^{(\ref{(y1y2)n})}& {y_j}y_i^s + \sum\limits_{t = 0}^{s- 1} { {Ay_i^{s - 1 - t}x_i^{- {b_{ij}} - 1}x_j^{{b_{ji}} - 1}B{x_{n + i}}y_i^t} }.
	\end{eqnarray}

	Now, we calculate
	\begin{eqnarray*}
&&\sum\limits_{r = 0}^{-{b_{ij}} + 1} {{{( - 1)}^r}{{\left( {{q^{{d_i}}}} \right)}^{\frac{{r(r - 1)}}{2}}}{{\left[ {\begin{array}{*{20}{c}}
					{-{b_{ij}} + 1}\\
					r
			\end{array}} \right]}_{{q^{{d_i}}}}}y_i^{-{b_{ij}} + 1 - r}{y_j}y_i^r}   = \sum\limits_{r = 0}^{-{b_{ij}} + 1} {{D_r}y_i^{-{b_{ij}} + 1 - r}{y_j}y_i^r} \\
&	=& {D_0}y_i^{-{b_{ij}} + 1}{y_j} +  \cdots  + {D_r}y_i^{-{b_{ij}} + 1 - r}{y_j}y_i^r +  \cdots  + {D_{-{b_{ij}}}}y_i^{}{y_j}y_i^{-{b_{ij}}} + {D_{-{b_{ij}} + 1}}{y_j}y_i^{-{b_{ij}} + 1}\\
	&\mathop  = \limits^{(\ref{(yinyj)})}& {D_0}{y_j}y_i^{-{b_{ij}} + 1} + {D_0}\sum\limits_{t = 0}^{-{b_{ij}}} {\left( {Ay_i^{-{b_{ij}} - t}x_i^{- {b_{ij}} - 1}x_j^{{b_{ji}} - 1}B{x_{n + i}}y_i^t} \right)}   \\
&&	+ {D_1}{y_j}y_i^{-{b_{ij}} + 1} + {D_1}\sum\limits_{t = 1}^{-{b_{ij}}} {\left( {Ay_i^{-{b_{ij}} - t}x_i^{- {b_{ij}} - 1}x_j^{{b_{ji}} - 1}B{x_{n + i}}y_i^t} \right)}   \\
	&&+  \cdots  + {D_r}{y_j}y_i^{-{b_{ij}} + 1} + {D_r}\sum\limits_{t = r}^{-{b_{ij}}} {\left( {Ay_i^{-{b_{ij}} - t}x_i^{- {b_{ij}} - 1}x_j^{{b_{ji}} - 1}B{x_{n + i}}y_i^t} \right)}  \\
	&&+  \cdots  + {D_{-{b_{ij}}}}{y_j}y_i^{-{b_{ij}} + 1} + {D_{-{b_{ij}}}}Ax_i^{- {b_{ij}} - 1}x_j^{{b_{ji}} - 1}B{x_{n + i}}y_i^{-{b_{ij}}}  + {D_{-{b_{ij}} + 1}}{y_j}y_i^{-{b_{ij}} + 1}\\
		\end{eqnarray*}
	\begin{eqnarray*}
&	=& \left( {\sum\limits_{r = 0}^{-{b_{ij}} + 1} {{D_r}} } \right){y_j}y_i^{-{b_{ij}} + 1} + {D_0}Ay_i^{-{b_{ij}}}x_i^{- {b_{ij}} - 1}x_j^{{b_{ji}} - 1}B{x_{n + i}}\\&& + \left( {\sum\limits_{r = 0}^1 {{D_r}} } \right)Ay_i^{-{b_{ij}} - 1}x_i^{- {b_{ij}} - 1}x_j^{{b_{ji}} - 1}B{x_{n + i}}y_i^{}\\
&&	+  \cdots  + \left( {\sum\limits_{r = 0}^t {{D_r}} } \right)Ay_i^{-{b_{ij}} - t}x_i^{- {b_{ij}} - 1}x_j^{{b_{ji}} - 1}B{x_{n + i}}y_i^t \\&&+  \cdots  + \left( {\sum\limits_{r = 0}^{-{b_{ij}}} {{D_r}} } \right)Ax_i^{- {b_{ij}} - 1}x_j^{{b_{ji}} - 1}B{x_{n + i}}y_i^{-{b_{ij}}}\\
&	=& \left( {\sum\limits_{r = 0}^{-{b_{ij}} + 1} {{D_r}} } \right){y_j}y_i^{-{b_{ij}} + 1} + \sum\limits_{t = 0}^{-{b_{ij}}} {\left( {\left( {\sum\limits_{r = 0}^t {{D_r}} } \right)Ay_i^{-{b_{ij}} - t}x_i^{- {b_{ij}} - 1}x_j^{{b_{ji}} - 1}B{x_{n + i}}y_i^t} \right)}  ,
				\end{eqnarray*}
	where ${D_r} = {( - 1)^r}{\left( {{q^{{d_i}}}} \right)^{\frac{{r(r - 1)}}{2}}}{\left[ {\begin{array}{*{20}{c}}
				{-{b_{ij}} + 1}\\
				r
		\end{array}} \right]_{{q^{{d_i}}}}}$.	
	
Note that
\[\sum\limits_{r = 0}^{-{b_{ij}} + 1} {{D_r}}  = \sum\limits_{r = 0}^{-{b_{ij}} + 1} {{{( - 1)}^r}{{\left( {{q^{{d_i}}}} \right)}^{\frac{{r(r - 1)}}{2}}}{{\left[ {\begin{array}{*{20}{c}}
					{-{b_{ij}} + 1}\\
					r
			\end{array}} \right]}_{{q^{{d_i}}}}}}  \mathop  = \limits^{(\ref{(q=0)})} 0\]
and
	\begin{eqnarray*}
&&\sum\limits_{t = 0}^{-{b_{ij}}} {\left( {\left( {\sum\limits_{r = 0}^t {{D_r}} } \right)Ay_i^{-{b_{ij}} - t}x_i^{- {b_{ij}} - 1}x_j^{{b_{ji}} - 1}B{x_{n + i}}y_i^t} \right)} \\& = &A\left( {\sum\limits_{t = 0}^{ - {b_{ij}}} {{{\left( {{q^{{d_i}}}} \right)}^{ - t}}\left( {\sum\limits_{r = 0}^t {{D_r}} } \right)y_i^{ - {b_{ij}} - t}x_i^{ - {b_{ij}} - 1}y_i^t} } \right)x_j^{{b_{ji}} - 1}B{x_{n + i}}
	\mathop  = \limits^{(\ref{(lemma2)})}0.
	\end{eqnarray*}

\item By using the exchange relations, we have
\begin{eqnarray*}
	{y_i}& =& {q^{ - \frac{{{d_i}}}{2}}}x_i^{ - 1}\prod\limits_{k \in [1,n],k \ne j}^ \triangleleft  {x_k^{{{\left[ {{b_{ki}}} \right]}_ + }}}  \cdot {x_{n + i}} + x_i^{ - 1}x_j^{-{b_{ji}}}\prod\limits_{k \in [1,n],k \ne j}^ \triangleleft  {x_k^{{{\left[ { - {b_{ki}}} \right]}_ + }}},\\
{y_j} &=& {q^{ - \frac{{{d_j}}}{2}}}x_j^{ - 1}x_i^{{b_{ij}}}\prod\limits_{k \in [1,n],k \ne i}^ \triangleleft  {x_k^{{{\left[ {{b_{kj}}} \right]}_ + }}}  \cdot {x_{n + j}} + x_j^{ - 1}\prod\limits_{k \in [1,n],k \ne i}^ \triangleleft  {x_k^{{{\left[ { - {b_{kj}}} \right]}_ + }}}.
\end{eqnarray*}

Thus, we have	
\begin{eqnarray*}
	{y_i}{y_j}& =& {q^{ - \frac{{{d_i} + {d_j}}}{2}}}x_i^{ - 1}\prod\limits_{k \in [1,n],k \ne j}^ \triangleleft  {x_k^{{{\left[ {{b_{ki}}} \right]}_ + }}}  \cdot {x_{n + i}}x_j^{ - 1}x_i^{{b_{ij}}}\prod\limits_{k \in [1,n],k \ne i}^ \triangleleft  {x_k^{{{\left[ {{b_{kj}}} \right]}_ + }}}  \cdot {x_{n + j}}\\
&&	+ {q^{ - \frac{{{d_j}}}{2}}}x_i^{ - 1}x_j^{-{b_{ji}}}\prod\limits_{k \in [1,n],k \ne j}^ \triangleleft  {x_k^{{{\left[ { - {b_{ki}}} \right]}_ + }}}  \cdot x_j^{ - 1}x_i^{{b_{ij}}}\prod\limits_{k \in [1,n],k \ne i}^ \triangleleft  {x_k^{{{\left[ {{b_{kj}}} \right]}_ + }}}  \cdot {x_{n + j}}\\
		\end{eqnarray*}
\begin{eqnarray*}
&&	+ {q^{ - \frac{{{d_i}}}{2}}}x_i^{ - 1}\prod\limits_{k \in [1,n],k \ne j}^ \triangleleft  {x_k^{{{\left[ {{b_{ki}}} \right]}_ + }}}  \cdot {x_{n + i}}x_j^{ - 1}\prod\limits_{k \in [1,n],k \ne i}^ \triangleleft  {x_k^{{{\left[ { - {b_{kj}}} \right]}_ + }}} \\
&&	+ x_i^{ - 1}x_j^{-{b_{ji}}}\prod\limits_{k \in [1,n],k \ne j}^ \triangleleft  {x_k^{{{\left[ { - {b_{ki}}} \right]}_ + }}}  \cdot x_j^{ - 1}\prod\limits_{k \in [1,n],k \ne i}^ \triangleleft  {x_k^{{{\left[ { - {b_{kj}}} \right]}_ + }}} \\
&	=& {q^{ - \frac{{{d_i} + {d_j}}}{2} + {d_i}{b_{ij}}}}x_i^{{b_{ij}} - 1}x_j^{ - 1}\prod\limits_{k \in [1,n],k \ne j}^ \triangleleft  {x_k^{{{\left[ {{b_{ki}}} \right]}_ + }}}  \cdot \prod\limits_{k \in [1,n],k \ne i}^ \triangleleft  {x_k^{{{\left[ {{b_{kj}}} \right]}_ + }}}  \cdot {x_{n + i}}{x_{n + j}}\\
&&	+ {q^{ - \frac{{{d_j}}}{2}}}x_i^{{b_{ij}} - 1}x_j^{-{b_{ji}} - 1}\prod\limits_{k \in [1,n],k \ne j}^ \triangleleft  {x_k^{{{\left[ { - {b_{ki}}} \right]}_ + }}}  \cdot \prod\limits_{k \in [1,n],k \ne i}^ \triangleleft  {x_k^{{{\left[ {{b_{kj}}} \right]}_ + }}}  \cdot {x_{n + j}}\\
&&	+ {q^{ - \frac{{{d_i}}}{2}}}x_i^{ - 1}x_j^{ - 1}\prod\limits_{k \in [1,n],k \ne j}^ \triangleleft  {x_k^{{{\left[ {{b_{ki}}} \right]}_ + }}}  \cdot \prod\limits_{k \in [1,n],k \ne i}^ \triangleleft  {x_k^{{{\left[ { - {b_{kj}}} \right]}_ + }}}  \cdot {x_{n + i}}\\
&&	+ x_i^{ - 1}x_j^{-{b_{ji}} - 1}\prod\limits_{k \in [1,n],k \ne j}^ \triangleleft  {x_k^{{{\left[ { - {b_{ki}}} \right]}_ + }}}  \cdot \prod\limits_{k \in [1,n],k \ne i}^ \triangleleft  {x_k^{{{\left[ { - {b_{kj}}} \right]}_ + }}}
\end{eqnarray*}		
and	
\begin{eqnarray*}
	{y_j}{y_i} &= &{q^{ - \frac{{{d_j} + {d_i}}}{2} + {d_i}{b_{ij}}}}x_i^{{b_{ij}} - 1}x_j^{ - 1}\prod\limits_{k \in [1,n],k \ne j}^ \triangleleft  {x_k^{{{\left[ {{b_{ki}}} \right]}_ + }}}  \cdot \prod\limits_{k \in [1,n],k \ne i}^ \triangleleft  {x_k^{{{\left[ {{b_{kj}}} \right]}_ + }}}  \cdot   {x_{n + i}}{x_{n + j}}\\
	&&+ {q^{ - \frac{{{d_i}}}{2}}}x_i^{ - 1}x_j^{ - 1}\prod\limits_{k \in [1,n],k \ne j}^ \triangleleft  {x_k^{{{\left[ {{b_{ki}}} \right]}_ + }}}  \cdot \prod\limits_{k \in [1,n],k \ne i}^ \triangleleft  {x_k^{{{\left[ { - {b_{kj}}} \right]}_ + }}}  \cdot {x_{n + i}}\\
	&&+ {q^{ - \frac{{{d_j}}}{2} - {d_j}{b_{ji}}}}x_i^{{b_{ij}} - 1}x_j^{-{b_{ji}} - 1}\prod\limits_{k \in [1,n],k \ne j}^ \triangleleft  {x_k^{{{\left[ { - {b_{ki}}} \right]}_ + }}}  \cdot \prod\limits_{k \in [1,n],k \ne i}^ \triangleleft  {x_k^{{{\left[ {{b_{kj}}} \right]}_ + }}}  \cdot {x_{n + j}}\\
&&	+ x_i^{ - 1}x_j^{-{b_{ji}} - 1}\prod\limits_{k \in [1,n],k \ne j}^ \triangleleft  {x_k^{{{\left[ { - {b_{ki}}} \right]}_ + }}}  \cdot \prod\limits_{k \in [1,n],k \ne i}^ \triangleleft  {x_k^{{{\left[ { - {b_{kj}}} \right]}_ + }}}  .
\end{eqnarray*}	

So, we obtain
\begin{eqnarray*}
{y_i}{y_j}-{y_j}{y_i}  =  \left( {{q^{ - \frac{{{d_j}}}{2}}} - {q^{ - \frac{{{d_j}}}{2} - {d_j}{b_{ji}}}}} \right)x_i^{{b_{ij}} - 1}x_j^{-{b_{ji}} - 1}\prod\limits_{k \in [1,n],k \ne j}^ \triangleleft  {x_k^{{{\left[ { - {b_{ki}}} \right]}_ + }}}  \cdot \prod\limits_{k \in [1,n],k \ne i}^ \triangleleft  {x_k^{{{\left[ {{b_{kj}}} \right]}_ + }}}  \cdot {x_{n + j}}.
\end{eqnarray*}	

We rewrite the above formula as follows
\begin{eqnarray}\label{(y1y2)n11}
{y_i}{y_j} = {y_j}{y_i} + Cx_i^{{b_{ij}} - 1}x_j^{-{b_{ji}} - 1}D{x_{n + j}},
\end{eqnarray}
where $C = {{q^{ - \frac{{{d_j}}}{2}}} - {q^{ - \frac{{{d_j}}}{2} - {d_j}{b_{ji}}}}}$  and $D =  \prod\limits_{k \in [1,n],k \ne j}^ \triangleleft  {x_k^{{{\left[ { - {b_{ki}}} \right]}_ + }}}  \cdot \prod\limits_{k \in [1,n],k \ne i}^ \triangleleft  {x_k^{{{\left[ {{b_{kj}}} \right]}_ + }}} $.

Thus, for any $s\geq 1$, we have
\begin{eqnarray}\label{(yinyj)1}
	y_i^s{y_j}&\mathop  = \limits^{(\ref{(y1y2)n11})}& y_i^{s - 1}\left( {{y_j}{y_i} + Cx_i^{{b_{ij}} - 1}x_j^{-{b_{ji}} - 1}D{x_{n + j}}} \right)\nonumber\\
&	= &y_i^{s- 1}{y_j}{y_i} + Cy_i^{s - 1}x_i^{{b_{ij}} - 1}x_j^{-{b_{ji}} - 1}D{x_{n + j}}\nonumber\\
&	\mathop  = \limits^{(\ref{(y1y2)n11})} &y_i^{s - 2}{y_j}y_i^2 + Cy_i^{s - 2}x_i^{{b_{ij}} - 1}x_j^{-{b_{ji}} - 1}D{x_{n + j}}{y_i} + Cy_i^{s - 1}x_i^{{b_{ij}} - 1}x_j^{-{b_{ji}} - 1}D{x_{n + j}}\nonumber\\
&	\mathop  = \limits^{(\ref{(y1y2)n11})} &{y_j}y_i^s + \sum\limits_{t = 0}^{s - 1} {Cy_i^{s - 1 - t}x_i^{{b_{ij}} - 1}x_j^{-{b_{ji}} - 1}D{x_{n + j}}y_i^t}.
\end{eqnarray}

Now, we calculate
\begin{eqnarray*}
&&	\sum\limits_{r = 0}^{{b_{ij}} + 1} {{{( - 1)}^r}{{\left( {{q^{{d_i}}}} \right)}^{\frac{{r(r - 1)}}{2} - r{b_{ij}}}}{{\left[ {\begin{array}{*{20}{c}}
						{{b_{ij}} + 1}\\
						r
				\end{array}} \right]}_{{q^{{d_i}}}}}y_i^{{b_{ij}} + 1 - r}{y_j}y_i^r}  = \sum\limits_{r = 0}^{{b_{ij}} + 1} {{F_r}y_i^{{b_{ij}} + 1 - r}{y_j}y_i^r} \\
	&=& {F_0}y_i^{{b_{ij}} + 1}{y_j} + {F_1}y_i^{{b_{ij}}}{y_j}y_i^{} + {F_2}y_i^{{b_{ij}} - 1}{y_j}y_i^2 +  \cdots  + {F_{{b_{ij}}}}y_i^{}{y_j}y_i^{{b_{ij}}} + {F_{{b_{ij}} + 1}}{y_j}y_i^{{b_{ij}} + 1}\\
	&\mathop  = \limits^{(\ref{(yinyj)1})}&
		{F_0}{y_j}y_i^{{b_{ij}} + 1} + {F_0}\sum\limits_{t = 0}^{{b_{ij}}} {Cy_i^{{b_{ij}} - t}x_i^{{b_{ij}} - 1}x_j^{-{b_{ji}} - 1}D{x_{n + j}}y_i^t} \\
	&&	+ {F_1}{y_j}y_i^{{b_{ij}} + 1} + {F_1}\sum\limits_{t = 1}^{{b_{ij}}} {Cy_i^{{b_{ij}} - t}x_i^{{b_{ij}} - 1}x_j^{-{b_{ji}} - 1}D{x_{n + j}}y_i^t} \\
	&&	+ {F_2}{y_j}y_i^{{b_{ij}} + 1} + {F_2}\sum\limits_{t = 2}^{{b_{ij}}} {Cy_i^{{b_{ij}} - t}x_i^{{b_{ij}} - 1}x_j^{-{b_{ji}} - 1}D{x_{n + j}}y_i^t} \\
	&&	+  \cdots  + {F_{{b_{ij}}}}{y_j}y_i^{{b_{ij}} + 1} + {F_{{b_{ij}}}}Cx_i^{{b_{ij}} - 1}x_j^{-{b_{ji}} - 1}D{x_{n + j}}y_i^{{b_{ij}}} + {F_{{b_{ij}} + 1}}{y_j}y_i^{{b_{ij}} + 1}\\
	&	=& \left( {\sum\limits_{r = 0}^{{b_{ij}} + 1} {{F_r}} } \right){y_j}y_i^{{b_{ij}} + 1} + {F_0}Cy_i^{{b_{ij}}}x_i^{{b_{ij}} - 1}x_j^{-{b_{ji}} - 1}D{x_{n + j}}\\
		&& + \left( {\sum\limits_{r = 0}^1 {{F_r}} } \right)Cy_i^{{b_{ij}} - 1}x_i^{{b_{ij}} - 1}x_j^{-{b_{ji}} - 1}D{x_{n + j}}y_i^{}\\
	&&	+  \cdots  + \left( {\sum\limits_{r = 0}^t {{F_r}} } \right)Cy_i^{{b_{ij}} - t}x_i^{{b_{ij}} - 1}x_j^{{b_{ji}} - 1}D{x_{n + j}}y_i^t \\&&+  \cdots  + \left( {\sum\limits_{r = 0}^{{b_{ij}}} {{F_r}} } \right)Cx_i^{{b_{ij}} - 1}x_j^{-{b_{ji}} - 1}D{x_{n + j}}y_i^{{b_{ij}}}\\
	&	= &\left( {\sum\limits_{r = 0}^{{b_{ij}} + 1} {{F_r}} } \right){y_j}y_i^{{b_{ij}} + 1} + \sum\limits_{t = 0}^{{b_{ij}}} {\left( {\left( {\sum\limits_{r = 0}^t {{F_r}} } \right)Cy_i^{{b_{ij}} - t}x_i^{{b_{ij}} - 1}x_j^{-{b_{ji}} - 1}D{x_{n + j}}y_i^t} \right)}  ,
\end{eqnarray*}
where ${F_r} =  {{{( - 1)}^r}{{\left( {{q^{{d_i}}}} \right)}^{\frac{{r(r - 1)}}{2} - r{b_{ij}}}}{{\left[ {\begin{array}{*{20}{c}}
					{{b_{ij}} + 1}\\
					r
			\end{array}} \right]}_{{q^{{d_i}}}}}} $.	

Note that
\[\sum\limits_{r = 0}^{ {b_{ij}} + 1} {{F_r}}  =  \sum\limits_{r = 0}^{{b_{ij}} + 1} {{{( - 1)}^r}{{\left( {{q^{{d_i}}}} \right)}^{\frac{{r(r - 1)}}{2} - r{b_{ij}}}}{{\left[ {\begin{array}{*{20}{c}}
					{{b_{ij}} + 1}\\
					r
			\end{array}} \right]}_{{q^{{d_i}}}}}}    \mathop  = \limits^{(\ref{(-cr)})} 0.\]

On the other hand, 	according to  (\ref{(A)}), we have
\begin{eqnarray*}
(\Lambda_{i})_{in+j}& =&  - {\lambda _{in + j}} + \sum\limits_{t = 1}^{2n} {{{[{b_{ti}}]}_ + }{\lambda _{tn + j}}}  = \sum\limits_{t = 1}^n {{{[{b_{ti}}]}_ + }{\lambda _{tn + j}}}  + \sum\limits_{t = n + 1}^{2n} {{{[{b_{ti}}]}_ + }{\lambda _{tn + j}}} \\
	&	=& {[{b_{ji}}]_ + }{\lambda _{jn + j}} + {[{b_{n + ii}}]_ + }{\lambda _{n + in + j}} = {\lambda _{n + in + j}} =  - {d_i}{b_{ij}}.
\end{eqnarray*}	

Thus, we obtain
\begin{eqnarray*}
&&	\sum\limits_{t = 0}^{{b_{ij}}} {\left( {\left( {\sum\limits_{r = 0}^t {{F_r}} } \right)Cy_i^{{b_{ij}} - t}x_i^{{b_{ij}} - 1}x_j^{-{b_{ji}} - 1}D{x_{n + j}}y_i^t} \right)} \\
	&=&C\left( {\sum\limits_{t = 0}^{{b_{ij}}} {{{\left( {{q^{{d_i}}}} \right)}^{t{b_{ij}}}}\left( {\sum\limits_{r = 0}^t {{F_r}} } \right)y_i^{{b_{ij}} - t}x_i^{{b_{ij}} - 1}y_i^t} } \right)x_j^{ - {b_{ji}} - 1}D{x_{n + j}}\mathop  = \limits^{(\ref{(lemma21)})} 0.
\end{eqnarray*}
\end{enumerate}
This finishes the proof.
\end{proof}

For any given skew-symmetrizable matrix $B$, one can consider the corresponding Cartan matrix $C=(c_{ij})_{n\times n }$ with the same symmetrizer matrix $D=\operatorname{diag}(d_1, d_2, \dots, d_n)$, where $c_{ii}=2$ for $1\leq i\leq n$ and $c_{ij}=- \left| {b_{ij}} \right|$ for $1\leq i\neq j\leq n$.
Denote by ${{\widehat U^ + } }\left( {C,q} \right)$ the quotient algebra of  the free  algebra  over $\mathbb{Z}[{q^{\pm\frac{1}{2}}}]$ generated by ${\theta _i}\ (1\le i\le n)$ subject to the relations
\begin{eqnarray*}
	\sum\limits_{r = 0}^{\left| {{b_{ij}}} \right| + 1} {{{( - 1)}^r}{{\left( {{q^{{d_i}}}} \right)}^{\frac{{r(r - 1)}}{2}}}{{\left[ {\begin{array}{*{20}{c}}
						{\left| {{b_{ij}}} \right| + 1}\\
						r
				\end{array}} \right]}_{{q^{{d_i}}}}}\theta _i^{\left| {{b_{ij}}} \right| + 1 - r}{\theta _j}\theta _i^r},\\
	\sum\limits_{r = 0}^{\left| {{b_{ji}}} \right| + 1} {{{( - 1)}^r}{{\left( {{q^{{d_j}}}} \right)}^{\frac{{r(r - 1)}}{2}}}{{\left[ {\begin{array}{*{20}{c}}
						{\left| {{b_{ji}}} \right| + 1}\\
						r
				\end{array}} \right]}_{{q^{{d_j}}}}}\theta _j^r{\theta _i}\theta _j^{\left| {{b_{ji}}} \right| + 1 - r}},
\end{eqnarray*}
for all pairs $i\neq j$. Note that  after we twist the multiplication of this algebra by using the Euler form and
we obtain the positive part of the standard quantum group associated with the Cartan matrix $C=(c_{ij})_{n\times n }$ (See Ringel's comments on page $60$ in \cite{R4}).

We have the following result.
\begin{corollary}
	The assignment ${\theta _i} \rightarrow {y_i}\ (1 \le i \le n)$ defines a homomorphism of algebras from   ${{\widehat U^ + }}\left( {C,q} \right)$ to   $\mathcal{A}_{q}$.
\end{corollary}
\begin{proof}
Without loss of generality, assume that $b_{ij}\le 0$. Applying the bar-involution $-$ to the formula in Theorem \ref{serry} (2), we have

\begin{eqnarray*}
	0& =& \sum\limits_{r = 0}^{{b_{ji}} + 1} {{{( - 1)}^r}{{\left( {{q^{{d_j}}}} \right)}^{ - \frac{{r(r - 1)}}{2} + r{b_{ji}}}}\overline {{{\left[ {\begin{array}{*{20}{c}}
							{{b_{ji}} + 1}\\
							r
					\end{array}} \right]}_{{q^{{d_j}}}}}} y_j^r{y_i}y_j^{{b_{ji}} + 1 - r}} \\
&	\mathop  = \limits^{(\ref{q+-b1})}& \sum\limits_{r = 0}^{{b_{ji}} + 1} {{{( - 1)}^r}{{\left( {{q^{{d_j}}}} \right)}^{ - \frac{{r(r - 1)}}{2} + r{b_{ji}}}}{{\left( {{q^{{d_j}}}} \right)}^{ - r({b_{ji}} + 1 - r)}}{{\left[ {\begin{array}{*{20}{c}}
						{{b_{ji}} + 1}\\
						r
				\end{array}} \right]}_{{q^{{d_j}}}}}y_j^r{y_i}y_j^{{b_{ji}} + 1 - r}} \\
&	=& \sum\limits_{r = 0}^{{b_{ji}} + 1} {{{( - 1)}^r}{{\left( {{q^{{d_j}}}} \right)}^{\frac{{r(r - 1)}}{2}}}{{\left[ {\begin{array}{*{20}{c}}
						{{b_{ji}} + 1}\\
						r
				\end{array}} \right]}_{{q^{{d_j}}}}}y_j^r{y_i}y_j^{{b_{ji}} + 1 - r}.}
\end{eqnarray*}

Then, combining Theorem \ref{serry} (1), the proof is finished.
\end{proof}

\begin{example}\label{exam1}
	Given an initial quantum seed $\big(\widetilde{\mathbf{x}}, \Lambda, \widetilde{B}\big)$, where ${\widetilde{\mathbf{x}}}=\{x_1,x_2,x_3,x_{4}\}$, $$\Lambda  = \left( {\begin{array}{*{20}{c}}
			0&0&{ - 2}&0\\
			0&0&0&{ - 1}\\
			2&0&0&-2\\
			0&1&{ 2}&0
	\end{array}} \right) \text{ and }\,\,{\widetilde{B}}=\left( {\begin{array}{*{20}{c}}
			0&{ 1}\\
			-2&0\\
			1&0\\
			0&1
	\end{array}} \right).$$
By mutating the seed $(\widetilde{\mathbf{x}},\Lambda,\widetilde{B})$ in the direction $1$, we obtain the seed $(\widetilde{\mathbf{x}}_1,\Lambda_1,\widetilde{B}_1)$, where $\widetilde{\mathbf{x}}_1=\{y_1,x_2,x_3,x_{4}\}$, $\Lambda_1  = \left( {\begin{array}{*{20}{c}}
			0&0&{2}&-2\\
			0&0&0&{ - 1}\\
			-2&0&0&-2\\
			2&1&{ 2}&0
	\end{array}} \right)$ and  ${\widetilde{B}_1}=\left( {\begin{array}{*{20}{c}}
			0&{-1}\\
			2&0\\
			-1&1\\
			0&1
	\end{array}} \right)$ with the  exchange relation ${x_1}{y_1} = {q^{ - 1}}{x_3} + x_2^2.$
By mutating the seed $(\widetilde{\mathbf{x}},\Lambda,\widetilde{B})$ in the direction $2$, we obtain the seed $(\widetilde{\mathbf{x}}_2,\Lambda_2,\widetilde{B}_2)$, where $\widetilde{\mathbf{x}}_2=\{x_1,y_2,x_3,x_{4}\}$,
$$\Lambda_2  = \left( {\begin{array}{*{20}{c}}
			0&0&{ - 2}&0\\
			0&0&0&{  1}\\
			2&0&0&-2\\
			0&-1&{ 2}&0
	\end{array}} \right) \text{  and }  {\widetilde{B}_2}=\left( {\begin{array}{*{20}{c}}
			0&-1\\
			{  2}&0\\
			1&0\\
			0&{ - 1}
	\end{array}} \right)$$with the  exchange relation ${x_2}{y_2} = {q^{ - \frac{1}{2}}}{x_1}{x_4} + 1.$
We claim that
	\[\begin{array}{l}
		\sum\limits_{r = 0}^3 {{{( - 1)}^r}{q^{\frac{{r(r - 1)}}{2}}}{{\left[ {\begin{array}{*{20}{c}}
							3\\
							r
					\end{array}} \right]}_q}y_2^{3 - r}{y_1}y_2^r}  = 0,\\
		\sum\limits_{r = 0}^2 {{{( - 1)}^r}{{\left( {{q^2}} \right)}^{\frac{{r(r - 1)}}{2} - r}}{{\left[ {\begin{array}{*{20}{c}}
							2\\
							r
					\end{array}} \right]}_{{q^2}}}y_1^{2 - r}{y_2}y_1^r}  = 0.
	\end{array}\]	
	
	It is easy to check that
	\[{y_1}{y_2} = {q^{\frac{1}{2}}}x_2^{ - 1}{x_3}{x_4} + {q^{ -1}}x_1^{ - 1}x_2^{ - 1}{x_3} + {q^{ - \frac{1}{2}}}x_2^{}{x_4} + x_1^{ - 1}x_2^{},\]	
	\[{y_2}{y_1} = {q^{\frac{1}{2}}}x_2^{ - 1}{x_3}{x_4} + {q^{ - 1}}x_1^{ - 1}x_2^{ - 1}{x_3} + {q^{ - \frac{1}{2} + 2}}x_2^{}{x_4} + x_1^{ - 1}x_2^{}.\]
	
	It follows that	
	\begin{equation}\label{(ex)}
		{y_2}{y_1} = {y_1}{y_2} +A{x_2}{x_4},\end{equation}
	where $A={{q^{\frac{3}{2}}} - {q^{ - \frac{1}{2}}}}$.	
	
	Then, we have
	\begin{eqnarray*}
		&&\sum\limits_{r = 0}^3 {{{( - 1)}^r}{q^{\frac{{r(r - 1)}}{2}}}{{\left[ {\begin{array}{*{20}{c}}
							3\\
							r
					\end{array}} \right]}_q}y_2^{3 - r}{y_1}y_2^r} \\
		&	=& y_2^3{y_1} - {\left[ {\begin{array}{*{20}{c}}
					3\\
					1
			\end{array}} \right]_q}y_2^2{y_1}y_2^{} + q{\left[ {\begin{array}{*{20}{c}}
					3\\
					2
			\end{array}} \right]_q}y_2^{}{y_1}y_2^2 - {q^3}{y_1}y_2^3\\
		&	\mathop  = \limits^{(\ref{(ex)})}& {y_1}y_2^3 + Ay_2^2{x_2}{x_4} + A{y_2}{x_2}{x_4}{y_2} + A{x_2}{x_4}y_2^2\\
		&&	- {\left[ {\begin{array}{*{20}{c}}
					3\\
					1
			\end{array}} \right]_q}{y_1}y_2^3 - {\left[ {\begin{array}{*{20}{c}}
					3\\
					1
			\end{array}} \right]_q}A{y_2}{x_2}{x_4}{y_2} - {\left[ {\begin{array}{*{20}{c}}
					3\\
					1
			\end{array}} \right]_q}A{x_2}{x_4}y_2^2\\
		&&	+ q{\left[ {\begin{array}{*{20}{c}}
					3\\
					2
			\end{array}} \right]_q}{y_1}y_2^3 + q{\left[ {\begin{array}{*{20}{c}}
					3\\
					2
			\end{array}} \right]_q}A{x_2}{x_4}y_2^2 - {q^3}{y_1}y_2^3\\
		&	=& \left( {1 - {{\left[ {\begin{array}{*{20}{c}}
							3\\
							1
					\end{array}} \right]}_q} + q{{\left[ {\begin{array}{*{20}{c}}
							3\\
							2
					\end{array}} \right]}_q} - {q^3}} \right){y_1}y_2^3 + Ay_2^2{x_2}{x_4}\\
		&&	+ \left( {1 - {{\left[ {\begin{array}{*{20}{c}}
							3\\
							1
					\end{array}} \right]}_q}} \right)A{y_2}{x_2}{x_4}{y_2} + \left( {1 - {{\left[ {\begin{array}{*{20}{c}}
							3\\
							1
					\end{array}} \right]}_q} + q{{\left[ {\begin{array}{*{20}{c}}
							3\\
							2
					\end{array}} \right]}_q}} \right)A{x_2}{x_4}y_2^2.
	\end{eqnarray*}
	
	Note that
	\[1 - {\left[ {\begin{array}{*{20}{c}}
				3\\
				1
		\end{array}} \right]_q} + q{\left[ {\begin{array}{*{20}{c}}
				3\\
				2
		\end{array}} \right]_q} - {q^3} = 1 - 1 - q - {q^2} + q + {q^2} + {q^3} - {q^3} = 0.\]
	
	Thus, we get
	\begin{eqnarray*}
		&&	\sum\limits_{r = 0}^3 {{{( - 1)}^r}{q^{\frac{{r(r - 1)}}{2}}}{{\left[ {\begin{array}{*{20}{c}}
							3\\
							r
					\end{array}} \right]}_q}y_2^{3 - r}{y_1}y_2^r} \\
		&	=& Ay_2^2{x_2}{x_4} + \left( {1 - {{\left[ {\begin{array}{*{20}{c}}
							3\\
							1
					\end{array}} \right]}_q}} \right)A{y_2}{x_2}{x_4}{y_2} + \left( {1 - {{\left[ {\begin{array}{*{20}{c}}
							3\\
							1
					\end{array}} \right]}_q} + q{{\left[ {\begin{array}{*{20}{c}}
							3\\
							2
					\end{array}} \right]}_q}} \right)A{x_2}{x_4}y_2^2\\
		&	=& Ay_2^2{x_2}{x_4} + {q^{ - 1}}\left( {1 - {{\left[ {\begin{array}{*{20}{c}}
							3\\
							1
					\end{array}} \right]}_q}} \right)A{y_2}{x_2}{y_2}{x_4} + {q^{ - 2}}\left( {1 - {{\left[ {\begin{array}{*{20}{c}}
							3\\
							1
					\end{array}} \right]}_q} + q{{\left[ {\begin{array}{*{20}{c}}
							3\\
							2
					\end{array}} \right]}_q}} \right)A{x_2}y_2^2{x_4}\\
		&	=& A\left( {y_2^2{x_2} + {q^{ - 1}}\left( {1 - {{\left[ {\begin{array}{*{20}{c}}
								3\\
								1
						\end{array}} \right]}_q}} \right){y_2}{x_2}{y_2} + {q^{ - 2}}\left( {1 - {{\left[ {\begin{array}{*{20}{c}}
								3\\
								1
						\end{array}} \right]}_q} + q{{\left[ {\begin{array}{*{20}{c}}
								3\\
								2
						\end{array}} \right]}_q}} \right){x_2}y_2^2} \right){x_4}.
	\end{eqnarray*}
	
	Hence, the claim can be deduced  from the following direct computation
	\begin{eqnarray*}
		&&	y_2^2{x_2} + {q^{ - 1}}\left( {1 - {{\left[ {\begin{array}{*{20}{c}}
							3\\
							1
					\end{array}} \right]}_q}} \right){y_2}{x_2}{y_2} + {q^{ - 2}}\left( {1 - {{\left[ {\begin{array}{*{20}{c}}
							3\\
							1
					\end{array}} \right]}_q} + q{{\left[ {\begin{array}{*{20}{c}}
							3\\
							2
					\end{array}} \right]}_q}} \right){x_2}y_2^2\\
		&	=& {y_2}\left( {{q^{\frac{1}{2}}}{x_1}{x_4} + 1} \right) + \left( { - 1 - q} \right)\left( {{q^{\frac{1}{2}}}{x_1}{x_4} + 1} \right){y_2} + q\left( {{q^{ - \frac{1}{2}}}{x_1}{x_4} + 1} \right){y_2}\\
		&	=& {q^{\frac{1}{2}}}{x_1}{y_2}{x_4} + {y_2} + \left( { - {q^{ - \frac{1}{2}}} - {q^{\frac{1}{2}}}} \right){x_1}{y_2}{x_4} + \left( { - 1 - q} \right){y_2} + {q^{ - \frac{1}{2}}}{x_1}{y_2}{x_4} + {q}{y_2}
		= 0.
	\end{eqnarray*}
	
	Similarly, we have
	\begin{eqnarray*}
		&&	\sum\limits_{r = 0}^2 {{{( - 1)}^r}{{\left( {{q^2}} \right)}^{\frac{{r(r - 1)}}{2} - r}}{{\left[ {\begin{array}{*{20}{c}}
							2\\
							r
					\end{array}} \right]}_{{q^2}}}y_1^{2 - r}{y_2}y_1^r} \\
		&	= &y_1^2{y_2} - {q^{ - 2}}{\left[ {\begin{array}{*{20}{c}}
					2\\
					1
			\end{array}} \right]_{{q^2}}}y_1^{}{y_2}y_1^{} + {q^{ - 2}}{y_2}y_1^2\\
		&	=& {y_2}y_1^2 - A{y_1}{x_2}{x_4} - A{x_2}{x_4}{y_1} - {q^{ - 2}}{\left[ {\begin{array}{*{20}{c}}
					2\\
					1
			\end{array}} \right]_{{q^2}}}{y_2}y_1^2 + {q^{ - 2}}{\left[ {\begin{array}{*{20}{c}}
					2\\
					1
			\end{array}} \right]_{{q^2}}}A{x_2}{x_4}{y_1} + {q^{ - 2}}{y_2}y_1^2\\
		&	=& \left( {1 - {q^{ - 2}}{{\left[ {\begin{array}{*{20}{c}}
							2\\
							1
					\end{array}} \right]}_{{q^2}}} + {q^{ - 2}}} \right){y_2}y_1^2 - A{y_1}{x_2}{x_4} - A{x_2}{x_4}{y_1} + {q^{ - 2}}{\left[ {\begin{array}{*{20}{c}}
					2\\
					1
			\end{array}} \right]_{{q^2}}}A{x_2}{x_4}{y_1}\\
		&	=& \left( {1 - {q^{ - 2}}\left( {1 + {q^2}} \right) + {q^{ - 2}}} \right){y_2}y_1^2  - A{y_1}{x_2}{x_4} - A{q^2}{y_1}{x_2}{x_4} + \left( {1 + {q^2}} \right)A{y_1}{x_2}{x_4}
		=0.
	\end{eqnarray*}	
\end{example}
\begin{example}\label{exam3}
	Given an initial quantum seed $\big(\widetilde{\mathbf{x}}, \Lambda, \widetilde{B}\big)$, where ${\widetilde{\mathbf{x}}}=\{x_1,x_2,\cdots ,x_{6}\}$, $$\Lambda  = \left( {\begin{array}{*{20}{c}}
			0&0&0&{ - 1}&0&0\\
			0&0&0&0&{ - 1}&0\\
			0&0&0&0&0&{ - 1}\\
			1&0&0&0&{ - 2}&2\\
			0&1&0&2&0&{ - 2}\\
			0&0&1&{ - 2}&2&0
	\end{array}} \right) \text{ and }{\widetilde{B}}=\left( {\begin{array}{*{20}{c}}
			0&2&{ - 2}\\
			{ - 2}&0&2\\
			2&{ - 2}&0\\
			1&0&0\\
			0&1&0\\
			0&0&1
	\end{array}} \right).$$
By mutating the seed $(\widetilde{\mathbf{x}},\Lambda,\widetilde{B})$ in the direction $1$, we obtain the seed $(\widetilde{\mathbf{x}}_1,\Lambda_1,\widetilde{B}_1)$, where $\widetilde{\mathbf{x}}_1=\{y_1,x_2,\cdots,x_{6}\}$, $\Lambda_1  =\left( {\begin{array}{*{20}{c}}
			0&0&0&1&{ - 2}&0\\
			0&0&0&0&{ - 1}&0\\
			0&0&0&0&0&{ - 1}\\
			{ - 1}&0&0&0&{ - 2}&2\\
			2&1&0&2&0&{ - 2}\\
			0&0&1&{ - 2}&2&0
	\end{array}} \right)$,  ${\widetilde{B}_1}=\left( {\begin{array}{*{20}{c}}
			0&{ - 2}&2\\
			2&0&{ - 2}\\
			{ - 2}&2&0\\
			{ - 1}&2&0\\
			0&1&0\\
			0&0&1
	\end{array}} \right)$\\
with the  exchange relation ${x_1}{y_1} = {q^{ - \frac{1}{2}}}x_3^2{x_4} + x_2^2.$
	
	By mutating the seed $(\widetilde{\mathbf{x}},\Lambda,\widetilde{B})$ in the direction $2$, we obtain the seed $(\widetilde{\mathbf{x}}_2,\Lambda_2,\widetilde{B}_2)$, where $\widetilde{\mathbf{x}}_2=\{x_1,y_2,x_3,\cdots,x_{6}\}$, $$\Lambda_2  = \left( {\begin{array}{*{20}{c}}
			0&0&0&{ - 1}&0&0\\
			0&0&0&0&1&{ - 2}\\
			0&0&0&0&0&{ - 1}\\
			1&0&0&0&{ - 2}&2\\
			0&{ - 1}&0&2&0&{ - 2}\\
			0&2&1&{ - 2}&2&0
	\end{array}} \right) \text{ and } {\widetilde{B}_2}=\left( {\begin{array}{*{20}{c}}
			0&{ - 2}&2\\
			2&0&{ - 2}\\
			{ - 2}&2&0\\
			1&0&0\\
			0&{ - 1}&2\\
			0&0&1
	\end{array}} \right)$$ with the exchange relation ${x_2}{y_2} = {q^{ - \frac{1}{2}}}x_1^2{x_5} + x_3^2.$
By mutating the seed $(\widetilde{\mathbf{x}},\Lambda,\widetilde{B})$ in the direction $3$, we obtain the seed $(\widetilde{\mathbf{x}}_3,\Lambda_3,\widetilde{B}_3)$, where $\widetilde{\mathbf{x}}_3=\{x_1,x_2,y_3,x_4,x_5,x_{6}\}$,

$$\Lambda_3  = \left( {\begin{array}{*{20}{c}}
			0&0&0&{ - 1}&0&0\\
			0&0&0&0&{ - 1}&0\\
			0&0&0&{ - 2}&0&1\\
			1&0&2&0&{ - 2}&2\\
			0&1&0&2&0&{ - 2}\\
			0&0&{ - 1}&{ - 2}&2&0
	\end{array}} \right) \text{ and } {\widetilde{B}_3}=\left( {\begin{array}{*{20}{c}}
			0&{ - 2}&2\\
			2&0&{ - 2}\\
			{ - 2}&2&0\\
			1&0&0\\
			0&1&0\\
			2&0&{ - 1}
	\end{array}} \right)$$ with the exchange relation ${x_3}{y_3} = {q^{ - \frac{1}{2}}}x_2^2{x_6} + x_1^2.$

	We claim that
	\begin{eqnarray*}
		\sum\limits_{r = 0}^3 {{{( - 1)}^r}{q^{\frac{{r(r - 1)}}{2}}}{{\left[ {\begin{array}{*{20}{c}}
							3\\
							r
					\end{array}} \right]}_q}y_2^{3 - r}{y_1}y_2^r}  = 0;\quad \sum\limits_{r = 0}^3 {{{( - 1)}^r}{q^{\frac{{r(r - 1)}}{2} - 2r}}{{\left[ {\begin{array}{*{20}{c}}
							3\\
							r
					\end{array}} \right]}_q}y_1^{3 - r}{y_2}y_1^r}  = 0;\\
		\sum\limits_{r = 0}^3 {{{( - 1)}^r}{q^{\frac{{r(r - 1)}}{2}}}{{\left[ {\begin{array}{*{20}{c}}
							3\\
							r
					\end{array}} \right]}_q}y_1^{3 - r}{y_3}y_1^r}  = 0;\quad \sum\limits_{r = 0}^3 {{{( - 1)}^r}{q^{\frac{{r(r - 1)}}{2} - 2r}}{{\left[ {\begin{array}{*{20}{c}}
							3\\
							r
					\end{array}} \right]}_q}y_3^{3 - r}{y_1}y_3^r}  = 0;\\
		\sum\limits_{r = 0}^3 {{{( - 1)}^r}{q^{\frac{{r(r - 1)}}{2}}}{{\left[ {\begin{array}{*{20}{c}}
							3\\
							r
					\end{array}} \right]}_q}y_3^{3 - r}{y_2}y_3^r}  = 0;\quad \sum\limits_{r = 0}^3 {{{( - 1)}^r}{q^{\frac{{r(r - 1)}}{2} - 2r}}{{\left[ {\begin{array}{*{20}{c}}
							3\\
							r
					\end{array}} \right]}_q}y_2^{3 - r}{y_3}y_2^r}  = 0.
	\end{eqnarray*}	
It is easy to check
	\begin{eqnarray*}
		{y_1}{y_2} &=& qx_1^{}x_2^{ - 1}x_3^2{x_4}{x_5} + {q^{ - \frac{1}{2}}}x_1^{}x_2^{}{x_5} + {q^{ - \frac{1}{2}}}x_1^{ - 1}x_2^{ - 1}x_3^4{x_4} + x_1^{ - 1}x_2^{}x_3^2\\
		{y_2}{y_1} &=& {q^{}}x_1^{}x_2^{ - 1}x_3^2{x_4}{x_5} + {q^{ - \frac{1}{2}}}x_1^{ - 1}x_2^{ - 1}x_3^4{x_4} + {q^{\frac{3}{2}}}x_1^{}x_2^{}{x_5} + x_1^{ - 1}x_2^{}x_3^2.
	\end{eqnarray*}	
	It follows that	
	\begin{equation}
		{y_2}{y_1} = {y_1}{y_2} + \left( {{q^{\frac{3}{2}}} - {q^{ - \frac{1}{2}}}} \right)x_1^{}x_2^{}{x_5},
	\end{equation}
	where $A={{q^{\frac{3}{2}}} - {q^{ - \frac{1}{2}}}}$.

	Then, we have
	\begin{eqnarray*}
		&&\sum\limits_{r = 0}^3 {{{( - 1)}^r}{q^{\frac{{r(r - 1)}}{2}}}{{\left[ {\begin{array}{*{20}{c}}
							3\\
							r
					\end{array}} \right]}_q}y_2^{3 - r}{y_1}y_2^r}  = y_2^3{y_1} - {\left[ {\begin{array}{*{20}{c}}
					3\\
					1
			\end{array}} \right]_q}y_2^2{y_1}y_2^{} + q{\left[ {\begin{array}{*{20}{c}}
					3\\
					2
			\end{array}} \right]_q}y_2^{}{y_1}y_2^2 - {q^3}{y_1}y_2^3\\
		&	=& {y_1}y_2^3 + Ay_2^2x_1^{}x_2^{}{x_5} + A{y_2}x_1^{}x_2^{}{x_5}{y_2} + Ax_1^{}x_2^{}{x_5}y_2^2\\
		&&	- {\left[ {\begin{array}{*{20}{c}}
					3\\
					1
			\end{array}} \right]_q}{y_1}y_2^3 - {\left[ {\begin{array}{*{20}{c}}
					3\\
					1
			\end{array}} \right]_q}A{y_2}x_1^{}x_2^{}{x_5}{y_2} - {\left[ {\begin{array}{*{20}{c}}
					3\\
					1
			\end{array}} \right]_q}Ax_1^{}x_2^{}{x_5}y_2^2\\
		&&	+ q{\left[ {\begin{array}{*{20}{c}}
					3\\
					2
			\end{array}} \right]_q}{y_1}y_2^3 + q{\left[ {\begin{array}{*{20}{c}}
					3\\
					2
			\end{array}} \right]_q}Ax_1^{}x_2^{}{x_5}y_2^2 - {q^3}{y_1}y_2^3\\
		&	=& \left( {1 - {{\left[ {\begin{array}{*{20}{c}}
							3\\
							1
					\end{array}} \right]}_q}+q{{\left[ {\begin{array}{*{20}{c}}
							3\\
							2
					\end{array}} \right]}_q} - {q^3}} \right){y_1}y_2^3 + Ay_2^2x_1^{}x_2^{}{x_5}\\
		&&	+ A\left( {1 - {{\left[ {\begin{array}{*{20}{c}}
							3\\
							1
					\end{array}} \right]}_q}} \right){y_2}x_1^{}x_2^{}{x_5}{y_2} + A\left( {1 - {{\left[ {\begin{array}{*{20}{c}}
							3\\
							1
					\end{array}} \right]}_q} + q{{\left[ {\begin{array}{*{20}{c}}
							3\\
							2
					\end{array}} \right]}_q}} \right)x_1^{}x_2^{}{x_5}y_2^2\\
				&	=& \left( {1 - 1 - q - {q^2} + q + {q^2} + {q^3} - {q^3}} \right){y_1}y_2^3 + Ay_2^2x_1^{}x_2^{}{x_5}\\
							\end{eqnarray*}		
				\begin{eqnarray*}
				&&	+ A\left( {1 - 1 - q - {q^2}} \right){y_2}x_1^{}x_2^{}{x_5}{y_2} + A\left( {1 - 1 - q - {q^2} + q + {q^2} + {q^3}} \right)x_1^{}x_2^{}{x_5}y_2^2\\
				&	=& Ay_2^2x_1^{}x_2^{}{x_5} - A\left( {q + {q^2}} \right){y_2}x_1^{}x_2^{}{x_5}{y_2} + A{q^3}x_1^{}x_2^{}{x_5}y_2^2\\
				&	=& A{x_1}\left( {y_2^2x_2^{} - \left( {1 + q} \right){y_2}x_2^{}{y_2} + qx_2^{}y_2^2} \right){x_5}.
	\end{eqnarray*}
	
	Hence, the claim can be deduced  from the following direct computation
	\begin{eqnarray*}
		&&y_2^2x_2^{} - \left( {1 + q} \right){y_2}x_2^{}{y_2} + qx_2^{}y_2^2\\
	&	=& {y_2}\left( {{q^{\frac{1}{2}}}x_1^2{x_5} + x_3^2} \right) - \left( {1 + q} \right){y_2}\left( {{q^{ - \frac{1}{2}}}x_1^2{x_5} + x_3^2} \right) + q\left( {{q^{ - \frac{1}{2}}}x_1^2{x_5} + x_3^2} \right){y_2}\\
	&	=& {q^{\frac{1}{2}}}{y_2}x_1^2{x_5} + {y_2}x_3^2 - {q^{ - \frac{1}{2}}}\left( {1 + q} \right){y_2}x_1^2{x_5} - \left( {1 + q} \right){y_2}x_3^2 + {q^{\frac{1}{2}}}x_1^2{x_5}{y_2} + qx_3^2{y_2}\\
	&	=& {q^{\frac{1}{2}}}{y_2}x_1^2{x_5} + {y_2}x_3^2 - \left( {{q^{ - \frac{1}{2}}} + {q^{\frac{1}{2}}}} \right){y_2}x_1^2{x_5} - \left( {1 + q} \right){y_2}x_3^2 + {q^{ - \frac{1}{2}}}x_1^2{y_2}{x_5} + q{y_2}x_3^2 = 0.
	\end{eqnarray*}

	Similarly, we have
	\begin{eqnarray*}
		&&	\sum\limits_{r = 0}^3 {{{( - 1)}^r}{q^{\frac{{r(r - 1)}}{2} - 2r}}{{\left[ {\begin{array}{*{20}{c}}
							3\\
							r
					\end{array}} \right]}_q}y_1^{3 - r}{y_2}y_1^r} \\
		&	=& y_1^3{y_2} - {q^{ - 2}}{\left[ {\begin{array}{*{20}{c}}
					3\\
					1
			\end{array}} \right]_q}y_1^2{y_2}{y_1} + {q^{ - 3}}{\left[ {\begin{array}{*{20}{c}}
					3\\
					2
			\end{array}} \right]_q}{y_1}{y_2}y_1^2 - {q^{ - 3}}{y_2}y_1^3\\
		&	=& {y_2}y_1^3 - Ay_1^2x_1^{}x_2^{}{x_5} - A{y_1}x_1^{}x_2^{}{x_5}{y_1} - Ax_1^{}x_2^{}{x_5}y_1^2\\
		&&	- {q^{ - 2}}{\left[ {\begin{array}{*{20}{c}}
					3\\
					1
			\end{array}} \right]_q}{y_2}y_1^3 + {q^{ - 2}}{\left[ {\begin{array}{*{20}{c}}
					3\\
					1
			\end{array}} \right]_q}A{y_1}x_1^{}x_2^{}{x_5}{y_1} + {q^{ - 2}}{\left[ {\begin{array}{*{20}{c}}
					3\\
					1
			\end{array}} \right]_q}Ax_1^{}x_2^{}{x_5}y_1^2\\
		&&	+{q^{ - 3}}{\left[ {\begin{array}{*{20}{c}}
					3\\
					2
			\end{array}} \right]_q}{y_2}y_1^3 - {q^{ - 3}}{\left[ {\begin{array}{*{20}{c}}
					3\\
					2
			\end{array}} \right]_q}Ax_1^{}x_2^{}{x_5}y_1^2 - {q^{ - 3}}{y_1}y_1^3\\
		&	=& \left( {1 - {q^{ - 2}}{{\left[ {\begin{array}{*{20}{c}}
							3\\
							1
					\end{array}} \right]}_q} + {q^{ - 3}}{{\left[ {\begin{array}{*{20}{c}}
							3\\
							2
					\end{array}} \right]}_q} - {q^{ - 3}}} \right){y_2}y_1^3 - Ay_1^2x_1^{}x_2^{}{x_5}\\
		&&	- \left( {1 - {q^{ - 2}}{{\left[ {\begin{array}{*{20}{c}}
							3\\
							1
					\end{array}} \right]}_q}} \right)A{y_1}x_1^{}x_2^{}{x_5}{y_1} - \left( {1 - {q^{ - 2}}{{\left[ {\begin{array}{*{20}{c}}
							3\\
							1
					\end{array}} \right]}_q} + {q^{ - 3}}{{\left[ {\begin{array}{*{20}{c}}
							3\\
							2
					\end{array}} \right]}_q}} \right)Ax_1^{}x_2^{}{x_5}y_1^2\\
		&=& \left( {1 - {q^{ - 2}} - {q^{ - 1}} - 1 + {q^{ - 3}} + {q^{ - 2}} + {q^{ - 1}} - {q^{ - 3}}} \right){y_2}y_1^3 - Ay_1^2x_1^{}x_2^{}{x_5}\\
	&&	- \left( {1 - {q^{ - 2}} - {q^{ - 1}} - 1} \right)A{y_1}x_1^{}x_2^{}{x_5}{y_1} \\
	&& - \left( {1 - {q^{ - 2}} - {q^{ - 1}} - 1 + {q^{ - 3}} + {q^{ - 2}} + {q^{ - 1}}} \right)Ax_1^{}x_2^{}{x_5}y_1^2\\
	&	= & - A\left( {y_1^2x_1^{} - {q^2}\left( {{q^{ - 2}} + {q^{ - 1}}} \right){y_1}x_1^{}{y_1} + {q^4}{q^{ - 3}}x_1^{}y_1^2} \right){x_2}{x_5}\\
	&	= & - A\left( {y_1^2x_1^{} - \left( {1 + q} \right){y_1}x_1^{}{y_1} + qx_1^{}y_1^2} \right){x_2}{x_5}.
	\end{eqnarray*}

	Hence, the claim can be deduced  from the following direct computation
	\begin{eqnarray*}
		&&	y_1^2x_1^{} - \left( {1 + q} \right){y_1}x_1^{}{y_1} + qx_1^{}y_1^2\\
		&=& y_1^{}\left( {{q^{\frac{1}{2}}}x_3^2{x_4} + x_2^2} \right) - \left( {1 + q} \right){y_1}\left( {{q^{ - \frac{1}{2}}}x_3^2{x_4} + x_2^2} \right) + q\left( {{q^{ - \frac{1}{2}}}x_3^2{x_4} + x_2^2} \right){y_1}\\
		&	=& {q^{\frac{1}{2}}}y_1^{}x_3^2{x_4} + y_1^{}x_2^2 - {q^{ - \frac{1}{2}}}\left( {1 + q} \right){y_1}x_3^2{x_4} - \left( {1 + q} \right){y_1}x_2^2 + {q^{\frac{1}{2}}}x_3^2{x_4}{y_1} + qx_2^2{y_1}\\
		&	=& {q^{\frac{1}{2}}}y_1^{}x_3^2{x_4} + y_1^{}x_2^2 - \left( {{q^{ - \frac{1}{2}}} + {q^{\frac{1}{2}}}} \right){y_1}x_3^2{x_4} - \left( {1 + q} \right){y_1}x_2^2 + {q^{ - \frac{1}{2}}}{y_1}x_3^2{x_4} + q{y_1}x_2^2\\
		&	=& 0.
	\end{eqnarray*}
	
	Similarly, we have	
	\begin{eqnarray*}
	\sum\limits_{r = 0}^3 {{{( - 1)}^r}{q^{\frac{{r(r - 1)}}{2}}}{{\left[ {\begin{array}{*{20}{c}}
						3\\
						r
				\end{array}} \right]}_q}y_1^{3 - r}{y_3}y_1^r}  = 0;\quad \sum\limits_{r = 0}^3 {{{( - 1)}^r}{q^{\frac{{r(r - 1)}}{2} - 2r}}{{\left[ {\begin{array}{*{20}{c}}
						3\\
						r
				\end{array}} \right]}_q}y_3^{3 - r}{y_1}y_3^r}  = 0,\\
	\sum\limits_{r = 0}^3 {{{( - 1)}^r}{q^{\frac{{r(r - 1)}}{2}}}{{\left[ {\begin{array}{*{20}{c}}
						3\\
						r
				\end{array}} \right]}_q}y_3^{3 - r}{y_2}y_3^r}  = 0;\quad \sum\limits_{r = 0}^3 {{{( - 1)}^r}{q^{\frac{{r(r - 1)}}{2} - 2r}}{{\left[ {\begin{array}{*{20}{c}}
						3\\
						r
				\end{array}} \right]}_q}y_2^{3 - r}{y_3}y_2^r}  = 0.
\end{eqnarray*}	
\end{example}

\section{Higher order fundamental relations }

In this section, we keep all the notations as in Sections 2 and 3. The following result can be  viewed as a generalization of Lemma \ref{lemma2}.
\begin{lemma}\label{Lemma1}
	For any  $i,j \in [1,n]$,
\begin{enumerate}
\item when  $b_{ij}<0$, $m\ge -lb_{ij} $, $-b_{ij}\ge l>0$ and $l-1\ge t\ge 0$, we have
	\begin{equation}\label{(lemma4.1)}
		\sum\limits_{s = 0}^{m} {{{\left( {{q^{{d_i}}}} \right)}^{ - s}}\left( {\sum\limits_{r = 0}^s {{D_r}} } \right)y_i^{ m - s}x_i^{ - {b_{ij}}(1 + t) - 1}y_i^s}  = 0,
	\end{equation}
	where ${D_r} = {( - 1)^r}{\left( {{q^{{d_i}}}} \right)^{\frac{{r(r - 1)}}{2}}}{\left[ {\begin{array}{*{20}{c}}
				{m + 1}\\
				r
		\end{array}} \right]_{{q^{{d_i}}}}};$
	
\item when $b_{ij}>0$, $m\ge lb_{ij} $, $b_{ij}\ge l>0$ and $l-1\ge t\ge 0$, we have	
	\begin{equation}\label{(lemma4.2)}
	\sum\limits_{s = 0}^m {{{\left( {{q^{{d_i}}}} \right)}^{s{b_{ij}}(1 + t)}}\left( {\sum\limits_{r = 0}^s {{F_r}} } \right)y_i^{m - s}x_i^{{b_{ij}}(1 + t) - 1}y_i^s} =0,
	\end{equation}
	where ${F_r} = {( - 1)^r}{\left( {{q^{{d_i}}}} \right)^{\frac{{r(r - 1)}}{2} - rm}}{\left[ {\begin{array}{*{20}{c}}
				{m + 1}\\
				r
		\end{array}} \right]_{{q^{{d_i}}}}}.$
\end{enumerate}
\end{lemma}
\begin{proof}
\begin{enumerate}
\item According to  (\ref{(ynxn)n}) and (\ref{(xnyn)n}), we obtain		
	\begin{eqnarray*}
&&	\sum\limits_{s = 0}^m {{{\left( {{q^{{d_i}}}} \right)}^{ - s}} \left({\sum\limits_{r = 0}^s {{D_r}} }\right)  y_i^{m - s}x_i^{ - {b_{ij}}(1 + t) - 1}y_i^s} \\
&	= &\sum\limits_{s = 0}^{ - {b_{ij}}(1 + t) - 1} {{{\left( {{q^{{d_i}}}} \right)}^{ - s}}  \left({\sum\limits_{r = 0}^s {{D_r}} }\right)   y_i^{m + {b_{ij}}(1 + t) + 1}y_i^{ - {b_{ij}}(1 + t) - 1 - s}x_i^{ - {b_{ij}}(1 + t) - 1 - s}x_i^sy_i^s} \\
	&&+ \sum\limits_{s =  - {b_{ij}}(1 + t)}^m {{{\left( {{q^{{d_i}}}} \right)}^{ - s}}  \left({\sum\limits_{r = 0}^s {{D_r}} }\right)   y_i^{m - s}x_i^{ - {b_{ij}}(1 + t) - 1}y_i^{ - {b_{ij}}(1 + t) - 1}y_i^{s + {b_{ij}}(1 + t) + 1}} \\	
&	=& \sum\limits_{s = 0}^{ - {b_{ij}}(1 + t) - 1} {\left( {{{\left( {{q^{{d_i}}}} \right)}^{ - s}}\left( {\sum\limits_{r = 0}^s {{D_r}} } \right)y_i^{m + {b_{ij}}(1 + t) + 1}\sum\limits_{k = 0}^{ - {b_{ij}}(1 + t) - 1 - s} {\left( {{{\left[ {\begin{array}{*{20}{c}}
									{ - {b_{ij}}(1 + t) - 1 - s}\\
									k
							\end{array}} \right]}_{{q^{{d_i}}}}}} \right.} } \right.} \\
	&&\left. {\left. { \cdot {{\left( {{q^{{d_i}}}} \right)}^{\frac{{{k^2}}}{2}}}{C_{ - {b_{ij}}(1 + t) - 1 - s,k}}x_{n + i}^k} \right)\sum\limits_{k = 0}^s {\left( {{{\left[ {\begin{array}{*{20}{c}}
								s\\
								k
						\end{array}} \right]}_{{q^{{d_i}}}}}{{\left( {{q^{{d_i}}}} \right)}^{\frac{{k(k - 2s)}}{2}}}{C_{s,k}}x_{n + i}^k} \right)} } \right)\\
								\end{eqnarray*}		
					\begin{eqnarray*}
&&	+ \sum\limits_{s =  - {b_{ij}}(1 + t)}^m {\left( {{{\left( {{q^{{d_i}}}} \right)}^{ - s}}\left( {\sum\limits_{r = 0}^s {{D_r}} } \right)y_i^{m - s}\sum\limits_{k = 0}^{ - {b_{ij}}(1 + t) - 1} {\left( {{{\left[ {\begin{array}{*{20}{c}}
									{ - {b_{ij}}(1 + t) - 1}\\
									k
							\end{array}} \right]}_{{q^{{d_i}}}}}} \right.} } \right.} \\
&&	\left. {\left. { \cdot {{\left( {{q^{{d_i}}}} \right)}^{\frac{{k(k - 2( - {b_{ij}}(1 + t) - 1))}}{2}}}{C_{ - {b_{ij}}(1 + t) - 1,k}}x_{n + i}^k} \right)y_i^{s + {b_{ij}}(1 + t) + 1}} \right)\\
&	=& \sum\limits_{s = 0}^{ - {b_{ij}}(1 + t) - 1} {\left( {{{\left( {{q^{{d_i}}}} \right)}^{ - s}}\left( {\sum\limits_{r = 0}^s {{D_r}} } \right)y_i^{m + {b_{ij}}(1 + t) + 1}\sum\limits_{k = 0}^{ - {b_{ij}}(1 + t) - 1 - s} {\left( {{{\left[ {\begin{array}{*{20}{c}}
									{ - {b_{ij}}(1 + t) - 1 - s}\\
									k
							\end{array}} \right]}_{{q^{{d_i}}}}}} \right.} } \right.} \\
&&	\left. {\left. { \cdot {{\left( {{q^{{d_i}}}} \right)}^{\frac{{{k^2}}}{2}}}{C_{ - {b_{ij}}(1 + t) - 1 - s,k}}x_{n + i}^k} \right)\sum\limits_{k = 0}^s {\left( {{{\left[ {\begin{array}{*{20}{c}}
								s\\
								k
						\end{array}} \right]}_{{q^{{d_i}}}}}{{\left( {{q^{{d_i}}}} \right)}^{\frac{{k(k - 2s)}}{2}}}{C_{s,k}}x_{n + i}^k} \right)} } \right)\\
	&&+ \sum\limits_{s =  - {b_{ij}}(1 + t)}^m {{{\left( {{q^{{d_i}}}} \right)}^{ - s}}\left( {\sum\limits_{r = 0}^s {{D_r}} } \right)y_i^{m + {b_{ij}}(1 + t) + 1}\sum\limits_{k = 0}^{ - {b_{ij}}(1 + t) - 1} {\left( {{{\left[ {\begin{array}{*{20}{c}}
								{ - {b_{ij}}(1 + t) - 1}\\
								k
						\end{array}} \right]}_{{q^{{d_i}}}}}} \right.} } \\
&&	\left. { \cdot {{\left( {{q^{{d_i}}}} \right)}^{\frac{{{k^2} - 2ks}}{2}}}{C_{ - {b_{ij}}(1 + t) - 1,k}}x_{n + i}^k} \right),
	\end{eqnarray*}
where ${C_{l,k}} = \prod\limits_{v \in [1,n]}^ \triangleleft  {x_v^{(l - k){{\left[ { - {b_{vi}}} \right]}_ + }}}  \cdot \prod\limits_{v \in [1,n]}^ \triangleleft  {x_v^{k{{\left[ {{b_{vi}}} \right]}_ + }}} .$
	
For any $k\in [0,{ - {b_{ij}}(1 + t) - 1}]$, we compute
	\begin{eqnarray*}
	&&	y_i^{m + {b_{ij}}(1 + t) + 1}\left( {{{\left[ {\begin{array}{*{20}{c}}
							{ - {b_{ij}}(1 + t) - 1}\\
							k
					\end{array}} \right]}_{{q^{{d_i}}}}}{{\left( {{q^{{d_i}}}} \right)}^{\frac{{{k^2}}}{2}}}{C_{ - {b_{ij}}(1 + t) - 1,k}}x_{n + i}^k} \right.{\rm{ }}\\
	&&	+ {\left( {{q^{{d_i}}}} \right)^{ - 1}}\left(\sum\limits_{r = 0}^1 {{D_r}} \right)   \left( {{{\left[ {\begin{array}{*{20}{c}}
							{ - {b_{ij}}(1 + t) - 2}\\
							k
					\end{array}} \right]}_{{q^{{d_i}}}}}{{\left( {{q^{{d_i}}}} \right)}^{\frac{{{k^2}}}{2}}}{C_{ - {b_{ij}}(1 + t) - 2,k}}x_{n + i}^k{C_{1,0}}} \right.\\
	&&	\left. { + {{\left[ {\begin{array}{*{20}{c}}
							{ - {b_{ij}}(1 + t) - 2}\\
							{k - 1}
					\end{array}} \right]}_{{q^{{d_i}}}}}{{\left( {{q^{{d_i}}}} \right)}^{\frac{{{{(k - 1)}^2} - 1}}{2}}}{C_{ - {b_{ij}}(1 + t) - 2,k - 1}}x_{n + i}^{k - 1}{C_{1,1}}x_{n + i}^{}} \right)\\
	&&	+  \cdots  + {\left( {{q^{{d_i}}}} \right)^{ - s}}\left(\sum\limits_{r = 0}^s {{D_r}}\right)    \left( {{{\left[ {\begin{array}{*{20}{c}}
							{ - {b_{ij}}(1 + t) - 1 - s}\\
							k
					\end{array}} \right]}_{{q^{{d_i}}}}}{{\left( {{q^{{d_i}}}} \right)}^{\frac{{{k^2}}}{2}}}{C_{ - {b_{ij}}(1 + t) - 1 - s,k}}x_{n + i}^k{C_{s,0}}} \right.\\
	&&	+ {\left[ {\begin{array}{*{20}{c}}
					{ - {b_{ij}}(1 + t) - 1 - s}\\
					{k - 1}
			\end{array}} \right]_{{q^{{d_i}}}}}{\left[ {\begin{array}{*{20}{c}}
					s\\
					1
			\end{array}} \right]_{{q^{{d_i}}}}}{\left( {{q^{{d_i}}}} \right)^{\frac{{{{(k - 1)}^2} + (1 - 2s)}}{2}}}{C_{ - {b_{ij}}(1 + t) - 1 - s,k - 1}}x_{n + i}^{k - 1}{C_{s,1}}x_{n + i}^{}\\
	&&	+  \cdots  + {\left[ {\begin{array}{*{20}{c}}
					{ - {b_{ij}}(1 + t) - 1 - s}\\
					{k - h}
			\end{array}} \right]_{{q^{{d_i}}}}}{\left[ {\begin{array}{*{20}{c}}
					s\\
					h
			\end{array}} \right]_{{q^{{d_i}}}}}{\left( {{q^{{d_i}}}} \right)^{\frac{{{{(k - h)}^2} + h(h - 2s)}}{2}}}{C_{ - {b_{ij}}(1 + t) - 1 - s,k - h}}x_{n + i}^{k - h}\\
	&&	\cdot \left. {{C_{s,h}}x_{n + i}^h +  \cdots  + {{\left[ {\begin{array}{*{20}{c}}
							s\\
							k
					\end{array}} \right]}_{{q^{{d_i}}}}}{{\left( {{q^{{d_i}}}} \right)}^{\frac{{k(k - 2s)}}{2}}}{C_{ - {b_{ij}}(1 + t) - 1 - s,0}}{C_{s,k}}x_{n + i}^k} \right)\\
	&&	+  \cdots  + {\left( {{q^{{d_i}}}} \right)^{{b_{ij}}(1 + t) + 1}}\left( {\sum\limits_{r = 0}^{ - {b_{ij}}(1 + t) - 1} {{D_r}} } \right){\left[ {\begin{array}{*{20}{c}}
					{ - {b_{ij}}(1 + t) - 1}\\
					k
			\end{array}} \right]_{{q^{{d_i}}}}}{\left( {{q^{{d_i}}}} \right)^{\frac{{k(k - 2( - {b_{ij}}(1 + t) - 1))}}{2}}}\\
					\end{eqnarray*}		
		\begin{eqnarray*}
	&&	\left. { \cdot {C_{ - {b_{ij}}(1 + t) - 1,k}}x_{n + i}^k} \right)\\
	&&	+ y_i^{m + {b_{ij}}(1 + t) + 1}\sum\limits_{s =  - {b_{ij}}(1 + t)}^m {\left( {{{\left( {{q^{{d_i}}}} \right)}^{ - s}}\left( {\sum\limits_{r = 0}^s {{D_r}} } \right){{\left[ {\begin{array}{*{20}{c}}
								{ - {b_{ij}}(1 + t) - 1}\\
								k
						\end{array}} \right]}_{{q^{{d_i}}}}}{{\left( {{q^{{d_i}}}} \right)}^{\frac{{{k^2} - 2ks}}{2}}}} \right.} \\
	&&	\left. { \cdot {C_{ - {b_{ij}}(1 + t) - 1,k}}x_{n + i}^k} \right) \\
	&	= &y_i^{m + {b_{ij}}(1 + t) + 1}{C_{ - {b_{ij}}(1 + t) - 1,k}}x_{n + i}^k{\left( {{q^{{d_i}}}} \right)^{\frac{{{k^2}}}{2}}}\left( {{{\left[ {\begin{array}{*{20}{c}}
							{ - {b_{ij}}(1 + t) - 1}\\
							k
					\end{array}} \right]}_{{q^{{d_i}}}}}} \right.\\
	&&	+ {\left( {{q^{{d_i}}}} \right)^{ - 1}}\left(\sum\limits_{r = 0}^1 {{D_r}}\right)    \left( {{{\left[ {\begin{array}{*{20}{c}}
							{ - {b_{ij}}(1 + t) - 2}\\
							k
					\end{array}} \right]}_{{q^{{d_i}}}}}} \right.\left. { + {{\left[ {\begin{array}{*{20}{c}}
							{ - {b_{ij}}(1 + t) - 2}\\
							{k - 1}
					\end{array}} \right]}_{{q^{{d_i}}}}}{{\left( {{q^{{d_i}}}} \right)}^{ - k}}} \right)\\
	&&	+  \cdots  + {\left( {{q^{{d_i}}}} \right)^{ - s}}\left(\sum\limits_{r = 0}^s {{D_r}}  \right)  \left( {{{\left[ {\begin{array}{*{20}{c}}
							{ - {b_{ij}}(1 + t) - 1 - s}\\
							k
					\end{array}} \right]}_{{q^{{d_i}}}}}} \right. \\
				&&+ {\left[ {\begin{array}{*{20}{c}}
					{ - {b_{ij}}(1 + t) - 1 - s}\\
					{k - 1}
			\end{array}} \right]_{{q^{{d_i}}}}}{\left[ {\begin{array}{*{20}{c}}
					s\\
					1
			\end{array}} \right]_{{q^{{d_i}}}}}{\left( {{q^{{d_i}}}} \right)^{ - k + 1 - s}}\\
	&&	\left. { +  \cdots  + {{\left[ {\begin{array}{*{20}{c}}
							{ - {b_{ij}}(1 + t) - 1 - s}\\
							{k - h}
					\end{array}} \right]}_{{q^{{d_i}}}}}{{\left[ {\begin{array}{*{20}{c}}
							s\\
							h
					\end{array}} \right]}_{{q^{{d_i}}}}}{{\left( {{q^{{d_i}}}} \right)}^{h( - k + h - s)}} +  \cdots  + {{\left[ {\begin{array}{*{20}{c}}
							s\\
							k
					\end{array}} \right]}_{{q^{{d_i}}}}}{{\left( {{q^{{d_i}}}} \right)}^{ - ks}}} \right)\\
	&&	+  \cdots  + {\left( {{q^{{d_i}}}} \right)^{{b_{ij}}(1 + t) + 1}}\left(\sum\limits_{r = 0}^{ - {b_{ij}}(1 + t) - 1} {{D_r}} \right)   {\left[ {\begin{array}{*{20}{c}}
					{ - {b_{ij}}(1 + t) - 1}\\
					k
			\end{array}} \right]_{{q^{{d_i}}}}}{\left( {{q^{{d_i}}}} \right)^{ - k( - {b_{ij}}(1 + t) - 1)}}\\
	&&	\left. { + \sum\limits_{s =  - {b_{ij}}(1 + t)}^m {{{\left( {{q^{{d_i}}}} \right)}^{ - s}}\left(\sum\limits_{r = 0}^s {{D_r}} \right)   {{{\left[ {\begin{array}{*{20}{c}}
									{ - {b_{ij}}(1 + t) - 1}\\
									k
							\end{array}} \right]}_{{q^{{d_i}}}}}{{\left( {{q^{{d_i}}}} \right)}^{ - ks}}} } } \right)\\
		&=& y_i^{m + {b_{ij}}(1 + t) + 1}{C_{ - {b_{ij}}(1 + t) - 1,k}}x_{n + i}^k{\left( {{q^{{d_i}}}} \right)^{\frac{{{k^2}}}{2}}}\left( {{{\left[ {\begin{array}{*{20}{c}}
							{ - {b_{ij}}(1 + t) - 1}\\
							k
					\end{array}} \right]}_{{q^{{d_i}}}}}} \right.\\
	&&	+ {\left( {{q^{{d_i}}}} \right)^{ - 1 - k}}\left(\sum\limits_{r = 0}^1 {{D_r}} \right) \left( {{{\left( {{q^{{d_i}}}} \right)}^k}{{\left[ {\begin{array}{*{20}{c}}
							{ - {b_{ij}}(1 + t) - 2}\\
							k
					\end{array}} \right]}_{{q^{{d_i}}}}}} \right.\left. { + {{\left[ {\begin{array}{*{20}{c}}
							{ - {b_{ij}}(1 + t) - 2}\\
							{k - 1}
					\end{array}} \right]}_{{q^{{d_i}}}}}} \right)\\
	&&	+  \cdots  + {\left( {{q^{{d_i}}}} \right)^{ - s - sk}}\left(\sum\limits_{r = 0}^s {{D_r}}\right)   \left( {{{\left( {{q^{{d_i}}}} \right)}^{sk}}{{\left[ {\begin{array}{*{20}{c}}
							{ - {b_{ij}}(1 + t) - 1 - s}\\
							k
					\end{array}} \right]}_{{q^{{d_i}}}}}} \right.\\
	&&	+ {\left[ {\begin{array}{*{20}{c}}
					{ - {b_{ij}}(1 + t) - 1 - s}\\
					{k - 1}
			\end{array}} \right]_{{q^{{d_i}}}}}{\left[ {\begin{array}{*{20}{c}}
					s\\
					1
			\end{array}} \right]_{{q^{{d_i}}}}}{\left( {{q^{{d_i}}}} \right)^{(k - 1)(s - 1)}}\\
	&&	\left. { +  \cdots  + {{\left[ {\begin{array}{*{20}{c}}
							{ - {b_{ij}}(1 + t) - 1 - s}\\
							{k - h}
					\end{array}} \right]}_{{q^{{d_i}}}}}{{\left[ {\begin{array}{*{20}{c}}
							s\\
							h
					\end{array}} \right]}_{{q^{{d_i}}}}}{{\left( {{q^{{d_i}}}} \right)}^{(k - h)(s - h)}} +  \cdots  + {{\left[ {\begin{array}{*{20}{c}}
							s\\
							k
					\end{array}} \right]}_{{q^{{d_i}}}}}} \right)\\
	&&	+  \cdots  + {\left( {{q^{{d_i}}}} \right)^{{b_{ij}}(1 + t) + 1 + k({b_{ij}}(1 + t) + 1)}}\left(\sum\limits_{r = 0}^{ - {b_{ij}}(1 + t) - 1} {{D_r}}\right) {\left[ {\begin{array}{*{20}{c}}
					{ - {b_{ij}}(1 + t) - 1}\\
					k
			\end{array}} \right]_{{q^{{d_i}}}}}\\
					\end{eqnarray*}		
		\begin{eqnarray*}
	&&	\left. { + \sum\limits_{s =  - {b_{ij}}(1 + t)}^m {{{\left( {{q^{{d_i}}}} \right)}^{ - s - ks}}\left(\sum\limits_{r = 0}^s {{D_r}} \right)  {{{\left[ {\begin{array}{*{20}{c}}
									{ - {b_{ij}}(1 + t) - 1}\\
									k
							\end{array}} \right]}_{{q^{{d_i}}}}}} } } \right)\\	
		&	\mathop  = \limits^{(\ref{(n-dk)})} &	y_i^{m + {b_{ij}}(1 + t) + 1}{C_{ - {b_{ij}}(1 + t) - 1,k}}x_{n + i}^k{\left( {{q^{{d_i}}}} \right)^{\frac{{{k^2}}}{2}}}\left( {{{\left[ {\begin{array}{*{20}{c}}
								{ - {b_{ij}}(1 + t) - 1}\\
								k
						\end{array}} \right]}_{{q^{{d_i}}}}}} \right.\\
		&&	+ {\left( {{q^{{d_i}}}} \right)^{ - 1 - k}}\left(\sum\limits_{r = 0}^1 {{D_r}}  \right)  {\left[ {\begin{array}{*{20}{c}}
						{ - {b_{ij}}(1 + t) - 1}\\
						k
				\end{array}} \right]_{{q^{{d_i}}}}} \\&&+  \cdots  + {\left( {{q^{{d_i}}}} \right)^{ - s - sk}}\sum\limits_{r = 0}^s {{D_r}}  \cdot {\left[ {\begin{array}{*{20}{c}}
						{ - {b_{ij}}(1 + t) - 1}\\
						k
				\end{array}} \right]_{{q^{{d_i}}}}}\\
		&&	+  \cdots  + {\left( {{q^{{d_i}}}} \right)^{{b_{ij}}(1 + t) + 1 + k({b_{ij}}(1 + t) + 1)}}\left(\sum\limits_{r = 0}^{ - {b_{ij}}(1 + t) - 1} {{D_r}}\right)   {\left[ {\begin{array}{*{20}{c}}
						{ - {b_{ij}}(1 + t) - 1}\\
						k
				\end{array}} \right]_{{q^{{d_i}}}}}\\
		&&	\left. { + \sum\limits_{s =  - {b_{ij}}(1 + t)}^m {{{\left( {{q^{{d_i}}}} \right)}^{ - s - ks}}\left(\sum\limits_{r = 0}^s {{D_r}} \right)  {{\left[ {\begin{array}{*{20}{c}}
									{ - {b_{ij}}(1 + t) - 1}\\
									k
							\end{array}} \right]}_{{q^{{d_i}}}}}} } \right)\\
		&	=& y_i^{m + {b_{ij}}(1 + t) + 1}{C_{ - {b_{ij}}(1 + t) - 1,k}}x_{n + i}^k{\left( {{q^{{d_i}}}} \right)^{\frac{{{k^2}}}{2}}}{\left[ {\begin{array}{*{20}{c}}
						{ - {b_{ij}}(1 + t) - 1}\\
						k
				\end{array}} \right]_{{q^{{d_i}}}}}\sum\limits_{s = 0}^m {{{\left( {{q^{{d_i}}}} \right)}^{ -s-sk }}\sum\limits_{r = 0}^s {{D_r}} }\\
			&	\mathop  = \limits^{(\ref{(lemma3)})}&0.
	\end{eqnarray*}

\item According to  (\ref{(ynxn)n}) and (\ref{(xnyn)n}), we obtain	
\begin{eqnarray*}
&&	\sum\limits_{s = 0}^m {{{\left( {{q^{{d_i}}}} \right)}^{s{b_{ij}}(1 + t)}}\left( {\sum\limits_{r = 0}^s {{F_r}} } \right)y_i^{m - s}x_i^{{b_{ij}}(1 + t) - 1}y_i^s} \\
&	=& \sum\limits_{s = 0}^{{b_{ij}}(1 + t) - 1} {{{\left( {{q^{{d_i}}}} \right)}^{s{b_{ij}}(1 + t)}}\left( {\sum\limits_{r = 0}^s {{F_r}} } \right)y_i^{m - {b_{ij}}(1 + t) + 1}y_i^{{b_{ij}}(1 + t) - 1 - s}x_i^{{b_{ij}}(1 + t) - 1 - s}x_i^sy_i^s} \\
&&	+ \sum\limits_{s = {b_{ij}}(1 + t)}^m {{{\left( {{q^{{d_i}}}} \right)}^{s{b_{ij}}(1 + t)}}\left( {\sum\limits_{r = 0}^s {{F_r}} } \right)y_i^{m - s}x_i^{{b_{ij}}(1 + t) - 1}y_i^{{b_{ij}}(1 + t) - 1}y_i^{s - {b_{ij}}(1 + t) + 1}} \\
&	=& \sum\limits_{s = 0}^{{b_{ij}}(1 + t) - 1} {\left( {{{\left( {{q^{{d_i}}}} \right)}^{s{b_{ij}}(1 + t)}}\left( {\sum\limits_{r = 0}^s {{F_r}} } \right)y_i^{m - {b_{ij}}(1 + t) + 1}\sum\limits_{k = 0}^{{b_{ij}}(1 + t) - 1 - s} {\left( {{{\left[ {\begin{array}{*{20}{c}}
									{{b_{ij}}(1 + t) - 1 - s}\\
									k
							\end{array}} \right]}_{{q^{{d_i}}}}}} \right.} } \right.} \\
&&	\left. {\left. { \cdot{{\left( {{q^{{d_i}}}} \right)}^{\frac{{{k^2}}}{2}}} {C_{{b_{ij}}(1 + t) - 1 - s,k}}x_{n + i}^k} \right)\sum\limits_{k = 0}^s {\left( {{{\left[ {\begin{array}{*{20}{c}}
								s\\
								k
						\end{array}} \right]}_{{q^{{d_i}}}}}{{\left( {{q^{{d_i}}}} \right)}^{\frac{{k(k - 2s)}}{2}}}{C_{s,k}}x_{n + i}^k} \right)} } \right)\\
&&	+ \sum\limits_{s = {b_{ij}}(1 + t)}^m {\left( {{{\left( {{q^{{d_i}}}} \right)}^{s{b_{ij}}(1 + t)}}\left( {\sum\limits_{r = 0}^s {{F_r}} } \right)y_i^{m - s}\sum\limits_{k = 0}^{{b_{ij}}(1 + t) - 1} {\left( {{{\left[ {\begin{array}{*{20}{c}}
									{{b_{ij}}(1 + t) - 1}\\
									k
							\end{array}} \right]}_{{q^{{d_i}}}}}} \right.} } \right.} \\
									\end{eqnarray*}		
						\begin{eqnarray*}
&&	\left. {\left. { \cdot{{\left( {{q^{{d_i}}}} \right)}^{\frac{{k(k - 2({b_{ij}}(1 + t) - 1))}}{2}}} {C_{{b_{ij}}(1 + t) - 1,k}}x_{n + i}^k} \right)y_i^{s - {b_{ij}}(1 + t) + 1}} \right)\\
&	=& \sum\limits_{s = 0}^{{b_{ij}}(1 + t) - 1} {\left( {{{\left( {{q^{{d_i}}}} \right)}^{s{b_{ij}}(1 + t)}}\left( {\sum\limits_{r = 0}^s {{F_r}} } \right)y_i^{m - {b_{ij}}(1 + t) + 1}\sum\limits_{k = 0}^{{b_{ij}}(1 + t) - 1 - s} {\left( {{{\left[ {\begin{array}{*{20}{c}}
									{{b_{ij}}(1 + t) - 1 - s}\\
									k
							\end{array}} \right]}_{{q^{{d_i}}}}}} \right.} } \right.} \\
&&	\left. {\left. { \cdot{{\left( {{q^{{d_i}}}} \right)}^{\frac{{{k^2}}}{2}}} {C_{{b_{ij}}(1 + t) - 1 - s,k}}x_{n + i}^k} \right)\sum\limits_{k = 0}^s {\left( {{{\left[ {\begin{array}{*{20}{c}}
								s\\
								k
													\end{array}} \right]}_{{q^{{d_i}}}}}{{\left( {{q^{{d_i}}}} \right)}^{\frac{{k(k - 2s)}}{2}}}{C_{s,k}}x_{n + i}^k} \right)} } \right)\\											
&&	+ \sum\limits_{s = {b_{ij}}(1 + t)}^m {{{\left( {{q^{{d_i}}}} \right)}^{s{b_{ij}}(1 + t)}}\left( {\sum\limits_{r = 0}^s {{F_r}} } \right)y_i^{m - {b_{ij}}(1 + t) + 1}\sum\limits_{k = 0}^{{b_{ij}}(1 + t) - 1} {\left( {{{\left[ {\begin{array}{*{20}{c}}
								{ {b_{ij}}(1 + t) - 1}\\
								k
						\end{array}} \right]}_{{q^{{d_i}}}}}} \right.} } \\
&&	\left. { \cdot {{\left( {{q^{{d_i}}}} \right)}^{\frac{{{k^2} - 2ks}}{2}}}{C_{{b_{ij}}(1 + t) - 1,k}}x_{n + i}^k} \right),
\end{eqnarray*}
where ${C_{l,k}} = \prod\limits_{v \in [1,n]}^ \triangleleft  {x_v^{(l - k){{\left[ { - {b_{vi}}} \right]}_ + }}}  \cdot \prod\limits_{v \in [1,n]}^ \triangleleft  {x_v^{k{{\left[ {{b_{vi}}} \right]}_ + }}} .$\\

For any $k\in [0,{  {b_{ij}}(1 + t) - 1}]$, we compute
\begin{eqnarray*}
	&&y_i^{m - {b_{ij}}(1 + t) + 1}\left( {{{\left[ {\begin{array}{*{20}{c}}
						{{b_{ij}}(1 + t) - 1}\\
						k
				\end{array}} \right]}_{{q^{{d_i}}}}}{{\left( {{q^{{d_i}}}} \right)}^{\frac{{{k^2}}}{2}}}{C_{{b_{ij}}(1 + t) - 1,k}}x_{n + i}^k} \right.\\
&&	+ {\left( {{q^{{d_i}}}} \right)^{{b_{ij}}(1 + t)}}\left( {\sum\limits_{r = 0}^1 {{F_r}} } \right)\left( {{{\left[ {\begin{array}{*{20}{c}}
						{{b_{ij}}(1 + t) - 2}\\
						k
				\end{array}} \right]}_{{q^{{d_i}}}}}{{\left( {{q^{{d_i}}}} \right)}^{\frac{{{k^2}}}{2}}}{C_{{b_{ij}}(1 + t) - 2,k}}x_{n + i}^k{C_{1,0}}} \right.\\
&&	\left. { + {{\left[ {\begin{array}{*{20}{c}}
						{{b_{ij}}(1 + t) - 2}\\
						{k - 1}
				\end{array}} \right]}_{{q^{{d_i}}}}}{{\left( {{q^{{d_i}}}} \right)}^{\frac{{{{(k - 1)}^2} - 1}}{2}}}{C_{{b_{ij}}(1 + t) - 2,k - 1}}x_{n + i}^{k - 1}{C_{1,1}}x_{n + i}^{}} \right)	+  \cdots\\
&&  + {\left( {{q^{{d_i}}}} \right)^{s{b_{ij}}(1 + t)}}\left( {\sum\limits_{r = 0}^s {{F_r}} } \right)\left( {{{\left[ {\begin{array}{*{20}{c}}
						{{b_{ij}}(1 + t) - 1 - s}\\
						k
				\end{array}} \right]}_{{q^{{d_i}}}}}{{\left[ {\begin{array}{*{20}{c}}
						s\\
						0
				\end{array}} \right]}_{{q^{{d_i}}}}}{{\left( {{q^{{d_i}}}} \right)}^{\frac{{{k^2}}}{2}}}{C_{{b_{ij}}(1 + t) - 1 - s,k}}x_{n + i}^k} \right.\\
	&&\cdot {C_{s,0}} + {\left[ {\begin{array}{*{20}{c}}
				{{b_{ij}}(1 + t) - 1 - s}\\
				{k - 1}
		\end{array}} \right]_{{q^{{d_i}}}}}{\left[ {\begin{array}{*{20}{c}}
				s\\
				1
		\end{array}} \right]_{{q^{{d_i}}}}}{\left( {{q^{{d_i}}}} \right)^{\frac{{{{(k - 1)}^2} + (1 - 2s)}}{2}}}{C_{{b_{ij}}(1 + t) - 1 - s,k - 1}}x_{n + i}^{k - 1}{C_{s,1}}x_{n + i}^{}\\
&&	+  \cdots  + {\left[ {\begin{array}{*{20}{c}}
				{{b_{ij}}(1 + t) - 1 - s}\\
				{k - h}
		\end{array}} \right]_{{q^{{d_i}}}}}{\left[ {\begin{array}{*{20}{c}}
				s\\
				h
		\end{array}} \right]_{{q^{{d_i}}}}}{\left( {{q^{{d_i}}}} \right)^{\frac{{{{(k - h)}^2} + h(h - 2s)}}{2}}}{C_{{b_{ij}}(1 + t) - 1 - s,k - h}}x_{n + i}^{k - h}{C_{s,h}}x_{n + i}^h\\
&&	\left. { +  \cdots  + {{\left[ {\begin{array}{*{20}{c}}
						{{b_{ij}}(1 + t) - 1 - s}\\
						0
				\end{array}} \right]}_{{q^{{d_i}}}}}{{\left[ {\begin{array}{*{20}{c}}
						s\\
						k
				\end{array}} \right]}_{{q^{{d_i}}}}}{{\left( {{q^{{d_i}}}} \right)}^{\frac{{k(k - 2s)}}{2}}}{C_{{b_{ij}}(1 + t) - 1 - s,0}}{C_{s,k}}x_{n + i}^k} \right)  \\
&&	+  \cdots  + {\left( {{q^{{d_i}}}} \right)^{({b_{ij}}(1 + t) - 1){b_{ij}}(1 + t)}}\left( {\sum\limits_{r = 0}^{{b_{ij}}(1 + t) - 1} {{F_r}} } \right){\left[ {\begin{array}{*{20}{c}}
				{{b_{ij}}(1 + t) - 1}\\
				k
		\end{array}} \right]_{{q^{{d_i}}}}}{\left( {{q^{{d_i}}}} \right)^{\frac{{k(k - 2({b_{ij}}(1 + t) - 1))}}{2}}}\\
&&	\cdot \left. {{C_{{b_{ij}}(1 + t) - 1,k}}x_{n + i}^k} \right)\\
			\end{eqnarray*}		
\begin{eqnarray*}
	&&	+ y_i^{m - {b_{ij}}(1 + t) + 1}\sum\limits_{s = {b_{ij}}(1 + t)}^m {\left( {{{\left( {{q^{{d_i}}}} \right)}^{s{b_{ij}}(1 + k)}}\left( {\sum\limits_{r = 0}^s {{F_r}} } \right){{\left[ {\begin{array}{*{20}{c}}
								{{b_{ij}}(1 + t) - 1}\\
								k
						\end{array}} \right]}_{{q^{{d_i}}}}}{{\left( {{q^{{d_i}}}} \right)}^{\frac{{{k^2} - 2ks}}{2}}}} \right.} \\
	&&	\left. { \cdot {C_{{b_{ij}}(1 + t) - 1,k}}x_{n + i}^k} \right)\\	
			&		= &y_i^{m - {b_{ij}}(1 + t) + 1}{C_{{b_{ij}}(1 + t) - 1,k}}x_{n + i}^k{\left( {{q^{{d_i}}}} \right)^{\frac{{{k^2}}}{2}}}\left( {{{\left[ {\begin{array}{*{20}{c}}
										{{b_{ij}}(1 + t) - 1}\\
										k
								\end{array}} \right]}_{{q^{{d_i}}}}}} \right.\\
				&&	+ {\left( {{q^{{d_i}}}} \right)^{{b_{ij}}(1 + t)}}\left( {\sum\limits_{r = 0}^1 {{F_r}} } \right)\left( {{{\left[ {\begin{array}{*{20}{c}}
										{{b_{ij}}(1 + t) - 2}\\
										k
								\end{array}} \right]}_{{q^{{d_i}}}}}} \right.\left. { + {{\left[ {\begin{array}{*{20}{c}}
										{{b_{ij}}(1 + t) - 2}\\
										{k - 1}
								\end{array}} \right]}_{{q^{{d_i}}}}}{{\left( {{q^{{d_i}}}} \right)}^{ - k}}} \right)\\
				&&	+  \cdots  + {\left( {{q^{{d_i}}}} \right)^{s{b_{ij}}(1 + t)}}\left( {\sum\limits_{r = 0}^s {{F_r}} } \right)\left( {{{\left[ {\begin{array}{*{20}{c}}
										{{b_{ij}}(1 + t) - 1 - s}\\
										k
								\end{array}} \right]}_{{q^{{d_i}}}}}} \right.\\
							&& + {\left[ {\begin{array}{*{20}{c}}
								{{b_{ij}}(1 + t) - 1 - s}\\
								{k - 1}
						\end{array}} \right]_{{q^{{d_i}}}}}{\left[ {\begin{array}{*{20}{c}}
								s\\
								1
						\end{array}} \right]_{{q^{{d_i}}}}}{\left( {{q^{{d_i}}}} \right)^{ - k + 1 - s}}\\
				&&	\left. { +  \cdots  + {{\left[ {\begin{array}{*{20}{c}}
										{{b_{ij}}(1 + t) - 1 - s}\\
										{k - h}
								\end{array}} \right]}_{{q^{{d_i}}}}}{{\left[ {\begin{array}{*{20}{c}}
										s\\
										h
								\end{array}} \right]}_{{q^{{d_i}}}}}{{\left( {{q^{{d_i}}}} \right)}^{h( - k + h - s)}} +  \cdots  + {{\left[ {\begin{array}{*{20}{c}}
										s\\
										k
								\end{array}} \right]}_{{q^{{d_i}}}}}{{\left( {{q^{{d_i}}}} \right)}^{ - ks}}} \right)\\
				&&	+  \cdots  + {\left( {{q^{{d_i}}}} \right)^{({b_{ij}}(1 + t) - 1){b_{ij}}(1 + t)}}\left( {\sum\limits_{r = 0}^{{b_{ij}}(1 + t) - 1} {{F_r}} } \right){\left[ {\begin{array}{*{20}{c}}
								{{b_{ij}}(1 + t) - 1}\\
								k
						\end{array}} \right]_{{q^{{d_i}}}}}{\left( {{q^{{d_i}}}} \right)^{ - k({b_{ij}}(1 + t) - 1)}}\\
				&&	\left. { + \sum\limits_{s = {b_{ij}}(1 + t)}^m {{{\left( {{q^{{d_i}}}} \right)}^{s{b_{ij}}(1 + t)}}\left( {\sum\limits_{r = 0}^s {{F_r}} } \right){{\left[ {\begin{array}{*{20}{c}}
											{{b_{ij}}(1 + t) - 1}\\
											k
									\end{array}} \right]}_{{q^{{d_i}}}}}{{\left( {{q^{{d_i}}}} \right)}^{ - ks}}} } \right)\\
&	=& y_i^{m - {b_{ij}}(1 + t) + 1}{C_{{b_{ij}}(1 + t) - 1,k}}x_{n + i}^k{\left( {{q^{{d_i}}}} \right)^{\frac{{{k^2}}}{2}}}\left( {{{\left[ {\begin{array}{*{20}{c}}
						{{b_{ij}}(1 + t) - 1}\\
						k
				\end{array}} \right]}_{{q^{{d_i}}}}}} \right.\\
&&	+ {\left( {{q^{{d_i}}}} \right)^{{b_{ij}}(1 + t) - k}}\left( {\sum\limits_{r = 0}^1 {{F_r}} } \right)\left( {{{\left( {{q^{{d_i}}}} \right)}^k}{{\left[ {\begin{array}{*{20}{c}}
						{{b_{ij}}(1 + t) - 2}\\
						k
				\end{array}} \right]}_{{q^{{d_i}}}}}} \right.\left. { + {{\left[ {\begin{array}{*{20}{c}}
						{{b_{ij}}(1 + t) - 2}\\
						{k - 1}
				\end{array}} \right]}_{{q^{{d_i}}}}}} \right)\\
&&	+  \cdots  + {\left( {{q^{{d_i}}}} \right)^{s{b_{ij}}(1 + t) - sk}}\left( {\sum\limits_{r = 0}^s {{F_r}} } \right)\left( {{{\left( {{q^{{d_i}}}} \right)}^{sk}}{{\left[ {\begin{array}{*{20}{c}}
						{{b_{ij}}(1 + t) - 1 - s}\\
						k
				\end{array}} \right]}_{{q^{{d_i}}}}}} \right.\\
&&	+ {\left[ {\begin{array}{*{20}{c}}
				{{b_{ij}}(1 + t) - 1 - s}\\
				{k - 1}
		\end{array}} \right]_{{q^{{d_i}}}}}{\left[ {\begin{array}{*{20}{c}}
				s\\
				1
		\end{array}} \right]_{{q^{{d_i}}}}}{\left( {{q^{{d_i}}}} \right)^{(k - 1)(s - 1)}}\\
&&	\left. { +  \cdots  + {{\left[ {\begin{array}{*{20}{c}}
						{{b_{ij}}(1 + t) - 1 - s}\\
						{k - h}
				\end{array}} \right]}_{{q^{{d_i}}}}}{{\left[ {\begin{array}{*{20}{c}}
						s\\
						h
				\end{array}} \right]}_{{q^{{d_i}}}}}{{\left( {{q^{{d_i}}}} \right)}^{(k - h)(s - h)}} +  \cdots  + {{\left[ {\begin{array}{*{20}{c}}
						s\\
						k
				\end{array}} \right]}_{{q^{{d_i}}}}}} \right)\\
&&	+  \cdots  + {\left( {{q^{{d_i}}}} \right)^{({b_{ij}}(1 + t) - 1){b_{ij}}(1 + t) - k({b_{ij}}(1 + t) - 1)}}\left( {\sum\limits_{r = 0}^{{b_{ij}}(1 + t) - 1} {{F_r}} } \right){\left[ {\begin{array}{*{20}{c}}
				{{b_{ij}}(1 + t) - 1}\\
				k
		\end{array}} \right]_{{q^{{d_i}}}}}\\
&&	\left. { + \sum\limits_{s = {b_{ij}}(1 + t)}^m {{{\left( {{q^{{d_i}}}} \right)}^{s{b_{ij}}(1 + t) - ks}}\left( {\sum\limits_{r = 0}^s {{F_r}} } \right){{\left[ {\begin{array}{*{20}{c}}
							{{b_{ij}}(1 + t) - 1}\\
							k
					\end{array}} \right]}_{{q^{{d_i}}}}}} } \right)\\
							\end{eqnarray*}		
				\begin{eqnarray*}
&	\mathop  = \limits^{(\ref{(n-dk)})}& y_i^{m - {b_{ij}}(1 + t) + 1}{C_{{b_{ij}}(1 + t) - 1,k}}x_{n + i}^k{\left( {{q^{{d_i}}}} \right)^{\frac{{{k^2}}}{2}}}\left( {{{\left[ {\begin{array}{*{20}{c}}
						{{b_{ij}}(1 + t) - 1}\\
						k
				\end{array}} \right]}_{{q^{{d_i}}}}}} \right.\\			
&&	+ {\left( {{q^{{d_i}}}} \right)^{{b_{ij}}(1 + t) - k}}\left( {\sum\limits_{r = 0}^1 {{F_r}} } \right){\left[ {\begin{array}{*{20}{c}}
				{{b_{ij}}(1 + t) - 1}\\
				k
		\end{array}} \right]_{{q^{{d_i}}}}} \\&&+  \cdots  + {\left( {{q^{{d_i}}}} \right)^{s{b_{ij}}(1 + t) - sk}}\left( {\sum\limits_{r = 0}^s {{F_r}} } \right){\left[ {\begin{array}{*{20}{c}}
				{{b_{ij}}(1 + t) - 1}\\
				k
		\end{array}} \right]_{{q^{{d_i}}}}}\\		
&&	+  \cdots  + {\left( {{q^{{d_i}}}} \right)^{({b_{ij}}(1 + t) - 1){b_{ij}}(1 + t) - k({b_{ij}}(1 + t) - 1)}}\left( {\sum\limits_{r = 0}^{{b_{ij}}(1 + t) - 1} {{F_r}} } \right){\left[ {\begin{array}{*{20}{c}}
				{{b_{ij}}(1 + t) - 1}\\
				k
		\end{array}} \right]_{{q^{{d_i}}}}}\\
&&	\left. { + \sum\limits_{s = {b_{ij}}(1 + t)}^m {{{\left( {{q^{{d_i}}}} \right)}^{s{b_{ij}}(1 + t) - ks}}\left( {\sum\limits_{r = 0}^s {{F_r}} } \right){{\left[ {\begin{array}{*{20}{c}}
							{{b_{ij}}(1 + t) - 1}\\
							k
					\end{array}} \right]}_{{q^{{d_i}}}}}} } \right)\\
&	=& y_i^{m - {b_{ij}}(1 + t) + 1}{C_{{b_{ij}}(1 + t) - 1,k}}x_{n + i}^k{\left( {{q^{{d_i}}}} \right)^{\frac{{{k^2}}}{2}}}{\left[ {\begin{array}{*{20}{c}}
				{{b_{ij}}(1 + t) - 1}\\
				k
		\end{array}} \right]_{{q^{{d_i}}}}}\sum\limits_{s = 0}^m {{{\left( {{q^{{d_i}}}} \right)}^{s{b_{ij}}(1 + t) - ks}}\sum\limits_{r = 0}^s {{F_r}} } \\
&	=  & y_i^{m - {b_{ij}}(1 + t) + 1}{C_{{b_{ij}}(1 + t) - 1,k}}x_{n + i}^k{\left( {{q^{{d_i}}}} \right)^{\frac{{{k^2}}}{2}}}{\left[ {\begin{array}{*{20}{c}}
				{{b_{ij}}(1 + t) - 1}\\
				k
		\end{array}} \right]_{{q^{{d_i}}}}}\sum\limits_{s = 0}^m {{{\left( {{q^{{d_i}}}} \right)}^{s\left( {{b_{ij}}(1 + t) - k} \right)}}\sum\limits_{r = 0}^s {{F_r}} } \\
	&	\mathop  = \limits^{(\ref{(lemma31)})}&0.
\end{eqnarray*}
\end{enumerate}
This completes the proof.
\end{proof}

The following result shows that quantum cluster algebra $\mathcal{A}_{q}$ satisfies higher order fundamental relations.
\begin{theorem}For any  $i\neq j \in [1,n]$,
\begin{enumerate}	
\item	when $b_{ij}= 0$ and $m \ge  0 $ or $b_{ij}< 0$, $-b_{ij}\ge l>0$ and $m \ge  - l{b_{ij}} $, we have
\begin{eqnarray*}
	\sum\limits_{r = 0}^{ m+1} {{{( - 1)}^r}{{\left( {{q^{{d_i}}}} \right)}^{\frac{{r(r - 1)}}{2}}}{{\left[ {\begin{array}{*{20}{c}}
						{ m + 1}\\
						r
				\end{array}} \right]}_{{q^{{d_i}}}}}y_i^{ m+ 1 - r}y_j^ly_i^r}  &=& 0;\\
\end{eqnarray*}
	
\item when $b_{ij}> 0$, $b_{ij}\ge l>0$ and $m \ge   l{b_{ij}} $, we have
\begin{eqnarray*}
	\sum\limits_{r = 0}^{ m+1} {{{( - 1)}^r}{{\left( {{q^{{d_i}}}} \right)}^{\frac{{r(r - 1)}}{2}-rm}}{{\left[ {\begin{array}{*{20}{c}}
						{ m + 1}\\
						r
				\end{array}} \right]}_{{q^{{d_i}}}}}y_i^{ m+ 1 - r}y_j^ly_i^r}  &=& 0.\\
\end{eqnarray*}
\end{enumerate}
\end{theorem}

\begin{proof}
\begin{enumerate}	
\item By using the exchange relations, we have
\begin{eqnarray*}	
{y_i}& =& {q^{ - \frac{{{d_i}}}{2}}}x_i^{ - 1} {\prod\limits_{k \in [1,n]}^ \triangleleft  {x_k^{{{\left[ {{b_{ki}}} \right]}_ + }}} }\cdot {x_{n + i}} +   {\prod\limits_{k \in [1,n]}^ \triangleleft  {x_k^{{{\left[ { - {b_{ki}}} \right]}_ + }}} } \cdot x_i^{ - 1},\\
	{y_j} &= &{q^{ - \frac{{{d_j}}}{2}}}x_j^{ - 1}  {\prod\limits_{k \in [1,n]}^ \triangleleft  {x_k^{{{\left[ {{b_{kj}}} \right]}_ + }}} } \cdot{x_{n + j}} +   {\prod\limits_{k \in [1,n]}^ \triangleleft  {x_k^{{{\left[ { - {b_{kj}}} \right]}_ + }}} }\cdot x_j^{ - 1}.
\end{eqnarray*}	

	According to  (\ref{(y1y2)n}), we have
	\begin{eqnarray}\label{(yiyj)}
	{y_i}{y_j} = {y_j}{y_i} + Ax_i^{-{b_{ij}} - 1}x_j^{{b_{ji}} - 1}B{x_{n + i}} ,
\end{eqnarray}
where $A = {q^{ - \frac{{{d_i}}}{2} - {d_i}{b_{ij}}}} - {q^{ - \frac{{{d_i}}}{2}}}$  and $B =   {\prod\limits_{k \in [1,n],k \ne j}^ \triangleleft  {x_k^{{{\left[ { - {b_{ki}}} \right]}_ + }}} } \cdot {\prod\limits_{k \in [1,n],k \ne i}^ \triangleleft  {x_k^{{{\left[ {{b_{kj}}} \right]}_ + }}} }  $.	

(i) When $b_{ij}=0$, we have
\[\sum\limits_{r = 0}^1 {{{( - 1)}^r}{{\left( {{q^{{d_i}}}} \right)}^{\frac{{r(r - 1)}}{2}}}{{\left[ {\begin{array}{*{20}{c}}
					1\\
					r
			\end{array}} \right]}_{{q^{{d_i}}}}}y_i^{1 - r}y_j^ly_i^r}  =  y_i^{}y_j^l - y_i^ly_j^{} = 0.\]

(ii) When $b_{ij}\ne 0$, note that,	
for any $s\geq 1$, we have
	\begin{eqnarray}\label{(y1y2)n1}
		{y_i}y_j^s&	\mathop  = \limits^{(\ref{(yiyj)})} &\left( {{y_j}{y_i} + Ax_i^{-{b_{ij}} - 1}x_j^{{b_{ji}} - 1}B{x_{n + i}}} \right)y_j^{s - 1}\nonumber\\
		&	=& {y_j}{y_i}y_j^{s - 1} + Ax_i^{-{b_{ij}} - 1}x_j^{{b_{ji}} - 1}B{x_{n + i}}y_j^{s - 1}\nonumber\\
		&	\mathop  = \limits^{(\ref{(yiyj)})} & {y_j}\left( {{y_j}{y_i} + Ax_i^{-{b_{ij}} - 1}x_j^{{b_{ji}} - 1}B{x_{n + i}}} \right)y_j^{s - 2} + Ax_i^{-{b_{ij}} - 1}x_j^{{b_{ji}} - 1}B{x_{n + i}}y_j^{s - 1}\nonumber\\
		&	=& y_j^2{y_i}y_j^{s - 2} + A{y_j}x_i^{-{b_{ij}} - 1}x_j^{{b_{ji}} - 1}B{x_{n + i}}y_j^{s - 2} + Ax_i^{-{b_{ij}} - 1}x_j^{{b_{ji}} - 1}B{x_{n + i}}y_j^{s - 1}\nonumber\\
	&	\mathop  = \limits^{(\ref{(yiyj)})} &y_j^s{y_i} + \sum\limits_{t = 1}^s {Ay_j^{t-1}x_i^{-{b_{ij}}- 1}x_j^{{b_{ji}} - 1}B{x_{n + i}}y_j^{s- t}}
	\end{eqnarray}
and for any $s\geq 1$, $k\geq 1$, we have
	\begin{eqnarray}\label{(y1y2)n2}
	y_i^ky_j^s& 		\mathop  = \limits^{(\ref{(y1y2)n1})} &y_i^{k - 1}\left( {y_j^s{y_i} + \sum\limits_{t = 1}^s {Ay_j^{t - 1}x_i^{ - {b_{ij}} - 1}x_j^{{b_{ji}} - 1}B{x_{n + i}}y_j^{s - t}} } \right)\nonumber\\
&	=& y_i^{k - 1}y_j^s{y_i} + \sum\limits_{t = 1}^s {Ay_i^{k - 1}y_j^{t - 1}x_i^{ - {b_{ij}} - 1}x_j^{{b_{ji}} - 1}B{x_{n + i}}y_j^{s - t}} \nonumber\\
& 		\mathop  = \limits^{(\ref{(y1y2)n1})} & y_i^{k - 2}y_j^sy_i^2 + \sum\limits_{t = 1}^s {Ay_i^{k - 1}y_j^{t - 1}x_i^{ - {b_{ij}} - 1}x_j^{{b_{ji}} - 1}B{x_{n + i}}y_j^{s - t}}\nonumber\\
&&  + \sum\limits_{t = 1}^s {Ay_i^{k - 2}y_j^{t - 1}x_i^{ - {b_{ij}} - 1}x_j^{{b_{ji}} - 1}B{x_{n + i}}y_j^{s - t}{y_i}}\nonumber \\
& 		\mathop  = \limits^{(\ref{(y1y2)n1})} &y_j^sy_i^k + \sum\limits_{p = 1}^k {\sum\limits_{t = 1}^s {Ay_i^{k - p}y_j^{t - 1}x_i^{ - {b_{ij}} - 1}x_j^{{b_{ji}} - 1}B{x_{n + i}}y_j^{s - t}y_i^{p - 1}} }.
	\end{eqnarray}	
	
We calculate
	\begin{eqnarray*}
	&&	\sum\limits_{r = 0}^{ m + 1} {{{( - 1)}^r}{{\left( {{q^{{d_i}}}} \right)}^{\frac{{r(r - 1)}}{2} }}{{\left[ {\begin{array}{*{20}{c}}
							{ m + 1}\\
							r
					\end{array}} \right]}_{{q^{{d_i}}}}}y_i^{ m + 1 - r}y_j^ly_i^r}  = \sum\limits_{r = 0}^{ m + 1} {{D_r}y_i^{ m+ 1 - r}y_j^ly_i^r} \\
	&	=& y_i^{ m + 1}y_j^l + {D_1}y_i^{ m}y_j^l{y_i} + {D_2}y_i^{ m - 1}y_j^ly_i^2 +  \cdots  + {D_{  m}}y_i^{}y_j^ly_i^{ m} + {D_{m + 1}}y_j^ly_i^{ m + 1}\\
	&	\mathop  = \limits^{(\ref{(y1y2)n2})}& y_j^ly_i^{ m + 1} + \sum\limits_{p = 1}^{ m + 1} {\sum\limits_{t = 1}^l {{A}y_i^{ m + 1 - p}y_j^{t - 1}x_i^{ - {b_{ij}} - 1}x_j^{{b_{ji}} - 1}B{x_{n + i}}y_j^{l- t}y_i^{p - 1}} } \\
				\end{eqnarray*}		
	\begin{eqnarray*}
	&&	+ {D_1}y_j^ly_i^{ m + 1} + {D_1}\sum\limits_{p = 1}^{ m} {\sum\limits_{t = 1}^l {{A}y_i^{m - p}y_j^{t - 1}x_i^{ - {b_{ij}} - 1}x_j^{{b_{ji}} - 1}B{x_{n + i}}y_j^{l - t}y_i^{p - 1}{y_i}} } \\	
	&&	+ {D_2}y_j^ly_i^{ m + 1} + {D_2}\sum\limits_{p = 1}^{ m - 1} {\sum\limits_{t = 1}^l {{A}y_i^{ m - 1 - p}y_j^{t - 1}x_i^{ - {b_{ij}} - 1}x_j^{{b_{ji}} - 1}B{x_{n + i}}y_j^{l - t}y_i^{p - 1}y_i^2} } \\
	&&	+  \cdots  + {D_{ m}}y_j^ly_i^{ m + 1} + {D_{ m}}\sum\limits_{t = 1}^l {{A}y_j^{t - 1}x_i^{ - {b_{ij}} - 1}x_j^{{b_{ji}} - 1}B{x_{n + i}}y_j^{l - t}y_i^{ m}}  + {D_{ m + 1}}y_j^ly_i^{ m + 1}\\
	&	=& \left( {\sum\limits_{r = 0}^{m+ 1} {{D_r}} } \right)y_j^ly_i^{ m + 1} + \sum\limits_{t = 1}^l {{A}y_i^{  m}y_j^{t - 1}x_i^{ - {b_{ij}} - 1}x_j^{{b_{ji}} - 1}B{x_{n + i}}y_j^{l - t}} \\
	&&	+ \left( {{D_0} + {D_1}} \right)\sum\limits_{t = 1}^l {{A}y_i^{ m - 1}y_j^{t - 1}x_i^{ - {b_{ij}} - 1}x_j^{{b_{ji}} - 1}B{x_{n + i}}y_j^{l - t}{y_i}} \\
	&&	+ \left( {\sum\limits_{r = 0}^2 {{D_r}} } \right)\sum\limits_{t = 1}^l {{A}y_i^{ m- 2}y_j^{t - 1}x_i^{ - {b_{ij}} - 1}x_j^{{b_{ji}} - 1}B{x_{n + i}}y_j^{l - t}y_i^2} \\
	&&	+  \cdots  + \left( {\sum\limits_{r = 0}^{ m} {{D_r}} } \right)\sum\limits_{t = 1}^l {{A }y_j^{t - 1}x_i^{ - {b_{ij}} - 1}x_j^{{b_{ji}} - 1}B{x_{n + i}}y_j^{l - t}y_i^{ m}} \\
	&	=& \left( {\sum\limits_{r = 0}^{m + 1} {{D_r}} } \right)y_j^ly_i^{ m + 1} + \sum\limits_{s = 0}^{ m} {\left( {\left( {\sum\limits_{r = 0}^s {{D_r}} } \right)\sum\limits_{t = 1}^l {{A}y_i^{ m- s}y_j^{t - 1}x_i^{ - {b_{ij}} - 1}x_j^{{b_{ji}} - 1}B{x_{n + i}}y_j^{l- t}y_i^s} } \right)},
	\end{eqnarray*}	
where ${D_r} = {{{( - 1)}^r}{{\left( {{q^{{d_i}}}} \right)}^{\frac{{r(r - 1)}}{2} }}{{\left[ {\begin{array}{*{20}{c}}
						{ m + 1}\\
						r
				\end{array}} \right]}_{{q^{{d_i}}}}}}.$
		
Note that
\[\sum\limits_{r = 0}^{ m + 1} {{D_r}}  = \sum\limits_{r = 0}^{ m + 1} {{{( - 1)}^r}{{\left( {{q^{{d_i}}}} \right)}^{\frac{{r(r - 1)}}{2} }}{{\left[ {\begin{array}{*{20}{c}}
					{ m + 1}\\
					r
			\end{array}} \right]}_{{q^{{d_i}}}}}}   \mathop  = \limits^{(\ref{(q=0)})} 0.\]
	
On the other hand,	according to (\ref{(A)}), we have
		\begin{eqnarray*}
	{\left( {{\Lambda _j}} \right)_{jn + i}} &=&  - {\lambda _{jn + i}} + \sum\limits_{t = 1}^{2n} {{{[{b_{tj}}]}_ + }{\lambda _{tn + i}}}  = \sum\limits_{t = 1}^n {{{[{b_{tj}}]}_ + }{\lambda _{tn + i}}}  + \sum\limits_{t = n + 1}^{2n} {{{[{b_{tj}}]}_ + }{\lambda _{tn + i}}} \\
	&	=& {[{b_{ij}}]_ + }{\lambda _{in + i}} + {[{b_{n + jj}}]_ + }{\lambda _{n + jn + i}} = {\lambda _{n + jn + i}} =  - {d_j}{b_{ji}},\\
		{\left( {{\Lambda _i}} \right)_{in + j}} &= & - {\lambda _{i + nj}} + \sum\limits_{t = 1}^{2n} {{{[{b_{ti}}]}_ + }{\lambda _{tn + j}}}  = \sum\limits_{t = 1}^n {{{[{b_{ti}}]}_ + }{\lambda _{tn + j}}}  + \sum\limits_{t = n + 1}^{2n} {{{[{b_{ti}}]}_ + }{\lambda _{tn + j}}} \\
	&	=& {[{b_{ji}}]_ + }{\lambda _{jn + j}} + {[{b_{n + ii}}]_ + }{\lambda _{n + in + j}} =  - {d_j}{b_{ji}}- {d_i}{b_{ij}} = 0.
\end{eqnarray*}	

Thus, we obtain
	\begin{eqnarray*}
	&&\sum\limits_{s = 0}^{ m} {\left( {\left( {\sum\limits_{r = 0}^s {{D_r}} } \right)\sum\limits_{t = 1}^l {{A}y_i^{ m - s}y_j^{t - 1}x_i^{ - {b_{ij}} - 1}x_j^{{b_{ji}} - 1}B{x_{n + i}}y_j^{l- t}y_i^s} } \right)} \\
				\end{eqnarray*}		
	\begin{eqnarray*}
&	=& {A}\sum\limits_{s = 0}^{m} {\left( {\left( {\sum\limits_{r = 0}^s {{D_r}} } \right)y_i^{ m - s}x_i^{ - {b_{ij}} - 1}\left( {\sum\limits_{t = 1}^l {y_j^{t - 1}x_j^{{b_{ji}} - 1}{x_{n + i}}y_j^{l- t}} } \right)y_i^s} \right)} B\\
&	=& {A}\sum\limits_{s = 0}^{ m} {\left( {{{\left( {{q^{{d_i}}}} \right)}^{ - s}}\left( {\sum\limits_{r = 0}^s {{D_r}} } \right)y_i^{ m - s}x_i^{ - {b_{ij}} - 1}\left( {\sum\limits_{t = 1}^l {{{\left( {{q^{{d_i}}}} \right)}^{ - {b_{ij}}(l- t)}}y_j^{t - 1}x_j^{{b_{ji}} - 1}y_j^{l - t}} } \right)y_i^s} \right)} B{x_{n + i}}.
\end{eqnarray*}

We only need to prove $$\sum\limits_{s = 0}^{m} {\left( {{{\left( {{q^{{d_i}}}} \right)}^{ - s}}\left( {\sum\limits_{r = 0}^s {{D_r}} } \right)y_i^{ m - s}x_i^{ - {b_{ij}} - 1}\left( {\sum\limits_{t = 1}^l {{{\left( {{q^{{d_i}}}} \right)}^{ - {b_{ij}}(l - t)}}y_j^{t - 1}x_j^{{b_{ji}} - 1}y_j^{l - t}} } \right)y_i^s} \right)}  = 0.$$

Now, we calculate ${\sum\limits_{t = 1}^l {{{\left( {{q^{{d_i}}}} \right)}^{ - {b_{ij}}(l - t)}}y_j^{t - 1}x_j^{{b_{ji}} - 1}y_j^{l - t}} }.$ 	According to (\ref{(xnyn)n}) and (\ref{(ynxn)n}), we have
	\begin{eqnarray*}
&&	\sum\limits_{t = 1}^l {{{\left( {{q^{{d_i}}}} \right)}^{ - {b_{ij}}(l - t)}}y_j^{t - 1}x_j^{{b_{ji}} - 1}y_j^{l - t}}  = \sum\limits_{t = 1}^l {{{\left( {{q^{{d_i}}}} \right)}^{ - {b_{ij}}(l - t)}}y_j^{t - 1}x_j^{t - 1}x_j^{{b_{ji}} - l}x_j^{l - t}y_j^{l - t}} \\
&	=& \sum\limits_{t = 1}^l {\left( {{{\left( {{q^{{d_i}}}} \right)}^{ - {b_{ij}}(l - t)}}\left( {\sum\limits_{k = 0}^{t - 1} {{{\left[ {\begin{array}{*{20}{c}}
									{t - 1}\\
									k
							\end{array}} \right]}_{{q^{{d_j}}}}}{{\left( {{q^{{d_j}}}} \right)}^{\frac{{{k^2}}}{2}}}\prod\limits_{v \in [1,n]}^ \triangleleft  {x_v^{(t - 1 - k){{\left[ { - {b_{vj}}} \right]}_ + }}}  \cdot \prod\limits_{v \in [1,n]}^ \triangleleft  {x_v^{k{{\left[ {{b_{vj}}} \right]}_ + }}}  \cdot x_{n + j}^k} } \right)} \right.} \\
&&	\left. { \cdot x_j^{{b_{ji}} - l} \left( {\sum\limits_{k = 0}^{l - t} {{{\left[ {\begin{array}{*{20}{c}}
								{l - t}\\
								k
						\end{array}} \right]}_{{q^{{d_j}}}}}{{\left( {{q^{{d_j}}}} \right)}^{\frac{{{k(k-2t)}}}{2}}}\prod\limits_{v \in [1,n]}^ \triangleleft  {x_v^{(l - t - k){{\left[ { - {b_{vj}}} \right]}_ + }}}  \cdot \prod\limits_{v \in [1,n]}^ \triangleleft  {x_v^{k{{\left[ {{b_{vj}}} \right]}_ + }}}  \cdot x_{n + j}^k} } \right)} \right)\\											
&	=& 	\sum\limits_{t = 1}^l {\left( {{{\left( {{q^{{d_i}}}} \right)}^{ - {b_{ij}}(l - t)}}\left( {\sum\limits_{k = 0}^{t - 1} {{{\left[ {\begin{array}{*{20}{c}}
									{t - 1}\\
									k
							\end{array}} \right]}_{{q^{{d_j}}}}}{{\left( {{q^{{d_j}}}} \right)}^{\frac{{{k^2}}}{2}}}x_i^{ - {b_{ij}}(t - 1 - k)}\prod\limits_{v \in [1,n],v \ne i}^ \triangleleft  {x_v^{(t - 1 - k){{\left[ { - {b_{vj}}} \right]}_ + }}} } } \right.} \right.} \\
	&&\left. { \cdot \prod\limits_{v \in [1,n]}^ \triangleleft  {x_v^{k{{\left[ {{b_{vj}}} \right]}_ + }}}  \cdot x_{n + j}^k} \right)x_j^{{b_{ji}} - l}\left( {\sum\limits_{k = 0}^{l - t} {{{\left[ {\begin{array}{*{20}{c}}
							{l - t}\\
							k
					\end{array}} \right]}_{{q^{{d_j}}}}}{{\left( {{q^{{d_j}}}} \right)}^{\frac{{k(k - 2t)}}{2}}}x_i^{ - {b_{ij}}(l - t - k)}} } \right.\\
&&	\left. { \cdot \left. {\prod\limits_{v \in [1,n],v \ne i}^ \triangleleft  {x_v^{(l - t - k){{\left[ { - {b_{vj}}} \right]}_ + }}}  \cdot \prod\limits_{v \in [1,n]}^ \triangleleft  {x_v^{k{{\left[ {{b_{vj}}} \right]}_ + }}}  \cdot x_{n + j}^k} \right)} \right).
\end{eqnarray*}

 Thus, we have
\begin{equation}\label{(f_k)}
\sum\limits_{t = 1}^l {{{\left( {{q^{{d_i}}}} \right)}^{ - {b_{ij}}(l - t)}}y_j^{t - 1}x_j^{{b_{ji}} - 1}y_j^{l - t}}  = \sum\limits_{k = 0}^{l - 1} {x_i^{ - {b_{ij}}k}{f_k}} ,
\end{equation}
where ${f_k} \in \mathbb{ZP}\left[ {{x_1}, \cdots ,{x_{i - 1}},{x_{i + 1}}, \cdots ,{x_n},{x_{n + j}}} \right].$

So, we obtain
	\begin{eqnarray*}
&&	\sum\limits_{s = 0}^m {\left( {{{\left( {{q^{{d_i}}}} \right)}^{ - s}}\left( {\sum\limits_{r = 0}^s {{D_r}} } \right)y_i^{m - s}x_i^{ - {b_{ij}} - 1}\left( {\sum\limits_{t = 1}^l {{{\left( {{q^{{d_i}}}} \right)}^{ - {b_{ij}}(l - t)}}y_j^{t - 1}x_j^{{b_{ji}} - 1}y_j^{l - t}} } \right)y_i^s} \right)} \\
&	= &\sum\limits_{s = 0}^m {\left( {{{\left( {{q^{{d_i}}}} \right)}^{ - s}}\left( {\sum\limits_{r = 0}^s {{D_r}} } \right)y_i^{m - s}x_i^{ - {b_{ij}} - 1}\left( {\sum\limits_{k = 0}^{l - 1} {x_i^{ - {b_{ij}}k}{f_k}} } \right)y_i^s} \right)} \\
			\end{eqnarray*}		
\begin{eqnarray*}
&	=& \sum\limits_{s = 0}^m {\sum\limits_{k = 0}^{l - 1} {{{\left( {{q^{{d_i}}}} \right)}^{ - s}}\left( {\sum\limits_{r = 0}^s {{D_r}} } \right)y_i^{m - s}x_i^{ - {b_{ij}} - 1}x_i^{ - {b_{ij}}k}y_i^s{f_k}} } \\
	&=& \sum\limits_{k = 0}^{l - 1} {\left( {\sum\limits_{s = 0}^m {{{\left( {{q^{{d_i}}}} \right)}^{ - s}}\left( {\sum\limits_{r = 0}^s {{D_r}} } \right)y_i^{m - s}x_i^{ - {b_{ij}}(1 + k) - 1}y_i^s} } \right){f_k}} .
\end{eqnarray*}

For any $k \in [0,l-1]$,  we have  $$\sum\limits_{s = 0}^{ m} {{{\left( {{q^{{d_i}}}} \right)}^{ - s}}\left( {\sum\limits_{r = 0}^s {{D_r}} } \right)y_i^{ m - s}x_i^{ - {b_{ij}}(1 + k) - 1}y_i^s} 		\mathop  = \limits^{(\ref{(lemma4.1)})}0.$$

\item By using the exchange relations, we have	
	
\begin{eqnarray*}	{y_i}& =& {q^{ - \frac{{{d_i}}}{2}}}x_i^{ - 1}\prod\limits_{k \in [1,n],k \ne j}^ \triangleleft  {x_k^{{{\left[ {{b_{ki}}} \right]}_ + }}}  \cdot {x_{n + i}} + x_i^{ - 1}x_j^{-{b_{ji}}}\prod\limits_{k \in [1,n],k \ne j}^ \triangleleft  {x_k^{{{\left[ { - {b_{ki}}} \right]}_ + }}},\\
{y_j}& =& {q^{ - \frac{{{d_j}}}{2}}}x_j^{ - 1}x_i^{{b_{ij}}}\prod\limits_{k \in [1,n],k \ne i}^ \triangleleft  {x_k^{{{\left[ {{b_{kj}}} \right]}_ + }}}  \cdot {x_{n + j}} + x_j^{ - 1}\prod\limits_{k \in [1,n],k \ne i}^ \triangleleft  {x_k^{{{\left[ { - {b_{kj}}} \right]}_ + }}}. 	\end{eqnarray*}

	According to  (\ref{(y1y2)n11}), we have
\begin{eqnarray}\label{(yiyj)1}
	{y_i}{y_j} = {y_j}{y_i} + Cx_i^{{b_{ij}} - 1}x_j^{-{b_{ji}} - 1}D{x_{n + j}},
\end{eqnarray}
where $C = {{q^{ - \frac{{{d_j}}}{2}}} - {q^{ - \frac{{{d_j}}}{2} - {d_j}{b_{ji}}}}}$  and $D =  \prod\limits_{k \in [1,n],k \ne j}^ \triangleleft  {x_k^{{{\left[ { - {b_{ki}}} \right]}_ + }}}  \cdot \prod\limits_{k \in [1,n],k \ne i}^ \triangleleft  {x_k^{{{\left[ {{b_{kj}}} \right]}_ + }}} $.

Thus, for any $s\geq 1$, we have
	\begin{eqnarray*}
		{y_i}y_j^s &= &\left( {{y_j}{y_i} + Cx_i^{{b_{ij}} - 1}x_j^{ - {b_{ji}} - 1}D{x_{n + j}}} \right)y_j^{s - 1}\nonumber\\
	&	=& {y_j}{y_i}y_j^{s - 1} + Cx_i^{{b_{ij}} - 1}x_j^{ - {b_{ji}} - 1}D{x_{n + j}}y_j^{s - 1}\nonumber\\
	&	=& {y_j}\left( {{y_j}{y_i} + Cx_i^{{b_{ij}} - 1}x_j^{ - {b_{ji}} - 1}D{x_{n + j}}} \right)y_j^{s - 2} + Cx_i^{{b_{ij}} - 1}x_j^{ - {b_{ji}} - 1}D{x_{n + j}}y_j^{s - 1}\nonumber\\
	&	=& y_j^2{y_i}y_j^{s - 2} + C{y_j}x_i^{{b_{ij}} - 1}x_j^{ - {b_{ji}} - 1}D{x_{n + j}}y_j^{s - 2} + Cx_i^{{b_{ij}} - 1}x_j^{ - {b_{ji}} - 1}D{x_{n + j}}y_j^{s - 1}\nonumber\\
	&	=& y_j^s{y_i} + \sum\limits_{t = 1}^s {Cy_j^{t - 1}x_i^{{b_{ij}} - 1}x_j^{ - {b_{ji}} - 1}D{x_{n + j}}y_j^{s - t}} .
	\end{eqnarray*}	
	and for any $s\geq 1$, $k\geq 1$ we have
	\begin{eqnarray*}
		y_i^ky_j^s& =& y_i^{k - 1}\left( {y_j^s{y_i} + \sum\limits_{t = 1}^s {Cy_j^{t - 1}x_i^{{b_{ij}} - 1}x_j^{ - {b_{ji}} - 1}D{x_{n + j}}y_j^{s - t}} } \right)\nonumber\\
	&	=& y_i^{k - 1}y_j^s{y_i} + \sum\limits_{t = 1}^s {Cy_i^{k - 1}y_j^{t - 1}x_i^{{b_{ij}} - 1}x_j^{ - {b_{ji}} - 1}D{x_{n + j}}y_j^{s - t}} \nonumber\\
	&	=& y_i^{k - 2}y_j^sy_i^2 + \sum\limits_{t = 1}^s {Cy_i^{k - 1}y_j^{t - 1}x_i^{{b_{ij}} - 1}x_j^{ - {b_{ji}} - 1}D{x_{n + j}}y_j^{s - t}} \nonumber \\
					\end{eqnarray*}		
	\begin{eqnarray}\label{(yiyjks)}
	&&+ \sum\limits_{t = 1}^s {Cy_i^{k - 2}y_j^{t - 1}x_i^{{b_{ij}} - 1}x_j^{ - {b_{ji}} - 1}D{x_{n + j}}y_j^{s - t}{y_i}} \nonumber\\
	&	=& y_j^sy_i^k + \sum\limits_{p = 1}^k {\sum\limits_{t = 1}^s {Cy_i^{k - p}y_j^{t - 1}x_i^{{b_{ij}} - 1}x_j^{ - {b_{ji}} - 1}D{x_{n + j}}y_j^{s - t}y_i^{p - 1}} } .
	\end{eqnarray}

Now, we calculate
	\begin{eqnarray*}
	&&	\sum\limits_{r = 0}^{m + 1} {{{( - 1)}^r}{{\left( {{q^{{d_i}}}} \right)}^{\frac{{r(r - 1)}}{2} - rm}}{{\left[ {\begin{array}{*{20}{c}}
							{m + 1}\\
							r
					\end{array}} \right]}_{{q^{{d_i}}}}}y_i^{m + 1 - r}y_j^ly_i^r} \\
	&	=& y_i^{m + 1}y_j^l + {F_1}y_i^my_j^l{y_i} + {F_2}y_i^{m - 1}y_j^ly_i^2 +  \cdots  + {F_m}y_i^{}y_j^ly_i^m + {F_{m + 1}}y_j^ly_i^{m + 1}\\
	&	\mathop  = \limits^{(\ref{(yiyjks)})} &y_j^ly_i^{m + 1} + \sum\limits_{p = 1}^{m + 1} {\sum\limits_{t = 1}^l {Cy_i^{m + 1 - p}y_j^{t - 1}x_i^{{b_{ij}} - 1}x_j^{ - {b_{ji}} - 1}D{x_{n + j}}y_j^{l - t}y_i^{p - 1}} } \\
	&&	+ {F_1}y_j^ly_i^{m + 1} + {F_1}\sum\limits_{p = 1}^m {\sum\limits_{t = 1}^l {Cy_i^{m - p}y_j^{t - 1}x_i^{{b_{ij}} - 1}x_j^{ - {b_{ji}} - 1}D{x_{n + j}}y_j^{l - t}y_i^{p - 1}{y_i}} } \\
	&&	+ {F_2}y_j^ly_i^{m + 1} + {F_2}\sum\limits_{p = 1}^{m - 1} {\sum\limits_{t = 1}^l {Cy_i^{m - 1 - p}y_j^{t - 1}x_i^{{b_{ij}} - 1}x_j^{ - {b_{ji}} - 1}D{x_{n + j}}y_j^{l - t}y_i^{p - 1}y_i^2} } \\	
	&&	+  \cdots  + {F_m}y_j^ly_i^{m + 1} + {F_m}\sum\limits_{t = 1}^l {Cy_j^{t - 1}x_i^{{b_{ij}} - 1}x_j^{ - {b_{ji}} - 1}D{x_{n + j}}y_j^{l - t}y_i^m}  + {F_{m + 1}}y_j^ly_i^{m + 1}\\	
	&	=& \left( {\sum\limits_{r = 0}^{m + 1} {{F_r}} } \right)y_j^ly_i^{m + 1} + \sum\limits_{t = 1}^l {Cy_i^my_j^{t - 1}x_i^{{b_{ij}} - 1}x_j^{ - {b_{ji}} - 1}D{x_{n + j}}y_j^{l - t}} \\
	&&	+ \left( {{F_0} + {F_1}} \right)\sum\limits_{t = 1}^l {Cy_i^{m - 1}y_j^{t - 1}x_i^{{b_{ij}} - 1}x_j^{ - {b_{ji}} - 1}D{x_{n + j}}y_j^{l - t}y_i^{}} \\
	&&	+ \left( {\sum\limits_{r = 0}^2 {{F_r}} } \right)\sum\limits_{t = 1}^l {Cy_i^{m - 2}y_j^{t - 1}x_i^{{b_{ij}} - 1}x_j^{ - {b_{ji}} - 1}D{x_{n + j}}y_j^{l - t}y_i^2} \\
	&&	+  \cdots  + \left( {\sum\limits_{r = 0}^m {{F_r}} } \right)\sum\limits_{t = 1}^l {Cy_j^{t - 1}x_i^{{b_{ij}} - 1}x_j^{ - {b_{ji}} - 1}D{x_{n + j}}y_j^{l - t}y_i^m} \\
	&	= &\left( {\sum\limits_{r = 0}^{m + 1} {{F_r}} } \right)y_j^ly_i^{m + 1} + \sum\limits_{s = 0}^m {\left( {\left( {\sum\limits_{r = 0}^s {{F_r}} } \right)\sum\limits_{t = 1}^l {Cy_i^{m - s}y_j^{t - 1}x_i^{{b_{ij}} - 1}x_j^{ - {b_{ji}} - 1}D{x_{n + j}}y_j^{l - t}y_i^s} } \right)} ,
	\end{eqnarray*}		
	where ${F_r} = {( - 1)^r}{\left( {{q^{{d_i}}}} \right)^{\frac{{r(r - 1)}}{2} - rm}}{\left[ {\begin{array}{*{20}{c}}
				{m + 1}\\
				r
		\end{array}} \right]_{{q^{{d_i}}}}}.$
	
	Note that
\[\sum\limits_{r = 0}^{m + 1} {{F_r}}  = \sum\limits_{r = 0}^{m + 1} {{{( - 1)}^r}{{\left( {{q^{{d_i}}}} \right)}^{\frac{{r(r - 1)}}{2} - rm}}{{\left[ {\begin{array}{*{20}{c}}
					{m + 1}\\
					r
			\end{array}} \right]}_{{q^{{d_i}}}}}}  \mathop  = \limits^{(\ref{(-cr)})} 0.\]
	
	On the other hand,	according to (\ref{(A)}), we have
	\begin{eqnarray*}
		{\left( {{\Lambda _i}} \right)_{in + j}} &= & - {\lambda _{in + j}} + \sum\limits_{t = 1}^{2n} {{{[{b_{ti}}]}_ + }{\lambda _{tn + j}}}  = \sum\limits_{t = 1}^n {{{[{b_{ti}}]}_ + }{\lambda _{tn + j}}}  + \sum\limits_{t = n + 1}^{2n} {{{[{b_{ti}}]}_ + }{\lambda _{tn + j}}} \\
	&	=& {[{b_{ji}}]_ + }{\lambda _{jn + j}} + {[{b_{n + ii}}]_ + }{\lambda _{n + in + j}} = {\lambda _{n + in + j}} = -{d_i}{b_{ij}},\\
	\end{eqnarray*}	
	
	Thus, we obtain
	\begin{eqnarray*}
		&&	\sum\limits_{s = 0}^m {\left( {\left( {\sum\limits_{r = 0}^s {{F_r}} } \right)\sum\limits_{t = 1}^l {Cy_i^{m - s}y_j^{t - 1}x_i^{{b_{ij}} - 1}x_j^{ - {b_{ji}} - 1}D{x_{n + j}}y_j^{l - t}y_i^s} } \right)} \\
		&	=& C\sum\limits_{s = 0}^m {\left( {\left( {\sum\limits_{r = 0}^s {{F_r}} } \right)y_i^{m - s}x_i^{{b_{ij}} - 1}\left( {\sum\limits_{t = 1}^l {y_j^{t - 1}x_j^{ - {b_{ji}} - 1}{x_{n + j}}y_j^{l - t}} } \right)y_i^s} \right)} D\\
		&	= &C\sum\limits_{s = 0}^m {\left( {{{\left( {{q^{{d_i}}}} \right)}^{  s{b_{ij}}}}\left( {\sum\limits_{r = 0}^s {{F_r}} } \right)y_i^{m - s}x_i^{{b_{ij}} - 1}\left( {\sum\limits_{t = 1}^l {{{\left( {{q^{{d_j}}}} \right)}^{t - l}}y_j^{t - 1}x_j^{ - {b_{ji}} - 1}y_j^{l - t}} } \right)y_i^s} \right)} D{x_{n + j}}.
	\end{eqnarray*}
	
	We only need to prove
	\[\sum\limits_{s = 0}^m {\left( {{{\left( {{q^{{d_i}}}} \right)}^{  s{b_{ij}}}}\left( {\sum\limits_{r = 0}^s {{F_r}} } \right)y_i^{m - s}x_i^{{b_{ij}} - 1}\left( {\sum\limits_{t = 1}^l {{{\left( {{q^{{d_j}}}} \right)}^{t - l}}y_j^{t - 1}x_j^{ - {b_{ji}} - 1}y_j^{l - t}} } \right)y_i^s} \right)}=0. \]
	
	Now, we calculate ${\sum\limits_{t = 1}^l {{{\left( {{q^{{d_j}}}} \right)}^{t - l}}y_j^{t - 1}x_j^{ - {b_{ji}} - 1}y_j^{l - t}} }.$	According to (\ref{(xnyn)n}) and (\ref{(ynxn)n}), we have
	\begin{eqnarray*}
&&	\sum\limits_{t = 1}^l {{{\left( {{q^{{d_j}}}} \right)}^{t - l}}y_j^{t - 1}x_j^{ - {b_{ji}} - 1}y_j^{l - t}}  = \sum\limits_{t = 1}^l {{{\left( {{q^{{d_j}}}} \right)}^{t - l}}y_j^{t - 1}x_j^{t - 1}x_j^{ - {b_{ji}} - l}x_j^{l - t}y_j^{l - t}} \\
	&=& \sum\limits_{t = 1}^l {\left( {{{\left( {{q^{{d_j}}}} \right)}^{t - l}}\left( {\sum\limits_{k = 0}^{t - 1} {{{\left[ {\begin{array}{*{20}{c}}
									{t - 1}\\
									k
							\end{array}} \right]}_{{q^{{d_j}}}}}{{\left( {{q^{{d_j}}}} \right)}^{\frac{{{k^2}}}{2}}}\prod\limits_{v \in [1,n]}^ \triangleleft  {x_v^{(t - 1 - k){{\left[ { - {b_{vj}}} \right]}_ + }}}  \cdot \prod\limits_{v \in [1,n]}^ \triangleleft  {x_v^{k{{\left[ {{b_{vj}}} \right]}_ + }}}  \cdot x_{n + j}^k} } \right)} \right.} \\
&&	\left. { \cdot x_j^{ - {b_{ji}} - l} \left( {\sum\limits_{k = 0}^{l - t} {{{\left[ {\begin{array}{*{20}{c}}
								{l - t}\\
								k
						\end{array}} \right]}_{{q^{{d_j}}}}}{{\left( {{q^{{d_j}}}} \right)}^{\frac{{k(k - 2(l-t))}}{2}}}\prod\limits_{v \in [1,n]}^ \triangleleft  {x_v^{(l - t - k){{\left[ { - {b_{vj}}} \right]}_ + }}}  \cdot \prod\limits_{v \in [1,n]}^ \triangleleft  {x_v^{k{{\left[ {{b_{vj}}} \right]}_ + }}}  \cdot x_{n + j}^k} } \right)} \right)\\
&	=& 	\sum\limits_{t = 1}^l {\left( {{{\left( {{q^{{d_j}}}} \right)}^{t - l}}\left( {\sum\limits_{k = 0}^{t - 1} {{{\left[ {\begin{array}{*{20}{c}}
									{t - 1}\\
									k
							\end{array}} \right]}_{{q^{{d_j}}}}}{{\left( {{q^{{d_j}}}} \right)}^{\frac{{{k^2}}}{2}}}\prod\limits_{v \in [1,n]}^ \triangleleft  {x_v^{(t - 1 - k){{\left[ { - {b_{vj}}} \right]}_ + }}}  \cdot \prod\limits_{v \in [1,n],v \ne i}^ \triangleleft  {x_v^{k{{\left[ {{b_{vj}}} \right]}_ + }}} } } \right.} \right.} \\
&&	\left. { \cdot x_i^{k{b_{ij}}}x_{n + j}^k} \right)x_j^{ - {b_{ji}} - l}\left( {\sum\limits_{k = 0}^{l - t} {{{\left[ {\begin{array}{*{20}{c}}
							{l - t}\\
							k
					\end{array}} \right]}_{{q^{{d_j}}}}}{{\left( {{q^{{d_j}}}} \right)}^{\frac{{k(k - 2(l - t))}}{2}}}\prod\limits_{v \in [1,n]}^ \triangleleft  {x_v^{(l - t - k){{\left[ { - {b_{vj}}} \right]}_ + }}} } } \right.\\
&&	\left. {\left. { \cdot \prod\limits_{v \in [1,n],v \ne i}^ \triangleleft  {x_v^{k{{\left[ {{b_{vj}}} \right]}_ + }}}  \cdot x_i^{k{b_{ij}}}x_{n + j}^k} \right)} \right).
	\end{eqnarray*}
	
	Thus, we have
	\begin{equation}\label{(g_k)}
	\sum\limits_{t = 1}^l {{{\left( {{q^{{d_j}}}} \right)}^{t - l}}y_j^{t - 1}x_j^{ - {b_{ji}} - 1}y_j^{l - t}}  = \sum\limits_{k = 0}^{l - 1} {x_i^{{b_{ij}}k}x_{n + j}^k{g_k}}   ,
	\end{equation}
	where ${g_k} \in \mathbb{ZP}\left[ {{x_1}, \cdots ,{x_{i - 1}},{x_{i + 1}}, \cdots ,{x_n}} \right].$
	
	So, we obtain	
	\begin{eqnarray*}
		&&\sum\limits_{s = 0}^m {\left( {{{\left( {{q^{{d_i}}}} \right)}^{s{b_{ij}}}}\left( {\sum\limits_{r = 0}^s {{F_r}} } \right)y_i^{m - s}x_i^{{b_{ij}} - 1}\left( {\sum\limits_{t = 1}^l {{{\left( {{q^{{d_j}}}} \right)}^{t - l}}y_j^{t - 1}x_j^{ - {b_{ji}} - 1}y_j^{l - t}} } \right)y_i^s} \right)} \\
	&	=& \sum\limits_{s = 0}^m {\left( {{{\left( {{q^{{d_i}}}} \right)}^{s{b_{ij}}}}\left( {\sum\limits_{r = 0}^s {{F_r}} } \right)y_i^{m - s}x_i^{{b_{ij}} - 1}\left( {\sum\limits_{k = 0}^{l - 1} {x_i^{{b_{ij}}k}x_{n + j}^k{g_k}} } \right)y_i^s} \right)} \\
	&	=& \sum\limits_{s = 0}^m {\sum\limits_{k = 0}^{l - 1} {{{\left( {{q^{{d_i}}}} \right)}^{s{b_{ij}}}}\left( {\sum\limits_{r = 0}^s {{F_r}} } \right)y_i^{m - s}x_i^{{b_{ij}} - 1}x_i^{{b_{ij}}k}x_{n + j}^k{g_k}y_i^s} } \\
		&=& \sum\limits_{k = 0}^{l - 1} {\left( {\sum\limits_{s = 0}^m {{{\left( {{q^{{d_i}}}} \right)}^{s{b_{ij}}(1 + k)}}\left( {\sum\limits_{r = 0}^s {{F_r}} } \right)y_i^{m - s}x_i^{{b_{ij}}(1 + k) - 1}y_i^s} } \right)x_{n + j}^k{g_k}}
	\end{eqnarray*}
	
	For any $k \in [0,l-1]$,  we have  $${\sum\limits_{s = 0}^m {{{\left( {{q^{{d_i}}}} \right)}^{s{b_{ij}}(1 + k)}}\left( {\sum\limits_{r = 0}^s {{F_r}} } \right)y_i^{m - s}x_i^{{b_{ij}}(1 + k) - 1}y_i^s} } 		\mathop  = \limits^{(\ref{(lemma4.2)})}0.$$
	
\end{enumerate}	
This completes the proof.		
\end{proof}
\begin{example}
Consider the same example as in Example \ref{exam1}.

We claim that
$$\sum\limits_{r = 0}^3 {{{( - 1)}^r}{{\left( {{q^2}} \right)}^{\frac{{r(r - 1)}}{2} - 2r}}{{\left[ {\begin{array}{*{20}{c}}
						3\\
						r
				\end{array}} \right]}_{{q^2}}}y_1^{3 - r}{y_2}y_1^r}   = 0,$$
$$\sum\limits_{r = 0}^5 {{{( - 1)}^r}{q^{\frac{{r(r - 1)}}{2}}}{{\left[ {\begin{array}{*{20}{c}}
							5\\
							r
					\end{array}} \right]}_q}y_2^{5 - r}y_1^2y_2^r}  = 0.	$$			
Note that
	\[{y_2}{y_1} = {y_1}{y_2} + A{x_2}{x_4},\]
	where $A = {q^{\frac{3}{2}}} - {q^{ - \frac{1}{2}}}.$
	
	Thus,   we obatin
	\begin{eqnarray*}
			{y_2}y_1^2 &=& \left( {{y_1}{y_2} + A{x_2}{x_4}} \right){y_1} = {y_1}{y_2}{y_1} + A{x_2}{x_4}{y_1}\\
		&	=& y_1^2{y_2} + A{y_1}{x_2}{x_4} + A{x_2}{x_4}{y_1} = \left( {y_1^2{y_2} + A\left( {1 + {q^{ - 2}}} \right){x_2}{x_4}{y_1}} \right)
	\end{eqnarray*}	
	and for any $s\geq 1,$ we have
	\begin{eqnarray*}
		y_2^sy_1^2& =&y_2^{s - 1}\left( {y_1^2{y_2} + A\left( {1 + {q^{ - 2}}} \right){x_2}{x_4}{y_1}} \right)\\
	&	= &y_2^{s - 1}y_1^2{y_2} + A\left( {1 + {q^{ - 2}}} \right)y_2^{s - 1}{x_2}{x_4}{y_1}\\
	&	= &y_2^{s - 2}y_1^2y_2^2 + A\left( {1 + {q^{ - 2}}} \right)y_2^{s - 1}{x_2}{x_4}{y_1} + A\left( {1 + {q^{ - 2}}} \right)y_2^{s - 2}{x_2}{x_4}{y_1}{y_2}\\
		\end{eqnarray*}		
	\begin{eqnarray*}
	&	=& y_1^2y_2^s + A\left( {1 + {q^{ - 2}}} \right)\sum\limits_{t = 1}^s {y_2^{s - t}{x_2}{x_4}{y_1}y_2^{t - 1}}.
	\end{eqnarray*}
	
	Then, we get
	\begin{eqnarray*}
	&&	\sum\limits_{r = 0}^5 {{{( - 1)}^r}{q^{\frac{{r(r - 1)}}{2}}}{{\left[ {\begin{array}{*{20}{c}}
							5\\
							r
					\end{array}} \right]}_q}y_2^{5 - r}y_1^2y_2^r} \\
	&	=& {D_0}y_2^5y_1^2 + {D_1}y_2^4y_1^2y_2^{} + {D_2}y_2^3y_1^2y_2^2 + {D_3}y_2^2y_1^2y_2^3 + {D_4}y_2^{}y_1^2y_2^4 + {D_5}y_1^2y_2^5\\
	&	=& {D_0}y_1^2y_2^5 + A\left( {1 + {q^{ - 2}}} \right){D_0}\sum\limits_{t = 1}^5 {y_2^{5 - t}{x_2}{x_4}{y_1}y_2^{t - 1}} \\
	&&	+ {D_1}y_1^2y_2^5 + A\left( {1 + {q^{ - 2}}} \right){D_1}\sum\limits_{t = 1}^4 {y_2^{4 - t}{x_2}{x_4}{y_1}y_2^{t - 1}{y_2}} \\
	&&	+ {D_2}y_1^2y_2^5 + A\left( {1 + {q^{ - 2}}} \right){D_2}\sum\limits_{t = 1}^3 {y_2^{3 - t}{x_2}{x_4}{y_1}y_2^{t - 1}y_2^2} \\	
	&&	+ {D_3}y_1^2y_2^5 + A\left( {1 + {q^{ - 2}}} \right){D_3}\sum\limits_{t = 1}^2 {y_2^{2 - t}{x_2}{x_4}{y_1}y_2^{t - 1}y_2^3} \\
	&&	+ {D_4}y_1^2y_2^5 + A\left( {1 + {q^{ - 2}}} \right){D_4}{x_2}{x_4}{y_1}y_2^4 + {D_5}y_1^2y_2^5\\
	&	=& \left( {\sum\limits_{r = 0}^5 {{D_r}} } \right)y_1^2y_2^5 + A\left( {1 + {q^{ - 2}}} \right)\left( {{D_0}y_2^4{x_2}{x_4}{y_1} + \left( {{D_0} + {D_1}} \right)y_2^3{x_2}{x_4}{y_1}{y_2}} \right.\\
	&&	\left. { + \left( {\sum\limits_{r = 0}^2 {{D_r}} } \right)y_2^2{x_2}{x_4}{y_1}y_2^2 + \left( {\sum\limits_{r = 0}^3 {{D_r}} } \right)y_2^{}{x_2}{x_4}{y_1}y_2^3 + \left( {\sum\limits_{r = 0}^4 {{D_r}} } \right){x_2}{x_4}{y_1}y_2^4} \right)\\
	&	=& A\left( {1 + {q^{ - 2}}} \right)\left( {{D_0}{q^2}y_2^4{x_2}{y_1} + \left( {{D_0} + {D_1}} \right)qy_2^3{x_2}{y_1}{y_2} + \left( {\sum\limits_{r = 0}^2 {{D_r}} } \right)y_2^2{x_2}{y_1}y_2^2} \right.\\
	&&	\left. { + \left( {\sum\limits_{r = 0}^3 {{D_r}} } \right){q^{ - 1}}y_2^{}{x_2}{y_1}y_2^3 + {q^{ - 2}}\left( {\sum\limits_{r = 0}^4 {{D_r}} } \right){x_2}{y_1}y_2^4} \right){x_4},
	\end{eqnarray*}
where ${D_r} = {( - 1)^r}{q^{\frac{{r(r - 1)}}{2}}}{\left[ {\begin{array}{*{20}{c}}
			5\\
			r
	\end{array}} \right]_q}.$
Hence, the claim can be deduced from the following direct computation
	\begin{eqnarray*}
	&&	{D_0}{q^2}y_2^4{x_2}{y_1} + \left( {{D_0} + {D_1}} \right)qy_2^3{x_2}{y_1}{y_2} + \left( {\sum\limits_{r = 0}^2 {{D_r}} } \right)y_2^2{x_2}{y_1}y_2^2\\
	&&	+ \left( {\sum\limits_{r = 0}^3 {{D_r}} } \right){q^{ - 1}}y_2^{}{x_2}{y_1}y_2^3 + {q^{ - 2}}\left( {\sum\limits_{r = 0}^4 {{D_r}} } \right){x_2}{y_1}y_2^4\\
	&	=& {D_0}{q^2}y_2^4{x_2}\left( {{q^{ - 1}}x_1^{ - 1}{x_3} + x_1^{ - 1}x_2^2} \right) + \left( {{D_0} + {D_1}} \right)qy_2^3{x_2}\left( {{q^{ - 1}}x_1^{ - 1}{x_3} + x_1^{ - 1}x_2^2} \right){y_2}\\
		\end{eqnarray*}		
	\begin{eqnarray*}
	&&	+ \left( {\sum\limits_{r = 0}^2 {{D_r}} } \right)y_2^2{x_2}\left( {{q^{ - 1}}x_1^{ - 1}{x_3} + x_1^{ - 1}x_2^2} \right)y_2^2 + \left( {\sum\limits_{r = 0}^3 {{D_r}} } \right){q^{ - 1}}y_2^{}{x_2}\left( {{q^{ - 1}}x_1^{ - 1}{x_3} + x_1^{ - 1}x_2^2} \right)y_2^3\\		
	&&	+ {q^{ - 2}}\left( {\sum\limits_{r = 0}^4 {{D_r}} } \right){x_2}\left( {{q^{ - 1}}x_1^{ - 1}{x_3} + x_1^{ - 1}x_2^2} \right)y_2^4\\
	&	=& {D_0}qy_2^4{x_2}x_1^{ - 1}{x_3} + \left( {{D_0} + {D_1}} \right)y_2^3{x_2}x_1^{ - 1}{x_3}{y_2} + \left( {\sum\limits_{r = 0}^2 {{D_r}} } \right){q^{ - 1}}y_2^2{x_2}x_1^{ - 1}{x_3}y_2^2\\
	&&	+ \left( {\sum\limits_{r = 0}^3 {{D_r}} } \right){q^{ - 2}}y_2^{}{x_2}x_1^{ - 1}{x_3}y_2^3 + {q^{ - 3}}\left( {\sum\limits_{r = 0}^4 {{D_r}} } \right){x_2}x_1^{ - 1}{x_3}y_2^4\\
	&&	+ {D_0}{q^2}y_2^4{x_2}x_1^{ - 1}x_2^2 + \left( {{D_0} + {D_1}} \right)qy_2^3{x_2}x_1^{ - 1}x_2^2{y_2} + \left( {\sum\limits_{r = 0}^2 {{D_r}} } \right)y_2^2{x_2}x_1^{ - 1}x_2^2y_2^2\\
	&&	+ \left( {\sum\limits_{r = 0}^3 {{D_r}} } \right){q^{ - 1}}y_2^{}{x_2}x_1^{ - 1}x_2^2y_2^3 + {q^{ - 2}}\left( {\sum\limits_{r = 0}^4 {{D_r}} } \right){x_2}x_1^{ - 1}x_2^2y_2^4\\	
	&	=& \left( {{D_0}qy_2^4{x_2} + \left( {{D_0} + {D_1}} \right)y_2^3{x_2}{y_2} + \left( {\sum\limits_{r = 0}^2 {{D_r}} } \right){q^{ - 1}}y_2^2{x_2}y_2^2} \right.\\
	&&	\left. { + \left( {\sum\limits_{r = 0}^3 {{D_r}} } \right){q^{ - 2}}y_2^{}{x_2}y_2^3 + {q^{ - 3}}\left( {\sum\limits_{r = 0}^4 {{D_r}} } \right){x_2}y_2^4} \right)x_1^{ - 1}{x_3}\\		
	&&+\left( {  {D_0}{q^2}y_2^4x_2^3x_1^{ - 1} + \left( {{D_0} + {D_1}} \right)qy_2^3x_2^3{y_2} + \left( {\sum\limits_{r = 0}^2 {{D_r}} } \right)y_2^2x_2^3y_2^2} \right.\\
	&&	\left. { + \left( {\sum\limits_{r = 0}^3 {{D_r}} } \right){q^{ - 1}}y_2^{}x_2^3y_2^3 + {q^{ - 2}}\left( {\sum\limits_{r = 0}^4 {{D_r}} } \right)x_2^3y_2^4} \right)x_1^{ - 1}.
	\end{eqnarray*}
	
Note that  	
	\begin{eqnarray*}
	&&	{D_0}qy_2^4{x_2} + \left( {{D_0} + {D_1}} \right)y_2^3{x_2}{y_2} + \left( {\sum\limits_{r = 0}^2 {{D_r}} } \right){q^{ - 1}}y_2^2{x_2}y_2^2\\&& + \left( {\sum\limits_{r = 0}^3 {{D_r}} } \right){q^{ - 2}}y_2^{}{x_2}y_2^3 + {q^{ - 3}}\left( {\sum\limits_{r = 0}^4 {{D_r}} } \right){x_2}y_2^4\\
	&	=&	{D_0}y_2^3\left( {{q^{\frac{3}{2}}}{x_1}{x_4} + q} \right) + \left( {{D_0} + {D_1}} \right)y_2^3\left( {{q^{ - \frac{1}{2}}}{x_1}{x_4} + 1} \right) + \left( {\sum\limits_{r = 0}^2 {{D_r}} } \right)y_2^2\left( {{q^{ - \frac{3}{2}}}{x_1}{x_4} + {q^{ - 1}}} \right)y_2^{}\\
	&&	+ \left( {\sum\limits_{r = 0}^3 {{D_r}} } \right)\left( {{q^{ - \frac{3}{2}}}{x_1}{x_4} + {q^{ - 2}}} \right)y_2^3 + \left( {\sum\limits_{r = 0}^4 {{D_r}} } \right)\left( {{q^{ - \frac{7}{2}}}{x_1}{x_4} + {q^{ - 3}}} \right)y_2^3\\	
	&	= &{q^{\frac{3}{2}}}{D_0}y_2^3{x_1}{x_4} + {q^{ - \frac{1}{2}}}\left( {{D_0} + {D_1}} \right)y_2^3{x_1}{x_4} + {q^{ - \frac{3}{2}}}\left( {\sum\limits_{r = 0}^2 {{D_r}} } \right)y_2^2{x_1}{x_4}{y_2} + {q^{ - \frac{3}{2}}}\left( {\sum\limits_{r = 0}^3 {{D_r}} } \right){x_1}{x_4}y_2^3\\
		\end{eqnarray*}		
	\begin{eqnarray*}
	&&	+ {q^{ - \frac{7}{2}}}\left( {\sum\limits_{r = 0}^4 {{D_r}} } \right){x_1}{x_4}y_2^3 \\
	&&+ \left( {q{D_0} + \left( {{D_0} + {D_1}} \right) + {q^{ - 1}}\left( {\sum\limits_{r = 0}^2 {{D_r}} } \right) + {q^{ - 2}}\left( {\sum\limits_{r = 0}^3 {{D_r}} } \right) + {q^{ - 3}}\left( {\sum\limits_{r = 0}^4 {{D_r}} } \right)} \right)y_2^3\\	
	&	=& \left( {{q^{\frac{9}{2}}}{D_0} + {q^{\frac{5}{2}}}\left( {{D_0} + {D_1}} \right) + {q^{\frac{1}{2}}}\left( {\sum\limits_{r = 0}^2 {{D_r}} } \right) + {q^{ - \frac{3}{2}}}\left( {\sum\limits_{r = 0}^3 {{D_r}} } \right) + {q^{ - \frac{7}{2}}}\left( {\sum\limits_{r = 0}^4 {{D_r}} } \right)} \right){x_1}{x_4}y_2^3\\
	&&	+ \left( {q\sum\limits_{t = 0}^4 {{q^{ - t}}\left( {\sum\limits_{r = 0}^t {{D_r}} } \right)} } \right)y_2^3\\
	&	=& \left( {{q^{\frac{9}{2}}}\sum\limits_{t = 0}^4 {{q^{ - 2t}}\left( {\sum\limits_{r = 0}^t {{D_r}} } \right)} } \right){x_1}{x_4}y_2^3 + \left( {q\sum\limits_{t = 0}^4 {{q^{ - t}}\left( {\sum\limits_{r = 0}^t {{D_r}} } \right)} } \right)y_2^3 	\mathop  = \limits^{(\ref{(lemma3)})} 0,
	\end{eqnarray*}
and for any $t\ge 1$,
	\begin{eqnarray*}
		y_2^tx_2^t =\sum\limits_{k = 0}^t {{{\left[ {\begin{array}{*{20}{c}}
							t\\
							k
					\end{array}} \right]}_q}{q^{\frac{{{k^2}}}{2}}}x_1^kx_4^k} ;\quad
		x_2^ty_2^t = \sum\limits_{k = 0}^t {{{\left[ {\begin{array}{*{20}{c}}
							t\\
							k
					\end{array}} \right]}_q}{q^{\frac{{k(k - 2t)}}{2}}}x_1^kx_4^k} .
	\end{eqnarray*}	
	
Thus, we obtain
	\begin{eqnarray*}
		&&	{D_0}{q^2}y_2^4x_2^3 + \left( {{D_0} + {D_1}} \right)qy_2^3x_2^3{y_2} + \left( {\sum\limits_{r = 0}^2 {{D_r}} } \right)y_2^2x_2^3y_2^2\\&& + \left( {\sum\limits_{r = 0}^3 {{D_r}} } \right){q^{ - 1}}y_2^{}x_2^3y_2^3 + {q^{ - 2}}\left( {\sum\limits_{r = 0}^4 {{D_r}} } \right)x_2^3y_2^4\\
		&	=& {D_0}{q^2}{y_2}\left( {\sum\limits_{k = 0}^3 {{{\left[ {\begin{array}{*{20}{c}}
									3\\
									k
							\end{array}} \right]}_q}{q^{\frac{{{k^2}}}{2}}}x_1^kx_4^k} } \right) + \left( {{D_0} + {D_1}} \right)q\left( {\sum\limits_{k = 0}^3 {{{\left[ {\begin{array}{*{20}{c}}
									3\\
									k
							\end{array}} \right]}_q}{q^{\frac{{{k^2}}}{2}}}x_1^kx_4^k} } \right){y_2}\\
		&&	+ \left( {\sum\limits_{r = 0}^2 {{D_r}} } \right)\left( {\sum\limits_{k = 0}^2 {{{\left[ {\begin{array}{*{20}{c}}
									2\\
									k
							\end{array}} \right]}_q}{q^{\frac{{{k^2}}}{2}}}x_1^kx_4^k} } \right)\left( {1 + {q^{ - \frac{1}{2}}}x_1^{}x_4^{}} \right){y_2} \\&&+ \left( {\sum\limits_{r = 0}^3 {{D_r}} } \right){q^{ - 1}}y_2^{}\left( {\sum\limits_{k = 0}^3 {{{\left[ {\begin{array}{*{20}{c}}
									3\\
									k
							\end{array}} \right]}_q}{q^{\frac{{k(k - 6)}}{2}}}x_1^kx_4^k} } \right)\\
		&&	+ {q^{ - 2}}\left( {\sum\limits_{r = 0}^4 {{D_r}} } \right)\left( {\sum\limits_{k = 0}^3 {{{\left[ {\begin{array}{*{20}{c}}
									3\\
									k
							\end{array}} \right]}_q}{q^{\frac{{k(k - 6)}}{2}}}x_1^kx_4^k} } \right){y_2}\\
		&	= &{D_0}{q^2}\left( {\sum\limits_{k = 0}^3 {{{\left[ {\begin{array}{*{20}{c}}
									3\\
									k
							\end{array}} \right]}_q}{q^{\frac{{{k^2} + 2k}}{2}}}x_1^kx_4^k} } \right){y_2} + \left( {{D_0} + {D_1}} \right)q\left( {\sum\limits_{k = 0}^3 {{{\left[ {\begin{array}{*{20}{c}}
									3\\
									k
							\end{array}} \right]}_q}{q^{\frac{{{k^2}}}{2}}}x_1^kx_4^k} } \right){y_2}\\
		&&	+ \left( {\sum\limits_{r = 0}^2 {{D_r}} } \right)\left( {\sum\limits_{k = 0}^2 {{{\left[ {\begin{array}{*{20}{c}}
									2\\
									k
							\end{array}} \right]}_q}{q^{\frac{{{k^2}}}{2}}}x_1^kx_4^k} } \right)\left( {1 + {q^{ - \frac{1}{2}}}x_1^{}x_4^{}} \right){y_2}\\
							\end{eqnarray*}		
						\begin{eqnarray*}
						&& + \left( {\sum\limits_{r = 0}^3 {{D_r}} } \right){q^{ - 1}}\left( {\sum\limits_{k = 0}^3 {{{\left[ {\begin{array}{*{20}{c}}
									3\\
									k
							\end{array}} \right]}_q}{q^{\frac{{k(k - 4)}}{2}}}x_1^kx_4^k} } \right){y_2}\\
		&&	+ {q^{ - 2}}\left( {\sum\limits_{r = 0}^4 {{D_r}} } \right)\left( {\sum\limits_{k = 0}^3 {{{\left[ {\begin{array}{*{20}{c}}
									3\\
									k
							\end{array}} \right]}_q}{q^{\frac{{k(k - 6)}}{2}}}x_1^kx_4^k} } \right){y_2}\\
		&	=& {y_2}\left( {{D_0}{q^2} + \left( {{D_0} + {D_1}} \right)q + \left( {\sum\limits_{r = 0}^2 {{D_r}} } \right) + \left( {\sum\limits_{r = 0}^3 {{D_r}} } \right){q^{ - 1}} + {q^{ - 2}}\left( {\sum\limits_{r = 0}^4 {{D_r}} } \right)} \right)\\
		&&	+ {x_1}{x_4}{y_2}{\left[ {\begin{array}{*{20}{c}}
						3\\
						1
				\end{array}} \right]_q}\left( {{D_0}{q^{\frac{7}{2}}} + \left( {{D_0} + {D_1}} \right){q^{\frac{3}{2}}} + \left( {\sum\limits_{r = 0}^2 {{D_r}} } \right){q^{ - \frac{1}{2}}} + \left( {\sum\limits_{r = 0}^3 {{D_r}} } \right){q^{ - \frac{5}{2}}} + {q^{ - \frac{9}{2}}}\left( {\sum\limits_{r = 0}^4 {{D_r}} } \right)} \right)\\
		&&	+ x_1^2x_4^2{y_2}{\left[ {\begin{array}{*{20}{c}}
						3\\
						2
				\end{array}} \right]_q}\left( {{D_0}{q^6} + \left( {{D_0} + {D_1}} \right){q^3} + \left( {\sum\limits_{r = 0}^2 {{D_r}} } \right) + \left( {\sum\limits_{r = 0}^3 {{D_r}} } \right){q^{ - 3}} + {q^{ - 6}}\left( {\sum\limits_{r = 0}^4 {{D_r}} } \right)} \right)\\
		&&	+ x_1^3x_4^3{y_2}\left( {{D_0}{q^{\frac{{19}}{2}}} + \left( {{D_0} + {D_1}} \right){q^{\frac{{11}}{2}}} + \left( {\sum\limits_{r = 0}^2 {{D_r}} } \right){q^{\frac{3}{2}}} + \left( {\sum\limits_{r = 0}^3 {{D_r}} } \right){q^{ - \frac{5}{2}}} + {q^{ - \frac{{13}}{2}}}\left( {\sum\limits_{r = 0}^4 {{D_r}} } \right)} \right)\\	
		&	=& {y_2}\left( {{q^2}\sum\limits_{t = 0}^4 {{q^{ - t}}\left( {\sum\limits_{r = 0}^t {{D_r}} } \right)} } \right) + {x_1}{x_4}{y_2}{\left[ {\begin{array}{*{20}{c}}
						3\\
						1
				\end{array}} \right]_q}\left( {{q^{\frac{7}{2}}}\sum\limits_{t = 0}^4 {{q^{ - 2t}}\left( {\sum\limits_{r = 0}^t {{D_r}} } \right)} } \right)\\
		&&	+ x_1^2x_4^2{y_2}{\left[ {\begin{array}{*{20}{c}}
						3\\
						2
				\end{array}} \right]_q}\left( {{q^6}\sum\limits_{t = 0}^4 {{q^{ - 3t}}\left( {\sum\limits_{r = 0}^t {{D_r}} } \right)} } \right) + x_1^3x_4^3{y_2}\left( {{q^{\frac{{19}}{2}}}\sum\limits_{t = 0}^4 {{q^{ - 4t}}\left( {\sum\limits_{r = 0}^t {{D_r}} } \right)} } \right)
			\mathop  = \limits^{(\ref{(lemma3)})} 0.
	\end{eqnarray*}

	Similarly, we have
	\begin{eqnarray*}
&&	\sum\limits_{r = 0}^3 {{{( - 1)}^r}{{\left( {{q^2}} \right)}^{\frac{{r(r - 1)}}{2} - 2r}}{{\left[ {\begin{array}{*{20}{c}}
						3\\
						r
				\end{array}} \right]}_{{q^2}}}y_1^{3 - r}{y_2}y_1^r} \\
&	=& y_1^3{y_2} - {q^{ - 4}}{\left[ {\begin{array}{*{20}{c}}
				3\\
				1
		\end{array}} \right]_{{q^2}}}y_1^2{y_2}y_1^{} + {q^{ - 6}}{\left[ {\begin{array}{*{20}{c}}
				3\\
				2
		\end{array}} \right]_{{q^2}}}y_1^{}{y_2}y_1^2 - {q^{ - 6}}{y_2}y_1^3\\
&	= &{y_2}y_1^3 - Ay_1^2{x_2}{x_4} - A{y_1}{x_2}{x_4}{y_1} - A{x_2}{x_4}y_1^2\\
&&	- {q^{ - 4}}{\left[ {\begin{array}{*{20}{c}}
				3\\
				1
		\end{array}} \right]_{{q^2}}}{y_2}y_1^3 + {q^{ - 4}}{\left[ {\begin{array}{*{20}{c}}
				3\\
				1
		\end{array}} \right]_{{q^2}}}A{y_1}{x_2}{x_4}{y_1} + {q^{ - 4}}{\left[ {\begin{array}{*{20}{c}}
				3\\
				1
		\end{array}} \right]_{{q^2}}}A{x_2}{x_4}y_1^2\\
&&	+ {q^{ - 6}}{\left[ {\begin{array}{*{20}{c}}
				3\\
				2
		\end{array}} \right]_{{q^2}}}{y_2}y_1^3 - {q^{ - 6}}{\left[ {\begin{array}{*{20}{c}}
				3\\
				2
		\end{array}} \right]_{{q^2}}}A{x_2}{x_4}y_1^2 - {q^{ - 6}}{y_2}y_1^3\\
&	=& \left( {1 - {q^{ - 4}}{{\left[ {\begin{array}{*{20}{c}}
						3\\
						1
				\end{array}} \right]}_{{q^2}}} + {q^{ - 6}}{{\left[ {\begin{array}{*{20}{c}}
						3\\
						2
				\end{array}} \right]}_{{q^2}}} - {q^{ - 6}}} \right){y_2}y_1^3 - Ay_1^2{x_2}{x_4}\\
&&	- A\left( {1 - {q^{ - 4}}{{\left[ {\begin{array}{*{20}{c}}
						3\\
						1
				\end{array}} \right]}_{{q^2}}}} \right){y_1}{x_2}{x_4}{y_1} - A\left( {1 - {q^{ - 4}}{{\left[ {\begin{array}{*{20}{c}}
						3\\
						1
				\end{array}} \right]}_{{q^2}}} + {q^{ - 6}}{{\left[ {\begin{array}{*{20}{c}}
						3\\
						2
				\end{array}} \right]}_{{q^2}}}} \right){x_2}{x_4}y_1^2\\
&	=& \left( {1 - {q^{ - 4}} - {q^{ - 2}} - 1 + {q^{ - 6}} + {q^{ - 4}} + {q^{ - 2}} - {q^{ - 6}}} \right){y_2}y_1^3 - A{q^{ - 4}}{x_2}{x_4}y_1^2\\
	\end{eqnarray*}		
\begin{eqnarray*}
&&	- A{q^{ - 2}}\left( {1 - {q^{ - 4}}{{\left[ {\begin{array}{*{20}{c}}
						3\\
						1
				\end{array}} \right]}_{{q^2}}}} \right){x_2}{x_4}y_1^2 - A\left( {1 - {q^{ - 4}}{{\left[ {\begin{array}{*{20}{c}}
						3\\
						1
				\end{array}} \right]}_{{q^2}}} + {q^{ - 6}}{{\left[ {\begin{array}{*{20}{c}}
						3\\
						2
				\end{array}} \right]}_{{q^2}}}} \right){x_2}{x_4}y_1^2\\
&	= & - A{x_2}{x_4}y_1^2\left( {{q^{ - 4}} + {q^{ - 2}} - {q^{ - 6}} - {q^{ - 4}} - {q^{ - 2}} + 1 - {q^{ - 4}} - {q^{ - 2}} - 1 + {q^{ - 6}} + {q^{ - 4}} + {q^{ - 2}}} \right)\\
&	=& 0.
	\end{eqnarray*}
\end{example}
\begin{example}
Consider the same example as in Example \ref{exam3}.
	
	We claim that
	\begin{eqnarray*}
			\sum\limits_{r = 0}^5 {{{( - 1)}^r}{q^{\frac{{r(r - 1)}}{2}}}{{\left[ {\begin{array}{*{20}{c}}
							5\\
							r
					\end{array}} \right]}_q}y_1^{5- r}{y_3^2}y_1^r} & =& 0.\\
	\end{eqnarray*}

	Note that
	\begin{eqnarray*}
		{y_1}{y_3}& =& {q^{ - 1}}x_1^{ - 1}x_2^2x_3^{}{x_4}{x_6} + {q^{ - \frac{1}{2}}}x_1^{ - 1}x_2^4x_3^{ - 1}{x_6} + {q^{\frac{3}{2}}}x_1^{}x_3^{}{x_4} + x_1^{}x_2^2x_3^{ - 1},\\
		{y_3}{y_1} &=& {q^{ - 1}}x_1^{ - 1}x_2^2x_3^{}{x_4}{x_6} + {q^{ - \frac{1}{2}}}x_1x_3^{}{x_4} + {q^{ - \frac{1}{2}}}x_1^{ - 1}x_2^4x_3^{ - 1}{x_6} + x_1^{}x_2^2x_3^{ - 1}.
	\end{eqnarray*}
	
	It follows that
\[{y_1}{y_3} = {y_3}{y_1} + \left( {{q^{\frac{3}{2}}} - {q^{ - \frac{1}{2}}}} \right)x_1^{}x_3^{}{x_4} = {y_3}{y_1} + Ax_1^{}x_3^{}{x_4},\]
	where $A = {q^{\frac{3}{2}}} - {q^{ - \frac{1}{2}}}.$
	
	Then,   we have
	\begin{eqnarray*}
		{y_1}y_3^2& =& \left( {{y_3}{y_1} + Ax_1^{}x_3^{}{x_4}} \right){y_3} = {y_3}{y_1}{y_3} + A{x_1}{x_3}{x_4}{y_3}\\
	&	=& y_3^2{y_1} + A{x_1}{x_3}{x_4}{y_3} + A{y_3}{x_1}{x_3}{x_4} = y_3^2{y_1} + A\left( {{q^2}{x_3}{y_3} + {y_3}{x_3}} \right){x_1}{x_4}\\
	&	=& y_3^2{y_1} + A\left( {\left( {{q^{\frac{3}{2}}} + {q^{\frac{1}{2}}}} \right)x_2^2{x_6} + \left( {{q^2} + 1} \right)x_1^2} \right){x_1}{x_4}\\
	&	=& y_3^2{y_1} + ABx_2^2{x_6}{x_1}{x_4} + ACx_1^3{x_4},
	\end{eqnarray*}
	where $B = {q^{\frac{3}{2}}} + {q^{\frac{1}{2}}}$ and $C = {q^2} + 1.$
	
	Thus, for any $s\geq 1,$ we obtain
	\begin{eqnarray*}
		y_1^sy_3^2 &=& y_1^{s - 1}\left( {y_3^2{y_1} + ABx_2^2{x_6}{x_1}{x_4} + ACx_1^3{x_4}} \right)\\
	&	=& y_1^{s - 1}y_3^2{y_1} + ABy_1^{s - 1}x_2^2{x_6}{x_1}{x_4} + ACy_1^{s - 1}x_1^3{x_4}\\
	&	=& y_1^{s - 2}y_3^2y_1^2 + ABy_1^{s - 1}x_2^2{x_6}{x_1}{x_4} + ACy_1^{s - 1}x_1^3{x_4}\\
	&& + ABy_1^{s - 2}x_2^2{x_6}{x_1}{x_4}{y_1} + ACy_1^{s - 2}x_1^3{x_4}{y_1}\\
	&	=& y_3^2y_1^s + AB\sum\limits_{t = 1}^s {y_1^{s - t}x_2^2{x_6}{x_1}{x_4}y_1^{t - 1}}  + AC\sum\limits_{t = 1}^s {y_1^{s - t}x_1^3{x_4}y_1^{t - 1}} .
	\end{eqnarray*}
	
	Now, we calculate
	\begin{eqnarray*}
	&&	\sum\limits_{r = 0}^5 {{{( - 1)}^r}{q^{\frac{{r(r - 1)}}{2}}}{{\left[ {\begin{array}{*{20}{c}}
							5\\
							r
					\end{array}} \right]}_q}y_1^{5 - r}y_3^2y_1^r} \\
	&	=&{D_0}y_1^5y_3^2 + {D_1}y_1^4y_3^2y_1^{} + {D_2}y_1^3y_3^2y_1^2 + {D_3}y_1^2y_3^2y_1^3 + {D_4}y_1^{}y_3^2y_1^4 + {D_5}y_3^2y_1^5\\
		\end{eqnarray*}		
	\begin{eqnarray*}
	&	=& {D_0}y_3^2y_1^5 + AB{D_0}\sum\limits_{t = 1}^5 {y_1^{5 - t}x_2^2{x_6}{x_1}{x_4}y_1^{t - 1}}  + AC{D_0}\sum\limits_{t = 1}^5 {y_1^{5 - t}x_1^3{x_4}y_1^{t - 1}} \\
	&&	+ {D_1}y_3^2y_1^5 + AB{D_1}\sum\limits_{t = 1}^4 {y_1^{4 - t}x_2^2{x_6}{x_1}{x_4}y_1^{t - 1}{y_1}}  + AC{D_1}\sum\limits_{t = 1}^4 {y_1^{4 - t}x_1^3{x_4}y_1^{t - 1}{y_1}} \\
	&&	+ {D_2}y_3^2y_1^5 + AB{D_2}\sum\limits_{t = 1}^3 {y_1^{3 - t}x_2^2{x_6}{x_1}{x_4}y_1^{t - 1}y_1^2}  + AC{D_2}\sum\limits_{t = 1}^3 {y_1^{3 - t}x_1^3{x_4}y_1^{t - 1}y_1^2} \\
	&&	+ {D_3}y_3^2y_1^5 + AB{D_3}\sum\limits_{t = 1}^2 {y_1^{2 - t}x_2^2{x_6}{x_1}{x_4}y_1^{t - 1}y_1^3}  + AC{D_3}\sum\limits_{t = 1}^2 {y_1^{2 - t}x_1^3{x_4}y_1^{t - 1}y_1^3} \\
	&&	+ {D_4}y_3^2y_1^5 + AB{D_4}x_2^2{x_6}{x_1}{x_4}y_1^4 + AC{D_4}x_1^3{x_4}y_1^4 + {D_5}y_3^2y_1^5\\
	&	= &\left( {\sum\limits_{r = 0}^5 {{D_r}} } \right)y_3^2y_1^5 + AB\left( {{D_0}y_1^4x_2^2{x_6}{x_1}{x_4} + \left( {{D_0} + {D_1}} \right)y_1^3x_2^2{x_6}{x_1}{x_4}y_1^{}} \right.\\
	&&	\left. { + \left( {\sum\limits_{r = 0}^2 {{D_r}} } \right)y_1^2x_2^2{x_6}{x_1}{x_4}y_1^2 + \left( {\sum\limits_{r = 0}^3 {{D_r}} } \right)y_1^{}x_2^2{x_6}{x_1}{x_4}y_1^3 + \left( {\sum\limits_{r = 0}^4 {{D_r}} } \right)x_2^2{x_6}{x_1}{x_4}y_1^4} \right)\\
	&&	+ AC\left( {{D_0}y_1^4x_1^3{x_4} + \left( {{D_0} + {D_1}} \right)y_1^3x_1^3{x_4}y_1^{} + \left( {\sum\limits_{r = 0}^2 {{D_r}} } \right)y_1^2x_1^3{x_4}y_1^2} \right.\\
	&&	\left. { + \left( {\sum\limits_{r = 0}^3 {{D_r}} } \right)y_1^{}x_1^3{x_4}y_1^3 + \left( {\sum\limits_{r = 0}^4 {{D_r}} } \right)x_1^3{x_4}y_1^4} \right)\\
	&	=& \left( {\sum\limits_{r = 0}^5 {{D_r}} } \right)y_3^2y_1^5 + AB\left( {{D_0}y_1^4{x_1} + {q^{ - 1}}\left( {{D_0} + {D_1}} \right)y_1^3{x_1}y_1^{}} \right.\\
	&&	\left. { + {q^{ - 2}}\left( {\sum\limits_{r = 0}^2 {{D_r}} } \right)y_1^2{x_1}y_1^2 + {q^{ - 3}}\left( {\sum\limits_{r = 0}^3 {{D_r}} } \right)y_1^{}{x_1}y_1^3 + {q^{ - 4}}\left( {\sum\limits_{r = 0}^4 {{D_r}} } \right){x_1}y_1^4} \right)x_2^2{x_6}{x_4}\\
	&&	+ AC\left( {{D_0}y_1^4x_1^3 + {q^{ - 1}}\left( {{D_0} + {D_1}} \right)y_1^3x_1^3y_1^{} + {q^{ - 2}}\left( {\sum\limits_{r = 0}^2 {{D_r}} } \right)y_1^2x_1^3y_1^2} \right.\\
	&&	\left. { + {q^{ - 3}}\left( {\sum\limits_{r = 0}^3 {{D_r}} } \right)y_1^{}x_1^3y_1^3 + {q^{ - 4}}\left( {\sum\limits_{r = 0}^4 {{D_r}} } \right)x_1^3y_1^4} \right){x_4},
	\end{eqnarray*}
	where ${D_r} = {( - 1)^r}{q^{\frac{{r(r - 1)}}{2}}}{\left[ {\begin{array}{*{20}{c}}
				5\\
				r
		\end{array}} \right]_q}.$
	
Note that 	$\sum\limits_{r = 0}^5 {{D_r}}  = 0,$  and for any $t\ge 1$, $$y_1^tx_1^t = \sum\limits_{k = 0}^t {{{\left[ {\begin{array}{*{20}{c}}
						t\\
						k
				\end{array}} \right]}_q}{q^{\frac{{{k^2}}}{2}}}x_2^{2(t - k)}x_3^{2k}x_4^k},\ x_1^ty_1^t = \sum\limits_{k = 0}^t {{{\left[ {\begin{array}{*{20}{c}}
						t\\
						k
				\end{array}} \right]}_q}{q^{\frac{{k(k - 2t)}}{2}}}x_2^{2(t - k)}x_3^{2k}x_4^k}.$$

Hence, the claim can be deduced from the following  computation
	\begin{eqnarray*}
	&&	{D_0}y_1^4{x_1} + {q^{ - 1}}\left( {{D_0} + {D_1}} \right)y_1^3{x_1}y_1^{} + {q^{ - 2}}\left( {\sum\limits_{r = 0}^2 {{D_r}} } \right)y_1^2{x_1}y_1^2\\
	&&	+ {q^{ - 3}}\left( {\sum\limits_{r = 0}^3 {{D_r}} } \right)y_1^{}{x_1}y_1^3 + {q^{ - 4}}\left( {\sum\limits_{r = 0}^4 {{D_r}} } \right){x_1}y_1^4\\
	&	=& {D_0}y_1^3\left( {{q^{\frac{1}{2}}}x_3^2{x_4} + x_2^2} \right) + {q^{ - 1}}\left( {{D_0} + {D_1}} \right)y_1^3\left( {{q^{ - \frac{1}{2}}}x_3^2{x_4} + x_2^2} \right)\\
	&&	+ {q^{ - 2}}\left( {\sum\limits_{r = 0}^2 {{D_r}} } \right)y_1^2\left( {{q^{ - \frac{1}{2}}}x_3^2{x_4} + x_2^2} \right){y_1} + {q^{ - 3}}\left( {\sum\limits_{r = 0}^3 {{D_r}} } \right)\left( {{q^{\frac{1}{2}}}x_3^2{x_4} + x_2^2} \right)y_1^3\\
	&&	+ {q^{ - 4}}\left( {\sum\limits_{r = 0}^4 {{D_r}} } \right)\left( {{q^{ - \frac{1}{2}}}x_3^2{x_4} + x_2^2} \right)y_1^3\\
	&	=& y_1^3x_3^2{x_4}\left( {{q^{\frac{1}{2}}}{D_0} + {q^{ - \frac{3}{2}}}\left( {{D_0} + {D_1}} \right) + {q^{ - \frac{7}{2}}}\left( {\sum\limits_{r = 0}^2 {{D_r}} } \right) + {q^{ - \frac{{11}}{2}}}\left( {\sum\limits_{r = 0}^3 {{D_r}} } \right) + {q^{ - \frac{{15}}{2}}}\left( {\sum\limits_{r = 0}^4 {{D_r}} } \right)} \right)\\
	&&	+ y_1^3x_2^2\left( {{D_0} + {q^{ - 1}}\left( {{D_0} + {D_1}} \right) + {q^{ - 2}}\left( {\sum\limits_{r = 0}^2 {{D_r}} } \right) + {q^{ - 3}}\left( {\sum\limits_{r = 0}^3 {{D_r}} } \right) + {q^{ - 4}}\left( {\sum\limits_{r = 0}^4 {{D_r}} } \right)} \right)\\
	&=& y_1^3x_3^2{x_4}\left( {{q^{\frac{1}{2}}}\sum\limits_{t = 0}^4 {{q^{ - 2t}}\left( {\sum\limits_{r = 0}^t {{D_r}} } \right)} } \right) + y_1^3x_2^2\left( {\sum\limits_{t = 0}^4 {{q^{ - t}}\left( {\sum\limits_{r = 0}^t {{D_r}} } \right)} } \right)\mathop  = \limits^{(\ref{(lemma3)})} 0,
	\end{eqnarray*}
and
	\begin{eqnarray*}
	&&	{D_0}y_1^4x_1^3 + {q^{ - 1}}\left( {{D_0} + {D_1}} \right)y_1^3x_1^3y_1^{} + {q^{ - 2}}\left( {\sum\limits_{r = 0}^2 {{D_r}} } \right)y_1^2x_1^3y_1^2\\
	&&	+ {q^{ - 3}}\left( {\sum\limits_{r = 0}^3 {{D_r}} } \right)y_1^{}x_1^3y_1^3 + {q^{ - 4}}\left( {\sum\limits_{r = 0}^4 {{D_r}} } \right)x_1^3y_1^4\\
	&	=& {D_0}{y_1}\sum\limits_{k = 0}^3 {{{\left[ {\begin{array}{*{20}{c}}
							3\\
							k
					\end{array}} \right]}_q}{q^{\frac{{{k^2}}}{2}}}x_2^{2(3 - k)}x_3^{2k}x_4^k}  + {q^{ - 1}}\left( {{D_0} + {D_1}} \right)\sum\limits_{k = 0}^3 {{{\left[ {\begin{array}{*{20}{c}}
							3\\
							k
					\end{array}} \right]}_q}{q^{\frac{{{k^2}}}{2}}}x_2^{2(3 - k)}x_3^{2k}x_4^k} y_1^{}\\
	&&	+ {q^{ - 2}}\left( {\sum\limits_{r = 0}^2 {{D_r}} } \right)\left( {\sum\limits_{k = 0}^2 {{{\left[ {\begin{array}{*{20}{c}}
								2\\
								k
						\end{array}} \right]}_q}{q^{\frac{{k^2}}{2}}}x_2^{2(2 - k)}x_3^{2k}x_4^k} } \right)\left( {{q^{ - \frac{1}{2}}}x_3^2{x_4} + x_2^2} \right)y_1^{}\\
	&&	+ {q^{ - 3}}\left( {\sum\limits_{r = 0}^3 {{D_r}} } \right)y_1^{}\sum\limits_{k = 0}^3 {{{\left[ {\begin{array}{*{20}{c}}
							3\\
							k
					\end{array}} \right]}_q}{q^{\frac{{k(k - 6)}}{2}}}x_2^{2(3 - k)}x_3^{2k}x_4^k} \\
	&&	+ {q^{ - 4}}\left( {\sum\limits_{r = 0}^4 {{D_r}} } \right)\sum\limits_{k = 0}^3 {{{\left[ {\begin{array}{*{20}{c}}
							3\\
							k
					\end{array}} \right]}_q}{q^{\frac{{k(k - 6)}}{2}}}x_2^{2(3 - k)}x_3^{2k}x_4^k} {y_1}\\
	&	= &{y_1}x_2^6\left( {{D_0} + {q^{ - 1}}\left( {{D_0} + {D_1}} \right)+{q^{ - 2}}\left( {\sum\limits_{r = 0}^2 {{D_r}} } \right) + {q^{ - 3}}\left( {\sum\limits_{r = 0}^3 {{D_r}} } \right) + {q^{ - 4}}\left( {\sum\limits_{r = 0}^4 {{D_r}} } \right)} \right)\\
		\end{eqnarray*}		
	\begin{eqnarray*}
	&&	+ {y_1}x_2^4x_3^2x_4^{}{\left[ {\begin{array}{*{20}{c}}
					3\\
					1
			\end{array}} \right]_q}\left( {{q^{\frac{1}{2}}}{D_0} + {q^{ - \frac{3}{2}}}\left( {{D_0} + {D_1}} \right) + {q^{ - \frac{7}{2}}}\left( {\sum\limits_{r = 0}^2 {{D_r}} } \right) + {q^{ - \frac{{11}}{2}}}\left( {\sum\limits_{r = 0}^3 {{D_r}} } \right)} \right.\\
	&&	\left. { + {q^{ - \frac{{15}}{2}}}\left( {\sum\limits_{r = 0}^4 {{D_r}} } \right)} \right) + {y_1}x_2^2x_3^4x_4^2{\left[ {\begin{array}{*{20}{c}}
					3\\
					2
			\end{array}} \right]_q}\left( {{q^2}{D_0} + {q^{ - 1}}\left( {{D_0} + {D_1}} \right) + {q^{ - 4}}\left( {\sum\limits_{r = 0}^2 {{D_r}} } \right)} \right.\\
	&&	+ {q^{ - 7}}\left( {\sum\limits_{r = 0}^3 {{D_r}} } \right)\left. { + {q^{ - 10}}\left( {\sum\limits_{r = 0}^4 {{D_r}} } \right)} \right) + {y_1}x_3^6x_4^3\left( {{q^{\frac{9}{2}}}{D_0} + {q^{\frac{1}{2}}}\left( {{D_0} + {D_1}} \right)} \right.\\
	&&	+ {q^{ - \frac{7}{2}}}\left( {\sum\limits_{r = 0}^2 {{D_r}} } \right) + {q^{ - \frac{{15}}{2}}}\left( {\sum\limits_{r = 0}^3 {{D_r}} } \right)\left. { + {q^{ - \frac{{23}}{2}}}\left( {\sum\limits_{r = 0}^4 {{D_r}} } \right)} \right)\\
	&	=& {y_1}x_2^6\left( {\sum\limits_{t = 0}^4 {{q^{ - t}}\left( {\sum\limits_{r = 0}^t {{D_r}} } \right)} } \right) + {y_1}x_2^4x_3^2x_4^{}{\left[ {\begin{array}{*{20}{c}}
					3\\
					1
			\end{array}} \right]_q}\left( {{q^{\frac{1}{2}}}\sum\limits_{t = 0}^4 {{q^{ - 2t}}\left( {\sum\limits_{r = 0}^t {{D_r}} } \right)} } \right)\\
			\end{eqnarray*}		
		\begin{eqnarray*}
	&&+	{y_1}x_2^2x_3^4x_4^2{\left[ {\begin{array}{*{20}{c}}
					3\\
					2
			\end{array}} \right]_q}\left( {{q^2}\sum\limits_{t = 0}^4 {{q^{ - 3t}}\left( {\sum\limits_{r = 0}^t {{D_r}} } \right)} } \right) + {y_1}x_3^6x_4^3\left( {{q^{\frac{9}{2}}}\sum\limits_{t = 0}^4 {{q^{ - 4t}}\left( {\sum\limits_{r = 0}^t {{D_r}} } \right)} } \right)\mathop  = \limits^{(\ref{(lemma3)})} 0.
	\end{eqnarray*}
\end{example}

\section*{Acknowledgments}
This work was supported by NSF of China (No. 12371036) and  Guangdong Basic and Applied Basic Research
Foundation (2023A1515011739).

\end{document}